\documentclass[a4paper, 10pt, notitlepage]{article}

\usepackage{amsthm, amsmath, amssymb, latexsym, eufrak}
\usepackage{mathrsfs,cite}
\usepackage[multiple]{footmisc}
\usepackage{graphicx}
\usepackage[ruled,vlined]{algorithm2e}
\theoremstyle{plain}
\newtheorem{theorem}{Theorem}
\newtheorem{lemma}[theorem]{Lemma}
\newtheorem{corollary}[theorem]{Corollary}
\newtheorem{proposition}[theorem]{Proposition}

\theoremstyle{definition}
\newtheorem{definition}{Definition}
\newtheorem{example}{Example}

\theoremstyle{remark}
\newtheorem{remark}{Remark}

\DeclareMathOperator{\trace}{Tr}
\DeclareMathOperator{\dist}{dist}
\DeclareMathOperator{\diag}{diag}
\DeclareMathOperator{\dom}{dom}
\DeclareMathOperator{\epigraph}{epi}
\DeclareMathOperator{\interior}{int}
\DeclareMathOperator{\sign}{sign}

\author{M.V. Dolgopolik\footnote{Institute for Problems in Mechanical Engineering of the Russian Academy of Sciences, 
Saint Petersburg, Russia}}
\title{Convergence analysis of primal-dual augmented Lagrangian methods and duality theory}

\begin{document}

\maketitle

\begin{abstract}
We develop a unified theory of augmented Lagrangians for nonconvex optimization problems that encompasses both duality
theory and convergence analysis of primal-dual augmented Lagrangian methods in the infinite dimensional setting. Our
goal is to present many well-known concepts and results related to augmented Lagrangians in a unified manner and bridge
a gap between existing convergence analysis of primal-dual augmented Lagrangian methods and abstract duality theory.
Within our theory we specifically emphasize the role of various fundamental duality concepts (such as duality gap,
optimal dual solutions, global saddle points, etc.) in convergence analysis of augmented Lagrangians methods and 
underline interconnections between all these concepts and convergence of primal and dual sequences generated by such
methods. In particular, we prove that the zero duality gap property is a necessary condition for the boundedness of the
primal sequence, while the existence of an optimal dual solution is a necessary condition for the boundedness of the
sequences of multipliers and penalty parameters, irrespective of the way in which the multipliers and the penalty
parameter are updated. Our theoretical results are applicable to many different augmented Lagrangians for various types
of cone constrained optimization problems, including Rockafellar-Wets' augmented Lagrangian, (penalized)
exponential/hyperbolic-type augmented Lagrangians, modified barrier functions, etc.
\end{abstract}

\section{Introduction}

Augmented Lagrangians play a fundamental role in optimization and many other closely related fields
\cite{Bertsekas,FortinGlowinski,ItoKunisch,BirginMartinez,BurmanHansboLarson}, and there is a vast literature on various
aspects of augmented Lagrangian theory and corresponding methods. A wide range of topics that is studied within the
more theoretically oriented part of the literature includes such important problems as analysis of the zero duality gap
property \cite{HuangYang2003,HuangYang2005,BurachikRubinov,LiuYang2008,ZhouYang2009,WangYangYang,YalchinKasimbeyli}, 
exact penalty representation
\cite{HuangYang2003,HuangYang2005,ZhouYang2009,BurachikIusemMelo,WangYangYang,ZhouYang2012,ZhouXiuWang,
BurachikIusemMelo2015, FeizollahiAhmedSun},
existence of global saddle points
\cite{SunLiMcKinnon,WuLuo2012b,WuLuoYang2014,ZhouChen2015,LuoWuLiu2015,WangLiuQu,Dolgopolik2018},
existence of augmented Lagrange multipliers
\cite{ShapiroSun,RuckmannShapiro,ZhouZhouYang2014,BurachikYangZhou,Dolgopolik2017}, etc. 

In turn, more practically oriented papers are usually devoted to convergence analysis of primal-dual augmented
Lagrangian methods and do not utilise any theoretical concepts or results from duality theory, except for the zero
duality gap property
\cite{BurachikGasimov,BurachickIusemMelo_SharpLagr,BagirovOzturkKasimbeyli,BurachikLiu2023,BurachikKayaLiu2023}. Very
few attempts have been made to connect convergence analysis of such methods with fundamental results from duality
theory. Paper \cite{Burachick2011}, in which the equivalence between primal convergence and differentiability of the
augmented dual function for the sharp Lagrangian for equality constrained problems was established, is perhaps the most
notable among them. 

In addition, papers dealing with convergence analysis of primal-dual augmented Lagrangian methods typically consider
only one particular augmented Lagrangian for one particular type of constraints. As a consequence of that, very similar
results have to be repeated multiple types in different settings (cf. convergence analysis of augmented Lagrangian
methods based on the Hestenes-Powell-Rockafellar augmented Lagrangian for mathematical programming problems
\cite{BirginMartinez}, nonconvex problems with second order cone constraints \cite{LiuZhang2007,LiuZhang2008}, nonconvex
semidefinite programming problems \cite{SunSunZhang2008,LuoWuChen2012,WuLuoDingChen2013}, as well as similar convergence
analysis of augmented Lagrangian methods based on the exponential-type augmented Lagrangian/modified barrier function
for nonconvex problems with second order cone constraints \cite{ZhangGuXiao2011} and nonconvex semidefinite programming
problems \cite{LiZhang2009}). Some attempts have been made to unify convergence analysis of a number of primal-dual
augmented Lagrangian methods \cite{LuoSunLi,LuoSunWu,WangLi2009}, but only for finite dimensional inequality
constrained optimization problems.

Thus, there are two fundamental challenges within the theory of augmented Lagrangians and corresponding optimization
methods. The first one is connected with a noticeable gap that exists between more theoretically oriented results
dealing with duality theory and more practically oriented results on convergence analysis of primal-dual methods.
Often, the duality theory plays little to no role in convergence analysis of primal-dual methods, and theoretical
results from this theory are almost never utilised to help better understand convergence of augmented Lagrangian
methods. The second challenge is connected with unification of numerous similar results on duality theory and augmented
Lagrangian methods that are scattered in the literature. 

The main goal of this article is to present a general theory of augmented Lagrangians for nonconvex cone constrained
optimization problems in the infinite dimensional setting that encompasses both fundamental theoretical concepts from
duality theory and convergence analysis of primal-dual augmented Lagrangian methods, and also highlights
interconnections between them, thus at least partially solving the two aforementioned challenges. 

Our other aim is to correct for the bias that exists in the literature on augmented Lagrangians, in which a
disproportionate number of papers is devoted to analysis of Rockafellar-Wets's augmented Lagrangian
\cite[Section~11.K]{RockafellarWets} and corresponding numerical methods, while very little attention is paid to
other classes of augmented Lagrangians. In particular, to the best of the author's knowledge, many concepts (such as
exact penalty map \cite{BurachikIusemMelo2015}, exact penalty representation, and augmented Lagrange multipliers) have
been introduced and studied exclusively in the context of Rockafellar-Wets' augmented Lagrangian. Our aim is to show
that these concepts can be defined and be useful in a much more broad setting and to demonstrate that many results from
duality theory and convergence analysis of primal-dual methods can be proved for Rockafellar-Wets' augmented Lagrangian
and many other augmented Lagrangians in a unified way. 

To achieve our goals, we adopt an axiomatic augmented Lagrangian setting developed by the author in
\cite{Dolgopolik2018} and inspired by \cite{LiuYang2008}. Within this setting, one defines an abstract augmented
Lagrangian without specifying its structure, while all theoretical results are proved using a set of axioms
(\textit{basic assumptions}). To demonstrate the broad applicability of this approach, we present many particular
examples of well-known augmented Lagrangians and prove that they satisfy these axioms. 

We also develop a general duality theory for augmented Lagrangians that complements the results of the author's earlier
papers \cite{Dolgopolik2018,Dolgopolik2021} and, in particular, provide simple necessary and sufficient conditions for
the zero duality gap property to hold true, from which many existing results on this property can be easily derived.
Finally, we present a general convergence analysis for a model augmented Lagrangian method with arbitrary rules for
updating multipliers and penalty parameter. Under some natural assumptions, that are satisfied in many particular cases,
we study primal and dual convergence of this method, making specific emphasis on the role of various fundamental
concepts from duality theory in convergence analysis of augmented Lagrangian methods. In particular, we points out
direct connections between primal convergence and the zero duality gap property, as well as direct connections between
dual convergence, boundedness of the penalty parameter, and the existence of optimal dual solutions/global saddle
points.

The paper is organized as follows. An abstract axiomatic augmented Lagrangian setting for cone constrained optimization
problems in normed spaces, including a list of \textit{basic assumptions} (axioms) on an augmented Lagrangian, is
presented in Section~\ref{sect:AxiomaticSetting}. Many particular examples of augmented Lagrangians for various types of
cone constrained optimization problems and the basic assumptions that these augmented Lagrangians satisfy are discussed
in details in Section~\ref{sect:Examples}. Section~\ref{sect:DualityTheory} is devoted to duality theory. In this
section we analyse the zero duality gap property for the augmented Lagrangian, and also study interconnections between
global saddle points, globally optimal solutions of the augmented dual problem, augmented Lagrange multipliers, and the
penalty map. Finally, a general convergence theory for a model augmented Lagrangian method, that encompasses many
existing primal-dual augmented Lagrangian methods as particular cases, is presented in
Section~\ref{sect:ConvergenceAnalysis}.

\section{Axiomatic augmented Lagrangian setting}
\label{sect:AxiomaticSetting}

We start by presenting an axiomatic approach to augmented Lagrangians that serves as a foundation for our duality and
convergence theories. Let $X$ and $Y$ be real normed spaces, $Q \subseteq X$ be a nonempty closed set (\textit{not}
necessarily convex), and $K \subset Y$ be a closed convex cone. Suppose also that some functions 
$f \colon X \to \mathbb{R} \cup \{ + \infty \}$ and $G \colon X \to Y$ are given. In this article we study augmented
Lagrangians for the following cone constrained optimization problem:
\[
  \min\: f(x) \quad \text{subject to} \quad G(x) \in K, \quad x \in Q.
  \qquad \eqno{(\mathcal{P})}
\]
We assume that the feasible region of this problem is nonempty and its optimal value, denoted by $f_*$, is finite.

The topological dual space of $Y$ is denoted by $Y^*$, and let $\langle \cdot, \cdot \rangle$ be either the inner
product in $\mathbb{R}^s$, $s \in \mathbb{N}$, or the duality pairing between $Y$ and its dual, depending on 
the context. Recall that $K^* = \{ y^* \in Y^* \mid \langle y^*, y \rangle \le 0 \enspace \forall y \in K \}$ is
the polar cone of the cone $K$. 

Denote by $\preceq$ the binary relation over $Y \times Y$ defined as $y_1 \preceq y_2$ if and only if 
$y_2 - y_1 \in -K$. We say that this binary relation is induced by the cone $-K$. As one can readily check, this 
binary relation is a partial order on $Y$ if and only if $K \cap (-K) = \{ 0 \}$. In this case $\preceq$ is called 
the partial order induced by the cone $-K$. Note that the cone constraint $G(x) \in K$ can be rewritten as 
$G(x) \preceq 0$. 

We define an \textit{augmented Lagrangian} for the problem $(\mathcal{P})$ as follows. Choose any function
$\Phi \colon Y \times Y^* \times (0, + \infty) \to \mathbb{R} \cup \{ \pm \infty \}$, $\Phi = \Phi(y, \lambda, c)$. 
An augmented Lagrangian for the problem $(\mathcal{P})$ is defined as
\[
  \mathscr{L}(x, \lambda, c) = f(x) + \Phi(G(x), \lambda, c)
\] 
if $\Phi(G(x), \lambda, c) > - \infty$, and $\mathscr{L}(x, \lambda, c) = - \infty$, otherwise. 
Here $\lambda \in Y^*$ is a multiplier and $c > 0$ is a penalty parameter. Note that only the constraint 
$G(x) \in K$ is incorporated into the augmented Lagrangian.

Unlike most existing works on augmented Lagrangians and corresponding optimization methods, we do not impose any
assumptions on the structure of the function $\Phi$. It can be defined in an arbitrary way. Instead, being inspired by
\cite{LiuYang2008} and following the ideas of our earlier paper \cite{Dolgopolik2018}, we propose to use a set of axioms
(assumptions) describing behaviour of the function $\Phi(y, \lambda, c)$ for various types of arguments (e.g. as $c$
increases unboundedly or when $y \in K$). This approach allows one to construct an axiomatic theory of augmented
Lagrangians and helps one to better understand what kind of assumptions the function $\Phi$ must satisfy to guarantee
that the augmented Lagrangian $\mathscr{L}(x, \lambda, c)$ and optimization methods based on this augmented Lagrangian
have certain desirable properties. As we will show below, most well-known augmented Lagrangians satisfy our proposed set
of axioms, which means that our axiomatic theory is rich enough and can be applied in many particular cases.

In order to unite several particular cases into one general theory, we formulate axioms/assumptions on the function
$\Phi$ with respect to a prespecified closed convex cone $\Lambda \subseteq Y^*$ of admissible multipliers. Usually,
$\Lambda = Y^*$ or $\Lambda = K^*$. 

We grouped the assumptions on $\Phi$ according to their meaning. If the function $\Phi(y, \lambda, c)$ is viewed as 
a black box with input $(y, \lambda, c)$ and output $\Phi(y, \lambda, c)$, then assumptions $(A1)$--$(A6)$ describe 
what kind of output is produced by this black box with respect to certain specific kinds of input. Assumptions 
$(A7)$--$(A11)$ describe general properties of the function $\Phi$, such as monotonicity, differentiability, and 
convexity. Assumptions $(A12)$--$(A15)$ impose restrictions on the way the function $\Phi(y, \lambda, c)$ behaves 
as $c$ increases unboundedly or along certain sequences $\{ (y_n, \lambda_n, c_n) \}$. Finally, the subscript ``s''
indicates a stronger, i.e. more restrictive, version of an assumption.

Denote $\dist(y, K) = \inf_{z \in K} \| y - z \|$ for any $y \in Y$, and let $B(x, r)$ be the closed ball with centre 
$x$ and radius $r > 0$. Below is the list of \textit{basic assumptions} on the function $\Phi$ that are utilised
throughout the article:
\begin{itemize}
\item[(A1)]{$\forall y \in K \: \forall \lambda \in \Lambda \: \forall c > 0$ one has $\Phi(y, \lambda, c) \le 0$;}

\item[(A2)]{$\forall y \in K \: \forall c > 0 \: \exists \lambda \in \Lambda$ such that $\Phi(y, \lambda, c) \ge 0$;}

\item[(A3)]{$\forall y \notin K \: \forall c > 0 \: \exists \lambda \in \Lambda$ such that 
$\Phi(y, t \lambda, c) \to + \infty$ as $t \to + \infty$;}

\item[(A4)]{$\forall y \in K \: \forall \lambda \in K^* \: \forall c > 0$ such that 
$\langle \lambda, y \rangle = 0$ one has $\Phi(y, \lambda, c) = 0$;
}

\item[(A5)]{$\forall y \in K \: \forall \lambda \in K^* \: \forall c > 0$ such that 
$\langle \lambda, y \rangle \ne 0$ one has $\Phi(y, \lambda, c) < 0$;
}

\item[(A6)]{$\forall y \in K \: \forall \lambda \in \Lambda \setminus K^* \: \forall c > 0$ 
one has $\Phi(y, \lambda, c) < 0$;
}

\item[(A7)]{$\forall y \in Y \: \forall \lambda \in \Lambda$ the function $c \mapsto \Phi(y, \lambda, c)$ is
non-decreasing;}

\item[(A8)]{$\forall \lambda \in \Lambda \: \forall c > 0$ the function $y \mapsto \Phi(y, \lambda, c)$ is convex and
non-decreasing with respect to the binary relation $\preceq$;}

\item[(A9)]{$\forall y \in Y \: \forall c > 0$ the function $\lambda \mapsto \Phi(y, \lambda, c)$ is concave;}

\item[$(A9)_s$]{$\forall y \in Y$ the function $(\lambda, c) \mapsto \Phi(y, \lambda, c)$ is concave;}

\item[(A10)]{$\forall y \in Y$ the function $(\lambda, c) \mapsto \Phi(y, \lambda, c)$ is upper semicontinuous;}

\item[(A11)]{$\forall y \in K \: \forall \lambda \in K^* \: \forall c > 0$ such that $\langle \lambda, y \rangle = 0$
the function $\Phi(\cdot, \lambda, c)$ is Fr\'{e}chet differentiable at $y$ and its Fr\'{e}chet derivative satisfies
the equality $D_y \Phi(y, \lambda, c) = \Phi_0(\lambda)$ for some surjective mapping $\Phi_0 \colon K^* \to K^*$ that 
does not depend on $y$ and $c$, and such that $\langle \Phi_0(\lambda), y \rangle = 0$ if and only if 
$\langle \lambda, y \rangle = 0$;
}

\item[(A12)]{$\forall \lambda \in \Lambda \: \forall c_0 > 0 \: \forall r > 0$ one has
\begin{multline*}
  \lim_{c \to + \infty} \inf\Big\{ \Phi(y, \lambda, c) - \Phi(y, \lambda, c_0) \Bigm|
  \\
  y \in Y \colon \dist(y, K) \ge r, \: |\Phi(y, \lambda, c_0)| < + \infty \Big\} = + \infty;
\end{multline*}
}

\vspace{-5mm}

\item[$(A12)_s$]{$\forall c_0 > 0 \: \forall r > 0$ and for any bounded subset $\Lambda_0 \subseteq \Lambda$ one has
\begin{multline*}
  \lim_{c \to + \infty} \inf_{\lambda \in \Lambda_0} \inf\Big\{ \Phi(y, \lambda, c) - \Phi(y, \lambda, c_0) \Bigm|
  \\
  y \in Y \colon \dist(y, K) \ge r, \: |\Phi(y, \lambda, c_0)| < + \infty \Big\} = + \infty;
\end{multline*}
}

\vspace{-5mm}

\item[(A13)]{$\forall \lambda \in \Lambda$ and for any sequences $\{ c_n \} \subset (0, + \infty)$ and 
$\{ y_n \} \subset Y$ such that $c_n \to + \infty$ and $\dist(y_n, K) \to 0$ as $n \to \infty$ one has 
$\liminf\limits_{n \to \infty} \Phi(y_n, \lambda, c_n) \ge 0$;}

\item[$(A13)_s$]{for any sequences $\{ c_n \} \subset (0, + \infty)$ and $\{ y_n \} \subset Y$ such that
$c_n \to + \infty$ and $\dist(y_n, K) \to 0$ as $n \to \infty$ and for any bounded subset $\Lambda_0 \subseteq \Lambda$
one has $\liminf\limits_{n \to \infty} \inf\limits_{\lambda \in \Lambda_0} \Phi(y_n, \lambda, c_n) \ge 0$;
}

\item[(A14)]{$\forall \lambda \in \Lambda \: \forall \{ c_n \} \subset (0, + \infty)$ such that $c_n \to + \infty$ as 
$n \to \infty$ there exists $\{ t_n \} \subset (0, + \infty)$ such that $t_n \to 0$ as $n \to \infty$ and for any 
$\{ y_n \} \subset Y$ with $\dist(y_n, K) \le t_n$ one has $\lim\limits_{n \to \infty} \Phi(y_n, \lambda, c_n) = 0$;
}

\item[$(A14)_s$]{$\forall \{ c_n \} \subset (0, + \infty)$ such that $c_n \to + \infty$ as $n \to \infty$ and for any
bounded sequence $\{ \lambda_n \} \subseteq \Lambda$ there exists $\{ t_n \} \subset (0, + \infty)$ such that 
$t_n \to 0$ as $n \to \infty$ and for any $\{ y_n \} \subset Y$ with $\dist(y_n, K) \le t_n$ the following equality
holds true $\lim\limits_{n \to \infty} \Phi(y_n, \lambda_n, c_n) = 0$; 
}

\item[(A15)]{for any bounded sequences $\{ \lambda_n \} \subset \Lambda$ and $\{ c_n \} \subset (0, + \infty)$ and for
any sequence $\{ y_n \} \subset Y$ such that $\dist(y_n, K) \to 0$ as $n \to \infty$ one has 
$\limsup\limits_{n \to \infty} \Phi(y_n, \lambda_n, c_n) \le 0$;
}
\end{itemize}

\begin{remark} \label{rmrk:RestrictedAssumptions}
In the case of assumptions $(A13)$, $(A13)_s$, $(A14)$, $(A14)_s$, and $(A15)$, we say that the function $\Phi$ 
satisfies the \textit{restricted} version of the corresponding assumption, if one replaces ``any sequence 
$\{ y_n \} \subset Y$'' in the formulation of corresponding assumption with ``any \textit{bounded} sequence 
$\{ y_n \} \subset Y$''. Note that the validity of a (non-restricted) assumption implies that its restricted version
also hold true.
\end{remark}

\begin{remark}
It should be noted that assumption $(A13)_s$ can be reformulated as follows: for any sequences 
$\{ c_n \} \subset (0, + \infty)$ and $\{ y_n \} \subset Y$ such that $c_n \to + \infty$ and $\dist(y_n, K) \to 0$ as 
$n \to \infty$ and for any bounded sequence $\{ \lambda_n \} \subseteq \Lambda$ one has 
$\liminf\limits_{n \to \infty} \Phi(y_n, \lambda_n, c_n) \ge 0$.
\end{remark}

Below we will often use the following corollary to assumptions $(A13)$ and $(A13)_s$ that can be readily verified
directly.

\begin{lemma} \label{lem:Assumpt(A16)}
If $\Phi$ satisfies assumption $(A13)$, then for any $\lambda \in \Lambda$ one has
\[
  \liminf_{c \to + \infty} \inf_{y \in K} \Phi(y, \lambda, c) \ge 0.
\]
If $\Phi$ satisfies assumption $(A13)_s$, then for any bounded subset $\Lambda_0 \subseteq \Lambda$ one has 
\[
  \liminf_{c \to + \infty} \inf_{\lambda \in \Lambda_0} \inf\limits_{y \in K} \Phi(y, \lambda, c) \ge 0.
\]
\end{lemma}

In the following section we will give many particular examples of augmented Lagrangians for different types of cone
constrained optimization problems (e.g. equality constrained problems, inequality constrained problems, problems with
semidefinite constraints, etc.). If an optimization problem has several different types of constraints simultaneously,
it is convenient to represent them as cone constraints of the form $G_i(x) \in K_i$ for some mappings 
$G_i \colon X \to Y_i$ and closed convex cones $K_i$ in real Banach spaces $Y_i$, $i \in \{ 1, \ldots, m \}$, with 
$m \in \mathbb{N}$. Then one can define $Y = Y_1 \times \ldots \times Y_m$ and
\[
  G(\cdot) = (G_1(\cdot), \ldots, G_m(\cdot)), \quad K = K_1 \times \ldots \times K_m
\]
to formally rewrite such optimization problems as the problem $(\mathcal{P})$. The space $Y$ is equipped with the norm
$\| y \| = \| y_1 \| + \ldots + \| y_m \|$ for all $y = (y_1, \ldots, y_m) \in Y$.

In order to define a function $\Phi(y, \lambda, c)$ in this case, one can define corresponding functions 
$\Phi_i(y_i, \lambda_i, c)$ for each constraint $G_i(x) \in K_i$ individually (here $y_i \in Y_i$ and 
$\lambda_i \in \Lambda_i \subseteq Y_i^*$) and then put
\begin{equation} \label{eq:SeparableAugmLagr}
  \Phi(y, \lambda, c) = \sum_{i = 1}^m \Phi_i(y_i, \lambda_i, c), \quad
  y = (y_1, \ldots, y_m), \quad \lambda = (\lambda_1, \ldots, \lambda_m),
\end{equation}
if $\Phi_i(y_i, \lambda_i, c) > - \infty$ for all $i$, and $\Phi(y, \lambda, c) = - \infty$, otherwise. Most (if not
all) existing augmented Lagrangians for problems with several different types constraints are defined in this way. Let
us show that, roughly speaking, if $\Lambda = \Lambda_1 \times \ldots \times \Lambda_m$ and all functions $\Phi_i$, 
$i \in \{ 1, \ldots, m \}$, satisfy some basic assumption simultaneously, then the function $\Phi$ also satisfies this
basic assumption. 

\begin{theorem} \label{thrm:BasicAssumptionSeparateCase}
Let $\Lambda = \Lambda_1 \times \ldots \times \Lambda_m$. Then the following statements hold true:
\begin{enumerate}
\item{if all functions $\Phi_i$, $i \in I := \{ 1, \ldots, m \}$, satisfy one of the basic assumptions simultaneously, 
except for assumptions $(A5)$, $(A6)$, $(A12)$, and $(A12)_s$, then the function $\Phi$ defined in 
\eqref{eq:SeparableAugmLagr} satisfies the same basic assumption;
}

\item{if all functions $\Phi_i$, $i \in I$, satisfy assumptions $(A1)$ and $(A5)$ (or $(A1)$ and $(A6)$)
simultaneously, then the function $\Phi$ defined in \eqref{eq:SeparableAugmLagr} also satisfies assumption $(A5)$ 
(or $(A6)$);
}

\item{if all functions $\Phi_i$, $i \in I$, satisfy assumptions $(A7)$ and $(A12)$ (or $(A7)$ and $(A12)_s$)
simultaneously, then the function $\Phi$ defined in \eqref{eq:SeparableAugmLagr} also satisfies assumption $(A12)$ 
(or $(A12)_s$).
}
\end{enumerate}
\end{theorem}

\begin{proof}
\textbf{Assumption (A1).} If $y \in K$ and $\lambda \in \Lambda$, then $y_i \in K_i$ and $\lambda_i \in \Lambda_i$
for all $i \in I$. Therefore, $\Phi_i(y_i, \lambda_i, c) \le 0$ by assumption $(A1)$, which implies that 
$\Phi(y_, \lambda, c) \le 0$.

\textbf{Assumption (A2).} Choose any $y \in K$ and $c > 0$. By assumption $(A2)$ for any $i \in I$ there exists
$\lambda_i \in \Lambda_i$ such that $\Phi_i(y_i, \lambda_i, c) \ge 0$. Put $\lambda = (\lambda_1, \ldots, \lambda_m)$.
Then $\Phi(y, \lambda, c) \ge 0$.

\textbf{Assumption (A3).} Choose any $y \in K$ and $c > 0$. By assumption $(A3)$ for any $i \in I$ there exists
$\lambda_i \in \Lambda_i$ such that $\Phi_i(y_i, t \lambda_i, c) \to + \infty$ as $t \to +\infty$. 
Put $\lambda = (\lambda_1, \ldots, \lambda_m)$. Then $\Phi(y, t \lambda, c) \to + \infty$ as $t \to \infty$.

\textbf{Assumption (A4).} Let $y \in K$ and $\lambda \in K^*$ be such that $\langle \lambda, y \rangle = 0$.
It is easily seen that the condition $\lambda \in K^*$ implies that $\lambda_i \in K^*_i$ for any $i \in I$ and, 
therefore, $\langle \lambda_i, y_i \rangle \le 0$ for all $i \in I$. Hence $\langle \lambda_i, y_i \rangle = 0$ for
all $i \in I$ and by assumption $(A4)$ one has $\Phi_i(y_i, \lambda_i, c) = 0$, which yields $\Phi(y, \lambda, c) = 0$.

\textbf{Assumption (A5).} If $y \in K$ and $\lambda \in K^*$ are such that $\langle \lambda, y \rangle \ne 0$, then
there exists $k \in I$ such that $\langle \lambda_k, y_k \rangle \ne 0$. Therefore, $\Phi_k(y_k, \lambda_k, c) < 0$. In
turn, for any $i \ne k$ one has $\Phi_i(y_i, \lambda_i, c) \le 0$ by assumption $(A1)$. Hence $\Phi(y, \lambda, c) < 0$.

\textbf{Assumption (A6).} If $y \in K$ and $\lambda \in \Lambda \setminus K^*$, then 
$\lambda_k \in \Lambda_k \setminus K_k^*$ for some $k \in I$. Therefore, $\Phi_k(y_k, \lambda_k, c) < 0$, while
$\Phi(y_i, \lambda_i, c) \le 0$ for any $i \ne k$ by assumption $(A1)$. Consequently, $\Phi(y, \lambda, c) < 0$.

\textbf{Assumption (A7).} The function $\Phi(y, \lambda, c)$ is non-decreasing in $c$ as the sum of non-decreasing in
$c$ functions.

\textbf{Assumption (A8).} Note that if the function $\Phi_i(\cdot, \lambda_i, c)$ is non-decreasing with respect to 
the binary relation induced by the cone $- K_i$, then the function $y \mapsto \Phi_i(y_i, \lambda_i, c)$ is
non-decreasing with respect to the binary relation induced by the cone $- K$. Therefore the function 
$\Phi(y, \lambda, c)$ is convex and non-decreasing with respect to the binary relation induced by the cone $- K$ as 
the sum of convex and non-decreasing functions.

\textbf{Assumptions (A9) and (A9)${}_s$.} The function $\Phi(y, \lambda, c)$ is concave in $\lambda$ 
(or $(\lambda, c)$) as the sum of concave functions.

\textbf{Assumption (A10).} The function $\Phi(y, \lambda, c)$ is upper semicontinuous as the sum of a finite number of
upper semicontinuous functions.

\textbf{Assumption (A11).} If $y \in K$ and $\lambda \in K^*$ are such that $\langle \lambda, y \rangle = 0$, then, as
was noted above, $y_i \in K_i$, $\lambda_i \in K_i^*$, and $\langle \lambda_i, y_i \rangle = 0$ for all $i \in I$. Then
each of the functions $\Phi_i(\cdot, \lambda_i, c)$ is Fr\'{e}chet differentiable at $y_i$ and its derivative has the
form $D_{y_i} \Phi_i(y_i, \lambda_i, c) = \Phi_{i0}(\lambda_i)$ for some function $\Phi_{i0} \colon K_i^* \to K_i^*$
such that $\langle \Phi_{i0}(\lambda_i), y_i \rangle = 0$ if and only if $\langle \lambda_i, y_i \rangle = 0$.
Therefore, the function $\Phi(\cdot, \lambda, c)$ is Fr\'{e}chet differentiable at the point $y$ and its Fr\'{e}chet
derivative is equal to $\Phi_0(\lambda) = (\Phi_{10}(\lambda_1), \ldots, \Phi_{m0}(\lambda_m))$. Clearly, one has
\begin{align*}
  \langle \Phi_0(\lambda), y \rangle = 0 \enspace \Leftrightarrow \enspace
  \langle \Phi_{i0}(\lambda_i), y_i \rangle = 0 \enspace \forall i \in I \enspace &\Leftrightarrow \enspace
  \langle \lambda_i, y_i \rangle = 0 \enspace \forall i \in I 
  \\
  &\Leftrightarrow \enspace \langle \lambda, y \rangle = 0, 
\end{align*}
which implies the required result.

\textbf{Assumption (A12).} Fix any $\lambda \in \Lambda$, $c_0 > 0$, and $r > 0$. Choose some $\alpha > 0$. Let 
$y \in Y$ be such that $|\Phi(y, \lambda, c_0)| < + \infty$ and $\dist(y, K) \ge r$. Then there exists $i \in I$ such 
that $\dist(y_i, K_i) \ge r/m$. By assumption $(A12)$ for the functions $\Phi_j$, for any $j \in I$ there exists 
$c_j(r/m, \alpha) \ge c_0$ such that for all $c \ge c_j(r/m, \alpha)$ one has
\begin{multline*}
  \inf\Big\{ \Phi_j(z_j, \lambda_j, c) - \Phi_j(z_j, \lambda_j, c_0) \Bigm|
  z_j \in Y_j \colon \dist(z_j, K_j) \ge r/m, 
  \\
  |\Phi_j(z_j, \lambda_j, c_0)| < + \infty \Big\} \ge \alpha.
\end{multline*}
Let $c(r, \alpha) := \max\{ c_j(r/m, \alpha) \mid j \in I \}$. Then with the use of assumption $(A7)$ one gets that 
for any $c \ge c(r, \alpha)$ the following inequalities hold true:
\begin{align*}
  \Phi(y, \lambda, c) - \Phi(y, \lambda, c_0) 
  &= \sum_{i = 1}^m \big( \Phi_i(y_i, \lambda_i, c) - \Phi_i(y_i, \lambda_i, c_0) \big)
  \\
  &\ge \Phi_i(y_i, \lambda_i, c) - \Phi_i(y_i, \lambda_i, c_0) \ge \alpha.
\end{align*}
Since $y \in Y$ such that $|\Phi(y, \lambda, c_0)| < + \infty$ and $\dist(y, K) \ge r$ were chosen arbitrarily, one can
conclude that for any $\alpha > 0$ and $r > 0$ there exists a number $c(r, \alpha) \ge c_0$ such that
\[
  \inf\Big\{ \Phi(y, \lambda, c) - \Phi(y, \lambda, c_0) \Bigm|
  y \in Y \colon \dist(y, K) \ge r, \:
  \Phi(y, \lambda, c_0) < + \infty \Big\} \ge \alpha
\]
for all $c \ge c(r, \alpha)$. Consequently, the function $\Phi$ satisfies assumption $(A12)$.

\textbf{Assumption (A12)${}_s$.} Choose some $c_0 > 0$ and let $\Lambda_0 \subset \Lambda$ be a bounded set. Clearly, 
one can find bounded sets $\Lambda_{i0} \subset \Lambda_i$ such that $\Lambda_0 \subseteq \widehat{\Lambda}_0$, where 
$\widehat{\Lambda}_0 := \Lambda_{10} \times \ldots \times \Lambda_{m0}$.

By assumption $(A12)_s$ for the functions $\Phi_i$, for any $i \in I$, $r > 0$, and $\alpha > 0$ one can find 
$c_i(r, \alpha, \Lambda_{i0}) \ge c_0$ such that
\[
  \inf_{\lambda \in \Lambda_{i0}} \inf\Big\{ \Phi_i(y_i, \lambda_i, c) - \Phi(y_i, \lambda_i, c_0) \Bigm|
  y_i \in Y_i(r, \Phi_i) \Big\} \ge \alpha \quad \forall c \ge c_i(r, \alpha, \Lambda_{i0}),
\]
where $Y_i(r, \Phi_i) := \{ y_i \in Y_i \mid \dist(y_i, K_i) \ge r, |\Phi_i(y_i, \lambda_i, c_0)| < + \infty \}$.

Denote $Y(r, \Phi) = \{ y \in Y \mid \dist(y, K) \ge r, \: |\Phi(y, \lambda, c_0)| < + \infty \}$. Fix some $r > 0$ 
and choose any $y \in Y(r, \Phi)$. Then $\dist(y_i, K_i) \ge r/m$ for some $i \in I$. Hence with the use of assumption
$(A7)$ one gets that 
\begin{multline*}
  \inf_{\lambda \in \Lambda_0} \inf\Big\{ \Phi(y, \lambda, c) - \Phi(y, \lambda, c_0) \Bigm| y \in Y(r, \Phi) \Big\}
  \\
  \ge \inf_{\lambda \in \widehat{\Lambda}_0} \inf\Big\{ \Phi(y, \lambda, c) - \Phi(y, \lambda, c_0) \Bigm| 
  y \in Y(r, \Phi) \Big\}
  \\
  \ge \inf_{\lambda \in \Lambda_{i0}} \inf\Big\{ \Phi_i(y_i, \lambda_i, c) - \Phi(y_i, \lambda_i, c_0) \Bigm|
  y_i \in Y_i(r/m, \Phi_i) \Big\} \ge \alpha 
\end{multline*}
for any $c \ge c(r, \alpha, \Lambda_0) := \max\{ c_j(r/m, \alpha, \Lambda_{i0}) \mid j \in I \}$. Since $\alpha > 0$
was chosen arbitrarily, one can conclude that the function $\Phi$ satisfies assumption $(A12)_s$.

\textbf{Assumption (A13).} If sequences $\{ c_n \} \subset (0, + \infty)$ and $\{ y_n \} \subset Y$ are such that
$c_n \to + \infty$ and $\dist(y_n, K) \to 0$ as $k \to \infty$, then $\dist(y_{in}, K_i) \to 0$ as $n \to \infty$ and
$\liminf_{n \to \infty} \Phi_i(y_{in}, \lambda, c_n) \ge 0$ for all $i \in I$. From these inequalities it obviously
follows that $\liminf_{n \to \infty} \Phi(y_n, \lambda, c_n) \ge 0$ as well.

\textbf{Assumption (A13)${}_s$.} The proof is essentially the same as the proof for assumption $(A13)$.

\textbf{Assumption (A14).} Fix any $\lambda \in \Lambda$ and a sequence $\{ c_n \} \subset (0, + \infty)$ such that
$c_n \to + \infty$ as $n \to \infty$. By our assumption for any $i \in I$ there exists a sequence 
$\{ t_{in} \} \subset (0, + \infty)$ converging to zero and such that for any $\{ y_{in} \} \subset Y_i$ with
$\dist(y_{in}, K_i) \le t_{in}$ for all $n \in \mathbb{N}$ one has $\Phi_i(y_{in}, \lambda, c_n) \to 0$ as 
$n \to \infty$. 

Define $t_n = \min\{ t_{1n}, \ldots, t_{mn} \}$ and choose any sequence $\{ y_n \} \subset Y$ such that
$\dist(y_n, K) \le t_n$ for all $n \in \mathbb{N}$. Then for any $i \in I$ and $n \in \mathbb{N}$ one has 
$\dist(y_{in}, K_i) \le t_{in}$, which implies that $\Phi_i(y_{in}, \lambda, c_n) \to 0$ as $n \to \infty$ and,
therefore, $\Phi(y_n, \lambda, c_n) \to 0$ as $n \to \infty$.

\textbf{Assumption (A14)${}_s$.} The proof is the same as the proof for assumption $(A14)$.

\textbf{Assumption (A15).} If $\{ \lambda_n \} \subset \Lambda$ and $\{ c_n \}$ are bounded sequences and a sequence
$\{ y_n \} \subset Y$ is such that $\dist(y_n, K) \to 0$ as $n \to \infty$, then the sequences $\{ \lambda_{in} \}$
are also bounded and $\dist(y_{in}, K_i) \to 0$ as $n \to \infty$ for all $i \in I$. Consequently, one has
\[
  \limsup_{n \to \infty} \Phi_i(y_{in}, \lambda_{in}, c_n) \le 0 \quad \forall i \in I.
\]
Applying these inequalities one can readily check that 
$\limsup\limits_{n \to \infty} \Phi(y_n, \lambda_n, c_n) \le 0$.

The proof of the claim of the theorem for the restricted versions of assumptions $(A13)$--$(A15)$, $(A13)_s$, and
$(A14)_s$ is essentially the same as the proofs for the non-restricted versions of these assumptions. 
\end{proof}

\begin{remark}
Let us note that assumptions $(A1)$ and $(A7)$ are satisfied for all particular augmented Lagrangians presented in this
paper and all augmented Lagrangians known to the author.
\end{remark}

Thus, the previous theorem allows one to analyse augmented Lagrangians for each particular type of cone constraints
separately and then simply define an augmented Lagrangian for problems with several different types of cone constraints
as in \eqref{eq:SeparableAugmLagr}. Then the basic assumptions will be satisfied for this augmented Lagrangian by
Theorem~\ref{thrm:BasicAssumptionSeparateCase}.

\section{Examples of augmented Lagrangians}
\label{sect:Examples}

Let us present some particular examples of augmented Lagrangians for various types of cone constraints and discuss which
of the basic assumptions are satisfied for these augmented Lagrangians. Our aim is to show that the basic assumptions
are not restrictive and satisfied for most augmented Lagrangians appearing in the literature.

Many examples of augmented Lagrangians were already presented in the author's previous paper \cite{Dolgopolik2018}.
However, many of the basic assumptions from this paper are completely new and were not used in \cite{Dolgopolik2018}
(e.g. assumptions $(A9)$, $(A10)$, and $(A13)$--$(A15)$, as well as their stronger and restricted versions). Therefore,
for the sake of completeness we will present detailed descriptions of all examples, even though they partially overlap
with the contents of \cite[Section~3]{Dolgopolik2018}.

\subsection{Augmented Lagrangians for problems with general cone constraints}

First, we consider an augmented Lagrangian that is defined for any cone constrained optimization problem of the form
$(\mathcal{P})$. Since particular versions of this augmented Lagrangian are apparently studied and used in optimization
methods more often than any other augmented Lagrangian, we will discuss it in more details than other examples.

\begin{example}[Rockafellar-Wets' augmented Lagrangian] \label{ex:RockafellarWetsAugmLagr}
Let $\sigma \colon Y \to [0, + \infty]$ be such that $\sigma(0) = 0$ and $\sigma(y) \ne 0$ for any $y \ne 0$.
In particular, one can set $\sigma(y) = \| y \|$ or $\sigma(y) = 0.5 \| y \|^2$. Define
\begin{equation} \label{def:RockafellarWetsAugmLagr}
  \Phi(y, \lambda, c) = \inf_{p \in K - y} \big( - \langle \lambda, p \rangle + c \sigma(p) \big).
\end{equation}
The augmented Lagrangian with the term $\Phi$ defined in this way was first introduced by Rockafellar and Wets
\cite[Section~11.K]{RockafellarWets} (see also \cite{HuangYang2003,ShapiroSun,HuangYang2005,Dolgopolik2017}). Various
generalizations of this augmented Lagrangian were studied in \cite{BurachikIusemMelo,WangYangYang,WangLiuQu}.

\begin{lemma}
Let $\Lambda = Y^*$. Then the function $\Phi$ defined in \eqref{def:RockafellarWetsAugmLagr} satisfies:
\begin{enumerate}
\item{assumptions $(A1)$--$(A4)$, $(A7)$, $(A9)$, $(A9)_s$, and $(A10)$ in the general case;}

\item{assumptions $(A5)$ and $(A6)$, if $\sigma(ty) = o(t)$ as $t \to 0$ 
(that is, $\sigma(ty) / t \to 0$  as $t \to 0$) for any $y \in Y$;}

\item{assumption $(A8)$, if the function $\sigma$ is convex;}

\item{assumption $(A11)$ with $\Phi_0(\lambda) \equiv \lambda$, if $\sigma(y) = 0.5 \| y \|^2$ and $Y$ is a Hilbert
space;}

\item{assumptions $(A12)$ and $(A12)_s$, if $\sigma$ has a valley at zero, that is, for any neighbourhood $U$ of zero
in $Y$ there exists $\delta > 0$ such that $\sigma(y) \ge \delta$ for all $y \in Y \setminus U$;}

\item{assumptions $(A13)$ and $(A13)_s$, if $\sigma(y) \ge \omega(\| y \|)$ for all $y \in Y$ and some continuous
function $\omega \colon [0, + \infty) \to [0, + \infty)$ such that $\omega(t) = 0$  if and only if $t = 0$, and
$\liminf_{t \to + \infty} \omega(t) / t > 0$;
}

\item{assumptions $(A14)$ and $(A14)_s$, if $\sigma$ is continuous at zero and there exists a continuous
function $\omega \colon [0, + \infty) \to [0, + \infty)$ such that $\sigma(y) \ge \omega(\| y \|)$ for all $y \in Y$,
$\omega(t) = 0$ if and only if $t = 0$, and $\liminf_{t \to + \infty} \omega(t) / t > 0$;
}

\item{assumption $(A15)$, if the function $\sigma$ is continuous at zero.}
\end{enumerate}
\end{lemma}

\begin{proof}
We divide the proof of the lemma into several parts corresponding to its separate statements.

\textbf{Part 1.} Assumptions $(A1)$ (set $p = 0$), $(A2)$ (set $\lambda = 0$), and $(A7)$ are obviously satisfied. 
Assumption $(A3)$ is satisfied for $\lambda \in Y^*$ from the separation theorem for the sets $\{ y \}$ and
$K$. Assumption $(A4)$ is satisfied, since if $\langle \lambda, y \rangle = 0$ and $\lambda \in K^*$, then
\[
  \Phi(y, \lambda, c) = \inf_{p \in K} \big( - \langle \lambda, p \rangle + c \sigma(p - y) \big)
  \ge c \inf_{p \in K} \sigma(p - y) \ge 0, 
\]
which along with assumption $(A1)$ implies that $\Phi(y, \lambda, c) = 0$. Assumptions $(A9)$, $(A9)_s$, and $(A10)$ 
are satisfied by virtue of the fact that the function $(\lambda, c) \mapsto \Phi(y, \lambda, c)$ is the infimum of 
a family of linear functions. 

\textbf{Part 2.} Let $y \in K$ and $\lambda \in K^*$ be such that $\langle \lambda, y \rangle \ne 0$. For any $t$ in a
sufficiently small neighbourhood of zero one has $z(t) = (1 + t \sign(\langle \lambda, y \rangle)) y \in K$, which
implies that for $p = z(t) - y \in K - y$ the following inequalities hold true:
\[
  \Phi(y, \lambda, c) \le - \langle \lambda, p \rangle + c \sigma(p)
  = - t \big| \langle \lambda, y \rangle \big| + c \sigma\big( t \sign(\langle \lambda, y \rangle) y \big).
\]
The last expression is negative for any sufficiently small $t$, since $\sigma(t y) = o(t)$, that is, assumption $(A5)$
holds true. 

Let now $y \in K$ and $\lambda \in \Lambda \setminus K^*$. For any such $\lambda$ one can find $p_0 \in K$ for which 
$\langle \lambda, p_0 \rangle < 0$. Then putting $p = t p_0$ for $t > 0$ (note that $t p_0 + y \in K$, since $y \in K$
and $K$ is a convex cone, which yields $p \in K - y$) one gets
\[
  \Phi(y, \lambda, c) \le - t \langle \lambda, p_0 \rangle + c \sigma(t p_0) < 0
\]
for any sufficiently small $t$, thanks to the fact that $\sigma(t p_0) = o(t)$. 

\textbf{Part 3.} Fix any $\alpha \in (0, 1)$ and $y_1, y_2 \in Y$. Choose any $M_i > \Phi(y_i, \lambda, c)$, 
$i \in \{ 1, 2 \}$. By definition one can find $p_i \in K - y_i$ such that 
$M_i > - \langle \lambda, p_i \rangle + c \sigma(p_i)$, $i \in \{ 1, 2 \}$. Then for 
$p(\alpha) = \alpha p_1 + (1 - \alpha) p_2 \in K - (\alpha y_1 + (1 - \alpha) y_2)$ one has
\begin{align*}
  \Phi(\alpha y_1 + (1 - \alpha) y_2, \lambda, c) &\le - \langle \lambda, p(\alpha) \rangle + c \sigma(p(\alpha))
  \\
  &\le \alpha \Big( - \langle \lambda, p_1 \rangle + c \sigma(p_1) \Big) 
  + (1 - \alpha) \Big( - \langle \lambda, p_2 \rangle + c \sigma(p_2) \Big)
  \\
  &< \alpha M_1 + (1 - \alpha) M_2.
\end{align*}
Hence by \cite[Theorem~I.4.2]{Rockafellar} one can conclude that the function $\Phi(\cdot, \lambda, c)$ is convex. Let
us now show that it is non-decreasing with respect to the binary relation induced by the cone $-K$. 

Indeed, fix any $y_1, y_2 \in Y$ such that $y_1 \preceq y_2$, i.e. $y_2 - y_1 \in -K$. One can obviously suppose that 
$\Phi(y_2, \lambda, c) < + \infty$. By definition for any $M > \Phi(y_2, \lambda, c)$ one can find $p \in K - y_2$ such
that $M \ge - \langle \lambda, p \rangle + c \sigma(p)$. Let $z \in K$ be such that $p = z - y_2$. Note that 
$z - (y_2 - y_1) \in K$ due to the fact that $-(y_2 - y_1) \in K$ by our assumption.
Then $p = z - (y_2 - y_1) - y_1 \in K - y_1$, which yields
\[
  \Phi(y_1, \lambda, c) \le - \langle \lambda, p \rangle + c \sigma(p) \le M.
\]
Since $M > \Phi(y_2, \lambda, c)$ was chosen arbitrarily, one can conclude that the function $\Phi(\cdot, \lambda, c)$
is non-decreasing with respect to the binary relation $\preceq$.

\textbf{Part 4.} The proof of this statement of the lemma can be found in \cite{ShapiroSun}.

\textbf{Part 5.} Fix any $c_0 > 0$. For any $y \in Y$ and $\lambda \in \Lambda$ such that $\Phi(y, \lambda, c_0)$ is
finite one has
\begin{align*}
  \Phi(y, \lambda, c) &= \inf_{p \in K - y} \big( - \langle \lambda, p \rangle + (c - c_0 + c_0) \sigma(p) \big)
  \\
  &\ge \inf_{p \in K - y} \big( - \langle \lambda, p \rangle + c_0 \sigma(p) \big)
  + (c - c_0) \inf_{p \in K - y} \sigma(p)
\end{align*}
for any $c \ge c_0$. If $\dist(y, K) \ge r$, then $\| p \| \ge r$ for any $p \in K - y$. Therefore by our assumption
there exists $\delta > 0$, independent on $\lambda \in \Lambda$, $c \ge c_0$, and 
$y \in \{ z \in K \mid \dist(z, K) \ge r \}$, such that
\[
  \Phi(y, \lambda, c) - \Phi(y, \lambda, c_0) \ge (c - c_0) \delta.
\]
With the use of this inequality one can easily prove that assumptions $(A12)$ and $(A12)_s$ hold true.

\textbf{Part 6.} Let $\{ c_n \} \subset (0, + \infty)$ be an increasing unbounded sequence. Fix any bounded set 
$\Lambda_0 \subset \Lambda$. Then for any $y \in Y$ and $\lambda \in \Lambda_0$ one has
\[
  \Phi(y, \lambda, c_n) \ge \inf_{p \in K - y} \big( - \| \lambda \| \| p \| + c_n \omega(\| p \|) \big)
  \ge \inf_{t \ge 0} \big( - R t + c_n \omega(t) \big),  
\]
where $R > 0$ is such that $\| \lambda \| \le R$ for all $\lambda \in \Lambda_0$. By applying the assumptions on 
the function $\omega$ one can readily check that
\[
  \liminf_{n \to \infty} \inf_{t \ge 0} \big( - R t + c_n \omega(t) \big) \ge 0,
\]
which yields
\begin{equation} \label{eq:RockafellarWetsLimitingValue}
  \liminf_{n \to \infty} \inf_{\lambda \in \Lambda_0} \inf_{y \in Y} \Phi(y, \lambda, c_n) \ge 0.
\end{equation}
Consequently, assumptions $(A13)$ and $(A13)_s$ hold true. 

\textbf{Part 7.} Let $\{ \lambda_n \} \subset Y^*$ be a bounded sequence and a sequence 
$\{ c_n \} \subset (0, + \infty)$ be such that $c_n \to + \infty$ as $n \to \infty$. Due to the continuity of $\sigma$
at zero, for any $n \in \mathbb{N}$ one can find $\delta_n > 0$ such that for all $p \in B(0, \delta_n)$ one has
$\sigma(p) \le 1 / (c_n n)$. One can obviously suppose that $\delta_n \to 0$ as $n \to \infty$.

Define $t_n = \delta_n / 2$ for all $n \in \mathbb{N}$. Then for any sequence $\{ y_n \} \subset Y$ such that
$\dist(y_n, K) \le t_n$ for all $n \in \mathbb{N}$ one can find a sequence $\{ z_n \} \subset K$ such that 
$\| z_n - y_n \| \le \delta_n$ for all $n \in \mathbb{N}$. Observe that for $p_n = z_n - y_n \in K - y_n$ one has
\begin{equation} \label{eq:RockafellarWetsUpperEstim}
  \Phi(y_n, \lambda_n, c_n) \le - \langle \lambda_n, p_n \rangle + c_n \sigma(p_n) 
  \quad \forall n \in \mathbb{N},
\end{equation}
and the right-hand side of this inequality converges to zero, since $p_n \to 0$ as $n \to \infty$ and the sequence 
$\{ \lambda_n \}$ is bounded. Combining this fact with inequality \eqref{eq:RockafellarWetsLimitingValue} one
obtains that we have found a sequence $\{ t_n \}$ for which assumptions $(A14)$ and $(A14)_s$ hold true.

\textbf{Part 8.} Let $\{ \lambda_n \} \subset Y^*$ and $\{ c_n \} \subset (0, + \infty)$ be  bounded sequences, and 
a sequence $\{ y_n \} \subset Y$ be such that $\dist(y_n, K) \to 0$ as $n \to \infty$. Then thanks to the continuity 
of $\sigma$ at zero one can find a sequence $\{ z_n \} \subset K$ such that $\| p_n \| \to 0$ and $\sigma(p_n) \to 0$
as $n \to \infty$, where $p_n = z_n - y_n$. Now, applying inequality \eqref{eq:RockafellarWetsUpperEstim} and 
taking into account the fact that the right-hand side of this inequality obviously converges to zero one can conclude
that assumption $(A15)$ is valid.
\end{proof}

Thus, \textit{all} basic assumptions are satisfied for $\sigma(y) = 0.5 \| y \|^2$. In the other important case
of the sharp Lagrangian
\cite{Gasimov,BurachikGasimov,BurachickIusemMelo_SharpLagr,BagirovOzturkKasimbeyli,BurachikLiu2023}, that is, the case
when $\sigma(y) = \| y \|$, all assumptions, except for $(A5)$, $(A6)$, and $(A11)$, hold true. It should be noted that
these assumptions are \textit{not} used in the main results presented in this article and needed only to strengthen some
of these results in the convex case.

\begin{remark}
Note that neither assumption $(A5)$ nor assumption $(A6)$ are satisfied for $\sigma(y) = \| y \|$ due to the fact that 
in this case $\Phi(y, \lambda, c) \ge 0$ for all $y \in Y$, if $c > \| \lambda \|$. Thus, if $\sigma(t y) \ne o(t)$
for some $y \in Y$, then assumptions $(A5)$ and $(A6)$ might not hold true.
\end{remark}
\end{example}

\subsection{Augmented Lagrangians for problems with equality constraints}

Let us now consider the following equality constrained problem:
\[
  \min\: f(x) \quad \text{subject to} \quad G(x) = 0, \quad x \in Q.
\]
The constraint $G(x) = 0$ can obviously be rewritten as the cone constraint $G(x) \in K$, if one puts $K = \{ 0 \}$.
The binary relation $\preceq$ in this case coincides with the equality relation ``$=$'', and all functions are
non-decreasing with respect to this relation.

\begin{example}[Hestenes-Powell's augmented Lagrangian] \label{ex:HestenesPowellAugmLagr}
Define $\Lambda = Y^*$ and
\[
  \Phi(y, \lambda, c) = \langle \lambda, y \rangle + \frac{c}{2} \| y \|^2.
\]
Then the corresponding augmented Lagrangian is a particular case of the augmented Lagrangian from
Example~\ref{ex:RockafellarWetsAugmLagr} with $\sigma(y) = 0.5 \| y \|^2$. Therefore, this function $\Phi$ satisfies
all basic assumptions, except for assumption $(A11)$, in the general case, and it satisfies assumption $(A11)$ with
$\Phi_0(\lambda) \equiv \lambda$, if $Y$ is a Hilbert space. Note that in the case $Y = \mathbb{R}^m$ 
the corresponding augmented Lagrangian $\mathscr{L}(\cdot)$ coincides with the Hestenes-Powell augmented Lagrangian
\cite{Hestenes,Powell,BirginMartinez}.
\end{example}

\begin{example}[sharp Lagrangian] \label{ex:SharpLagrangian_eq}
Define $\Lambda = Y^*$ and
\[
  \Phi(y, \lambda, c) = \langle \lambda, y \rangle + c \| y \|.
\]
Then the corresponding augmented Lagrangian is a particular case  of the augmented Lagrangian from
Example~\ref{ex:RockafellarWetsAugmLagr} with $\sigma(y) = \| y \|$. Therefore, this function $\Phi$ satisfied all
basic assumptions, except for assumptions $(A5)$, $(A6)$, and $(A11)$.
\end{example}

In the case of equality constrained problems of the form
\[
  \min\: f(x) \quad \text{subject to} \quad g_j(x) = 0, \quad i \in I, \quad x \in Q,
\]
where $I = \{ 1, \ldots, m \}$ and $g_i \colon X \to \mathbb{R}$ are given function, one can define a more general
class of augmented Lagrangians. This problem can be written as the problem $(\mathcal{P})$ with $Y = \mathbb{R}^m$,
$G(\cdot) = (g_1(\cdot), \ldots, g_m(\cdot))$, and $K = \{ 0 \}$. Note that in this particular case the dual space
$Y^*$ can be identified with $\mathbb{R}^m$.

\begin{example}[Mangasarian's augmented Lagrangian] \label{ex:MangasarianAugmLagr_eq}
Let $\phi \colon \mathbb{R} \to \mathbb{R}$ be a twice differentiable strictly convex function such that 
$\phi(0) = \phi'(0) = 0$ and $\phi'(\cdot)$ is surjective. Define
\[
  \Phi(y, \lambda, c) = \sum_{i = 1}^m \frac{1}{c} \big( \phi(c y_i + \lambda_i) - \phi(\lambda_i) \big)
\]
for all $y = (y_1, \ldots, y_m) \in Y$ and $\lambda = (\lambda_1, \ldots, \lambda_m) \in \Lambda$. Then the
corresponding augmented Lagrangian $\mathscr{L}(\cdot)$ coincides with Mangasarian's augmented Lagrangian from
\cite{Mangasarian} (see also \cite{WuLuo2012b}). In the case $\phi(t) = 0.5 t^2$, this augmented Lagrangian coincides
with the Hestenes-Powell augmented Lagrangian. One can also put, e.g. $\phi(t) = |t| t^{2n} / (2n + 1)$ or 
$\phi(t) = t^{2n}/2n$ for any $n \in \{ 1, 2, \ldots \}$.

Let $\Lambda = Y^* \cong \mathbb{R}^m$. Then one can readily check that all assumptions, except for assumptions
$(A9)$ and $(A9)_s$, are satisfied in the general case (assumption $(A11)$ holds true with 
$\Phi_0(\lambda) \equiv (\phi'(\lambda_1), \ldots, \phi'(\lambda_m))$). Assumptions $(A9)$ and $(A9)_s$ are
satisfied if and only if $\phi(t) = a t^2$ for some $a > 0$. The validity of these assumptions in the case
when the function $\phi$ is quadratic can be readily verified directly. Let us prove the converse statement.

Suppose that the function $\lambda \to \Phi(y, \lambda, c)$ is concave. Then applying the second order derivative test
for concavity one gets that $\phi''(\lambda_i) \ge \phi''(c y_i + \lambda_i)$ for all $\lambda_i, y_i \in \mathbb{R}$
and $c > 0$ or, equivalently, $\phi''(\cdot)$ is a constant function. Hence bearing in mind the conditions 
$\phi(0) = \phi'(0) = 0$ one gets that $\phi(t) = a t^2$ for some $a > 0$. 
\end{example}

\subsection{Augmented Lagrangians for problems with inequality constraints}
\label{subsect:AugmLagrInequalConstr}

Next we will present several examples of augmented Lagrangians for the inequality constrained problem
\begin{equation} \label{prob:InequalityConstr}
  \min\: f(x) \quad \text{subject to} \quad g_i(x) \le 0, \quad i \in I, \quad x \in Q,
\end{equation}
where $I = \{ 1, \ldots, m \}$ and $g_i \colon X \to \mathbb{R}$ are given functions. This problem can be written 
as the problem $(\mathcal{P})$ with $Y = \mathbb{R}^m$, $G(\cdot) = (g_1(\cdot), \ldots, g_m(\cdot))$, and
$K = \mathbb{R}_-^m$, where $\mathbb{R}_- = (- \infty, 0]$. The dual space $Y^*$ can be identified with $\mathbb{R}^m$, 
while $K^*$ can be identified with $\mathbb{R}^m_+$, where $\mathbb{R}_+ = [0, + \infty)$. The binary relation 
$\preceq$ in this case is the coordinate-wise partial order.

All particular augmented Lagrangians for problem \eqref{prob:InequalityConstr} used in optimization methods and known to
the author are \textit{separable} (except for nonlinear Lagrangians; see
\cite{RubinovYang,BurachikRubinov,WangYangYang}), that is, the corresponding function $\Phi(y, \lambda, c)$ has the form
\begin{equation} \label{eq:SeparableAugmLagr_IneqConstr}
  \Phi(y, \lambda, c) = \sum_{i = 1}^m \Phi_i(y_i, \lambda_i, c) 
  \quad \forall y = (y_1, \ldots, y_m), \: \lambda = (\lambda_1, \ldots, \lambda_m)
\end{equation}
for some functions $\Phi_i \colon \mathbb{R}^2 \times (0, + \infty) \to \mathbb{R} \cup \{ \pm \infty \}$. Although 
one can can choose different functions $\Phi_i$ for different $i \in I$ (that is, for different inequality constraints),
to the best of the author's knowledge, only the case when $\Phi_i$ are the same for all $i \in I$ is considered in 
the vast majority of papers on augmented Lagrangians for inequality constrained problems.

\begin{example}[essentially quadratic/Hestenes-Powell-Rockafellar's augmented Lagrangian] \label{ex:HPR_AugmLagr}
Let $\phi \colon \mathbb{R} \to \mathbb{R}$ be a twice continuously differentiable strictly convex function such that 
$\phi(0) = \phi'(0) = 0$, and the derivative $\phi'(\cdot)$ is surjective. Following Bertsekas
\cite[Section~5.1.2, Example~1]{Bertsekas}, for any $y, \lambda \in \mathbb{R}$ define
\[
  P(y, \lambda) = \begin{cases}
    \lambda y + \phi(y), & \text{if } \lambda + \phi'(y) \ge 0,
    \\
    \min_{t \in \mathbb{R}} \big( \lambda t + \phi(t) \big), & \text{otherwise}
  \end{cases}
\]
(note that the minimum is finite and attained at any $t$ such that $\lambda + \phi'(t) = 0$, which exists due to 
the surjectivity of $\phi'(\cdot)$) and put
\[
  \Phi_i(y_i, \lambda_i, c) = \frac{1}{c} P(c y_i, \lambda) \quad \forall i \in I.
\]
The corresponding augmented Lagrangian $\mathscr{L}(\cdot)$ is called \textit{the essentially quadratic} augmented
Lagrangian for problem \eqref{prob:InequalityConstr} (see \cite{SunLiMcKinnon,LiuYang2008,WangLi2009}). In the case
$\phi(t) = t^2 / 2$ one has
\[
  \Phi_i(y_i, \lambda_i, c) = \lambda_i \max\left\{ y_i, - \frac{\lambda_i}{c} \right\}
  + \frac{c}{2} \max\left\{ y_i, - \frac{\lambda_i}{c} \right\}^2,
\]
and $\mathscr{L}(\cdot)$ is the well-known Hestenes-Powell-Rockafellar augmented Lagrangian 
\cite{Hestenes,Powell,Rockafellar73,Rockafellar74,Rockafellar1993,BirginMartinez}, which is a particular case of the
augmented Lagrangian from Example~\ref{ex:RockafellarWetsAugmLagr} with $\sigma(y) = \| y \|^2 / 2$ and $\| \cdot \|$
being the Euclidean norm.

Let $\lambda = Y^* = \mathbb{R}^m$. Then one can readily verify that all basic assumptions, except for assumption
$(A9)_s$, hold true in the general case (assumption $(A11)$ is satisfied with $\Phi_0(\lambda) \equiv \lambda$).
Assumption $(A9)_s$ is satisfied for $\phi(t) = a t^2$, $a > 0$.
\end{example}

\begin{example}[cubic augmented Lagrangian] \label{ex:CubicAugmLagr}
Let
\[
  \Phi_i(y_i, \lambda_i, c) 
  = \frac{1}{3c} \Big( \max\big\{ \sign(\lambda_i) \sqrt{|\lambda_i|} + c y_i, 0 \big\}^3 - |\lambda_i|^{3/2} \Big)
  \quad \forall i \in I.
\]
Then $\mathscr{L}(\cdot)$ coincides with \textit{the cubic augmented Lagrangian} \cite{Kiwiel}. One can easily
check that all basic assumptions, except for assumptions $(A9)$ and $(A9)_s$, are satisfied in this case with 
$\Lambda = Y^* = \mathbb{R}^m$ (assumption $(A11)$ is satisfied with $\Phi_0(\lambda) \equiv \lambda$). Assumption
$(A9)$ holds true, provided $\Lambda \subseteq K^*$, while assumption $(A9)_s$ is not satisfied for any choice of 
$\Lambda$.
\end{example}

\begin{example}[Mangasarian's augmented Lagrangian] \label{ex:MangasarianAugmLagr_ineq}
Let $\phi \colon \mathbb{R} \to \mathbb{R}$ be a twice continuously differentiable strictly convex function such that 
$\phi(0) = \phi'(0) = 0$ and the function $\phi'(\cdot)$ is surjective. Define
\begin{equation} \label{eq:MangasarianAugmLagr}
  \Phi_i(y_i, \lambda_i, c) 
  = \frac{1}{c} \Big( \phi\big( \max\{ c y_i + \lambda_i, 0 \} \big) - \phi(\lambda_i) \Big) \quad \forall i \in I.
\end{equation}
Then $\mathscr{L}(\cdot)$ coincides with the augmented Lagrangian introduced by Mangasarian \cite{Mangasarian}
and studied, e.g. in \cite{WuLuo2012b}. Let $\Lambda = Y^* = \mathbb{R}^m$. Then all basic assumptions, except for
assumptions $(A9)$ and $(A9)_s$, hold true (assumption $(A11)$ is satisfied with 
$\Phi_0(\lambda) = (\phi'(\max\{ \lambda_1, 0 \}), \ldots, \phi'(\max\{ \lambda_m, 0 \}))$). Assumptions $(A9)$ and
$(A9)_s$ are satisfied for $\phi(t) = a t^2$ with $a > 0$.
\end{example}

\begin{example}[exponential-type augmented Lagrangian] \label{ex:ExpTypeAugmLagr}
Let $\phi \colon \mathbb{R} \to \mathbb{R}$ be a twice differentiable strictly increasing function such that 
$\phi(0) = 0$. Define
\[
  \Phi_i(y_i, \lambda_i, c) = \frac{\lambda_i}{c} \phi(c y_i) \quad \forall i \in I.
\]
If $\phi(t) = e^t - 1$, then $\mathscr{L}(\cdot)$ coincides with \textit{the exponential penalty function}
\cite{Bertsekas,TsengBertsekas,SunLiMcKinnon,LiuYang2008,WangLi2009}. In turn, if $\phi(t) = 2 (\ln(e^t + 1) - \ln 2)$,
then $\mathscr{L}(\cdot)$ is the Polyak's \textit{log-sigmoid Lagrangian} \cite{Polyak2001,Polyak2002}. In the general
case we call the corresponding function $\mathscr{L}(\cdot)$ \textit{the exponential-type augmented Lagrangian}.

Let $\Lambda = K^* = \mathbb{R}^m_+$. Then assumptions $(A1)$--$(A6)$, $(A9)$, $(A10)$, and $(A15)$ are satisfied in 
the general case. Assumptions $(A7)$ and $(A8)$ hold true, provided the function $\phi$ is convex. Assumption $(A11)$ 
is satisfied with $\Phi_0(\lambda) \equiv \phi'(0) \lambda$ if and only if $\phi'(0) \ne 0$. Restricted versions
of assumptions $(A13)$, $(A13)_s$, $(A14)$, and $(A14)_s$ (see Remark~\ref{rmrk:RestrictedAssumptions}) are satisfied if
and only if $\phi(t) / t \to 0$ as $t \to - \infty$, while non-restricted versions of these assumptions are satisfied if
and only if the function $\phi$ is bounded below. Finally, assumptions $(A9)_s$, $(A12)$, and $(A12)_s$ 
(put $\lambda = 0$) are never satisfied for the exponential-type augmented Lagrangian. 

Thus, all basic assumptions, except for assumptions $(A9)_s$, $(A12)$, and $(A12)_s$, are valid for the exponential
penalty function and the log-sigmoid Lagrangian.
\end{example}

\begin{example}[penalized exponential-type augmented Lagrangian] \label{ex:PenalizedExpTypeAugmLagr}
Suppose that $\phi \colon \mathbb{R} \to \mathbb{R}$ is a twice differentiable strictly increasing function such that 
$\phi(0) = 0$, and $\xi \colon \mathbb{R} \to \mathbb{R}$ is a twice continuously differentiable non-decreasing 
function such that $\xi(t) = 0$ for all $t \le 0$ and $\xi(t) > 0$ for all $t > 0$ (for example, one can set 
$\xi(t) = \max\{ 0, t \}^3$). Following Bertsekas \cite[Section~5.1.2, Example~2]{Bertsekas} define
\[
  \Phi_i(y_i, \lambda_i, c) = \frac{\lambda_i}{c} \phi(c y_i) + \frac{1}{c} \xi(c y_i) \quad \forall i \in I.
\]
Then the function $\mathscr{L}(\cdot)$ is called \textit{the penalized exponential-type augmented Lagrangian}
\cite{SunLiMcKinnon,LiuYang2008,WangLi2009}, since it is obtained from the augmented Lagrangian from the previous
example by adding the penalty term $\xi(c y_i) / c$.

Let $\Lambda = K^* = \mathbb{R}^m_+$. Then assumptions $(A1)$--$(A6)$, $(A9)$, $(A10)$, and $(A15)$, are satisfied in 
the general case. Assumptions $(A7)$ and $(A8)$ hold true, provided the functions $\phi$ and $\xi$ are convex.
Assumption $(A11)$ is satisfied with $\Phi_0(\lambda) \equiv \phi'(0) \lambda$ if and only if $\phi'(0) \ne 0$. 
Assumptions $(A12)$ and $(A12)_s$ are valid, provided $\xi(t) / t \to + \infty$ as $t \to \infty$ and $\phi$ is either 
bounded below or convex. Restricted assumptions $(A13)$, $(A13)_s$, $(A14)$, and $(A14)_s$, hold true if and only if
$\phi(t) / t \to 0$ as $t \to - \infty$, while non-restricted versions of these assumptions hold true if and only if
$\phi$ is bounded below.

Thus, if $\phi(t) = e^t - 1$ or $\phi(t) = 2 (\ln(e^t + 1) - \ln 2)$ and $\xi(t) = \max\{ 0, t \}^3$, then all basic
assumptions, except for assumption $(A9)_s$, hold true. 
\end{example}

\begin{example}[p-th power augmented Lagrangian] \label{ex:pthPowerAugmLagr}
Let $b \ge 0$ and a continuous non-decreasing function $\phi \colon \mathbb{R} \to \mathbb{R}_+$ be such that 
$\phi(t) > \phi(b) > 0$ for all $t > b$. For example, one can set $\phi(t) = e^t$ with $b \ge 0$ or 
$\phi(t) = \max\{ 0, t \}$ with $b > 0$. By our assumption the inequality $g_i(x) \le 0$ is satisfied if and only if
$\phi(g_i(x) + b) / \phi(b) \le 1$. Furthermore, $\phi(g_i(x) + b) \ge 0$ for all $x \in X$. Define
\[
  \Phi_i(y_i, \lambda_i, c) 
  = \frac{\lambda_i}{c + 1} \left( \left( \frac{\phi(y_i + b)}{\phi(b)} \right)^{c + 1} - 1 \right) 
  \quad \forall i \in I.
\]
Then $\mathscr{L}(\cdot)$ coincides with \textit{the p-th power augmented Lagrangian}
\cite{LiSun2001,WuLuo2012a,LiuYang2008}.

Let $\Lambda = K^* = \mathbb{R}^m_+$. Then assumptions $(A1)$--$(A7)$, $(A9)$, $(A10)$, $(A13)$--$(A15)$, $(A13)_s$, and
$(A14)_s$ hold true. Assumption $(A8)$ is satisfied, if the function $\phi$ is convex. Assumption $(A11)$ is satisfied
with $\Phi_0(\lambda) \equiv \phi'(b) \lambda$, provided $\phi$ is differentiable and $\phi'(b) \ne 0$. Finally,
assumptions $(A9)_s$, $(A12)$, and $(A12)_s$ are not satisfied for the p-th power augmented Lagrangian.
\end{example}

\begin{remark}
Let $\phi$ be as in the previous example and $\xi$ be as in Example~\ref{ex:PenalizedExpTypeAugmLagr}. Then by analogy
with the penalized exponential-type augmented Lagrangian one can define the \textit{penalized} p-th power augmented
Lagrangian as follows:
\[
  \Phi_i(y_i, \lambda_i, c) 
  = \frac{\lambda_i}{c + 1} \left( \left( \frac{\phi(y_i + b)}{\phi(b)} \right)^{c + 1} - 1 \right)
  + \frac{1}{c} \xi(c y_i) \quad \forall i \in I.
\]
If the function $\phi$ is convex and differentiable, $\phi'(b) \ne 0$, the function $\xi$ is convex, and 
$\xi(t) / t \to + \infty$ as $t \to \infty$, then one can verify that the penalized p-th power augmented Lagrangian
satisfies all basic assumption, except for assumption $(A9)_s$. Let us also mention that one can apply this trick of
adding the penalty term $\xi(c y_i)/c$ to any other augmented Lagrangian for inequality constrained problems, if it does
not satisfy assumptions $(A12)$ and $(A12)_s$, in order to construct the penalized version of this augmented Lagrangian
satisfying assumptions $(A12)$ and $(A12)_s$ and having all other properties of the non-penalized version.
\end{remark}

\begin{example}[hyperbolic-type augmented Lagrangian] \label{eq:HyperbolicAugmLagr}
Let $\phi \colon \mathbb{R} \to \mathbb{R}$ be a twice differentiable strictly increasing convex function such that 
$\phi(0) = 0$. Define
\[
  \Phi_i(y_i, \lambda_i, c) = \frac{1}{c} \phi(c \lambda_i y_i) \quad \forall i \in I.
\]
If $\phi(t) = t + \sqrt{t^2 + 1} - 1$, then $\mathscr{L}(\cdot)$ coincides with the hyperbolic augmented
Lagrangian \cite{Xavier,RamirezXavier}. In the general case we call such function $\mathscr{L}(x, \lambda, c)$
\textit{the hyperbolic-type augmented Lagrangian}.

Let $\Lambda = K^* = \mathbb{R}^m_+$. Then assumptions $(A1)$--$(A8)$, $(A10)$, $(A11)$, and $(A15)$ are satisfied in
the general case (assumption $(A11)$ is satisfied with $\Phi_0(\lambda) \equiv \phi'(0) \lambda$). Assumption $(A9)$ is
satisfied if and only if $\phi$ is a linear function. Restricted assumptions $(A13)$, $(A13)_s$, $(A14)$, and $(A14)_s$,
hold true if and only if $\phi(t) / t \to 0$ as $t \to - \infty$, while non-restricted versions of these assumptions
hold true if and only if $\phi$ is bounded below. Finally, assumptions $(A9)_s$, $(A12)$, and $(A12)_s$ are never
satisfied for the hyperbolic-type augmented Lagrangian.
\end{example}

\begin{example}[modified barrier function] \label{ex:ModifiedBarrierFunction}
Let $\phi \colon (- \infty, 1) \to \mathbb{R}$ be a twice differentiable strictly increasing function such that 
$\phi(0) = 0$ and $\phi(t) \to + \infty$ as $t \to 1$. Define
\[
  \Phi_i(y_i, \lambda_i, c) = \begin{cases}
    \frac{\lambda_i}{c} \phi(c y_i), & \text{if } cy_i < 1,
    \\
    + \infty, & \text{otherwise}
  \end{cases}
  \qquad \forall i \in I.
\]
Then augmented Lagrangian $\mathscr{L}(\cdot)$ coincides with the modified barrier function introduced by 
R. Polyak \cite{Polyak92}. In particular, in the case $\phi(t) = - \ln(1 - t)$ the augmented Lagrangian 
$\mathscr{L}(\cdot)$ is the modified Frisch function, while  in the case $\phi(t) = 1/(1 - t) - 1$ 
the augmented Lagrangian $\mathscr{L}(\cdot)$ is the modified Carrol function \cite{Polyak92} (see also
\cite{SunLiMcKinnon,LiuYang2008,WangLi2009}).

Let $\Lambda = K^* = \mathbb{R}^m_+$. Then assumptions $(A1)$--$(A6)$, $(A9)$, $(A10)$, $(A12)$, $(A12)_s$, and 
$(A15)$ are satisfied in the general case. Assumptions $(A7)$ and $(A8)$ hold true, if the function $\phi$ is 
convex. Assumption $(A11)$ is satisfied with $\Phi_0(\lambda) = \phi'(0) \lambda$ if and only if $\phi'(0) \ne 0$.
Restricted assumptions $(A13)$, $(A13)_s$, $(A14)$, and $(A14)_s$ hold true if and only if $\phi(t) / t \to 0$ as 
$t \to - \infty$, while non-restricted versions of these assumptions are valid if and only if the function $\phi$ is
bounded below. Finally, assumption $(A9)_s$ cannot hold true for the modified barrier function. 

Thus, the modified Carrol function satisfies all basic assumptions, except for assumption $(A9)_s$, while the modified
Frisch functions satisfies all assumptions, except for $(A9)_s$ and non-restricted assumptions $(A13)$, $(A13)_s$,
$(A14)$, and $(A14)_s$.
\end{example}

\begin{example}[He-Wu-Meng's augmented Lagrangian] \label{ex:HeWuMeng}
Let
\[
  \Phi_i(y_i, \lambda_i, c) = \frac{1}{c} \int_0^{c y_i} \big( \sqrt{t^2 + \lambda_i^2} + t \big) \, dt 
  \quad \forall i \in I.
\]
Then $\mathscr{L}(\cdot)$ coincides with the augmented Lagrangian introduced by He, Wu, and Meng
\cite{HeWuMeng}. Let us note that
\[
  \Phi_i(y_i, \lambda_i, c) = \frac{y_i}{2} \sqrt{(c y_i)^2 + \lambda_i^2} + \frac{c y_i^2}{2}
  + \frac{\lambda_i^2}{2c} \ln\big( \sqrt{(c y_i)^2 + \lambda_i^2} + c y_i \big) 
  - \frac{\lambda_i^2}{2c} \ln |\lambda_i|,
\]
if $\lambda_i \ne 0$, and $\Phi_i(y_i, 0, c) = c y_i(y_i + |y_i|) / 2$.

Let $\Lambda = Y^* = \mathbb{R}^m$. Then assumptions $(A1)$--$(A5)$, $(A7)$, $(A8)$, $(A10)$, $(A11)$ with 
$\Phi_0(\lambda) \equiv \lambda$, $(A12)$, $(A12)_s$, $(A15)$, and restricted assumptions $(A13)$, $(A13)_s$, $(A14)$,
and $(A14)_s$ are satisfied in the general case. Assumption $(A6)$ holds true if and only if $\Lambda \subseteq K^*$,
since $\Phi(0, \lambda, c) = 0$ for all $\lambda \in Y^*$. Finally, assumptions $(A9)$, $(A9)_s$, $(A13)$, $(A13)_s$,
$(A14)$, and $(A14)_s$ are not satisfied for He-Wu-Meng's augmented Lagrangian in the general case. Let us note that 
the non-restricted versions of the last 4 assumptions are not satisfied due to the fact that 
$\Phi_i(y, \lambda, c) \to - \infty$ as $y \to - \infty$.
\end{example}

\begin{remark}
Many more particular examples of augmented Lagrangians for inequality constrained optimization problems can be found
in \cite{BirginCastilloMartinez}.
\end{remark}

\subsection{Augmented Lagrangians for problems with second order cone constraints}

Let us now consider nonlinear second order cone programming problems:
\begin{equation} \label{prob:SecondOrderCone}
  \min\: f(x) \quad \text{subject to} \quad g_i(x) \in \mathcal{K}_{\ell_i + 1}, \quad i \in I, \quad x \in Q,
\end{equation}
where $g_i \colon X \to \mathbb{R}^{\ell_i + 1}$, $i \in I := \{ 1, \ldots, m \}$, are given functions,
\[
  \mathcal{K}_{\ell_i + 1} = \Big\{ y = (y^0, \overline{y}) \in \mathbb{R} \times \mathbb{R}^{\ell_i} \Bigm|
  y^0 \ge \| \overline{y} \| \Big\}
\]
is the second order (Lorentz/ice cream) cone of dimension $\ell_i + 1$, and $\| \cdot \|$ is the Euclidean norm.

Problem \eqref{prob:SecondOrderCone} can be rewritten as the problem $(\mathcal{P})$ with
\[
  Y = \mathbb{R}^{\ell_1 + 1} \times \ldots \times \mathbb{R}^{\ell_m + 1}, \quad
  K = \mathcal{K}_{\ell_1 + 1} \times \ldots \times \mathcal{K}_{\ell_m + 1},
\]
and $G(\cdot) = (g_1(\cdot), \ldots, g_m(\cdot))$. Note that the dual space $Y^*$ can be identified with $Y$,
while the polar cone $K^*$ can be identified with 
$(- \mathcal{K}_{\ell_1 + 1}) \times \ldots \times (- \mathcal{K}_{\ell_m + 1})$.

\begin{example}[Hestenes-Powell-Rockafellar's augmented Lagrangian] \label{ex:HPR_AugmLagr_2OrderCone}
For any $y = (y_1, \ldots, y_m) \in Y$, $\lambda = (\lambda_1, \ldots, \lambda_m) \in Y^*$, and $c > 0$ define
\[
  \Phi(y, \lambda, c) = \frac{c}{2} \sum_{i = 1}^m 
  \left[ \dist^2 \left( y_i + \frac{1}{c} \lambda_i, \mathcal{K}_{\ell_i + 1} \right) 
  - \frac{1}{c^2} \| \lambda_i \|^2 \right].
\]
This function $\Phi$ is a particular case of the function $\Phi$ from Example~\ref{ex:RockafellarWetsAugmLagr} with
$\sigma(y) = (\| y_1 \|^2 + \ldots + \| y_m \|^2) / 2$ (see \cite{LiuZhang2007,LiuZhang2008,ZhouChen2015}). Therefore it
satisfies all basic assumptions with $\Lambda = Y^*$ (assumption $(A11)$ holds true with 
$\Phi_0(\lambda) \equiv \lambda$).
\end{example}

To define another augmented Lagrangian for optimization problems with nonlinear second order cone constraints, recall
(see \cite{FukushimaLuoTseng,SunSun2008}) that in the context of such problems \textit{L\"{o}wner's operator} 
associated with a function $\psi \colon \mathbb{R} \to \mathbb{R}$ is defined as follows:
\[
  \Psi(y) = \frac{1}{2} \begin{pmatrix}
    \psi(y^0 + \| y \|) + \psi(y^0 - \| y \|)
    \\
    \Big( \psi(y^0 + \| y \|) - \psi(y^0 - \| y \|) \Big) \frac{\overline{y}}{\| \overline{y} \|} 
  \end{pmatrix}
\]
for any $y \in (y^0, \overline{y}) \in \mathbb{R} \times \mathbb{R}^l$ with $\overline{y} \ne 0$, and 
$\Psi(y) = (\psi(y^0), 0)$ for $y = (y^0, 0)$ . One can readily verify that if $\psi(0) = 0$ and the function $\psi$ 
is strictly increasing, then $-\Psi(-y) \in \mathcal{K}_{\ell + 1}$ for any $y \in \mathcal{K}_{\ell + 1}$, while 
$-\Psi(-y) \notin \mathcal{K}_{\ell + 1}$ for any $y \notin \mathcal{K}_{\ell + 1}$.

\begin{example}[exponential-type augmented Lagrangian/modified barrier function] \label{ex:ExpTypeAugmLagr_2OrderCone}
Let $\psi \colon \mathbb{R} \to \mathbb{R} \cup \{ + \infty \}$ be a non-decreasing convex function such that
$\dom \psi = (- \infty, \varepsilon_0)$ for some $\varepsilon_0 \in (0, + \infty]$, $\psi(t) \to + \infty$
as $t \to \varepsilon_0$ in the case $\varepsilon_0 < + \infty$ and $\psi(t) / t \to + \infty$ as $t \to +\infty$
in the case $\varepsilon_0 = +\infty$. Suppose also that $\psi$ is twice differentiable on $\dom \psi$, $\psi(0) = 0$, 
and $\psi'(0) = 1$. For any $y = (y_1, \ldots, y_m) \in Y$ and $\lambda = (\lambda_1, \ldots, \lambda_m) \in Y^*$
define
\[
  \Phi(y, \lambda, c) = - \frac{1}{c} \sum_{i = 1}^m \big\langle \lambda_i, \Psi(-c y_i) \big\rangle,
\]
if $- y^0 + \| \overline{y} \| < \varepsilon_0 / c$, and $\Phi(y, \lambda, c) = + \infty$ otherwise. The corresponding
augmented Lagrangian, which can be viewed as an extension of augmented Lagrangian from
Examples~\ref{ex:ExpTypeAugmLagr} and \ref{ex:ModifiedBarrierFunction} to the case of nonlinear second order cone
programming problems, was introduced in \cite{ZhangGuXiao2011}.

Let $\Lambda = K^*$. Then assumptions $(A1)$--$(A7)$, $(A9)$--$(A11)$, and $(A15)$ hold true in the general case 
(assumption $(A11)$ is satisfied with $\Phi_0(\lambda) \equiv \lambda$ by \cite[Lemma~3.1]{ZhangGuXiao2011}). 
Assumptions $(A12)$ and $(A12)_s$ are satisfied if and only if $\varepsilon_0 < + \infty$. Restricted assumptions
$(A13)$, $(A13)_s$, $(A14)$, and $(A14)_s$ hold true if and only if $\psi(t)/t \to 0$ as $t \to - \infty$, while
non-restricted versions of these assumptions hold true if and only if $\psi$ is bounded below. Finally, assumptions
$(A8)$ and $(A9)_s$ are not satisfied for the function $\Phi$ from this example.
\end{example}

\subsection{Augmented Lagrangians for problems with semidefinite constraints}

Let us now consider nonlinear semidefinite programming problems of the form:
\begin{equation} \label{prob:Semidefinite}
  \min\: f(x) \quad \text{subject to} \quad G(x) \preceq \mathbb{O}_{\ell \times \ell}, \quad x \in Q,
\end{equation}
where $G \colon X \to \mathbb{S}^{\ell}$ is a given function, $\mathbb{S}^{\ell}$ is the space of all real symmetric
matrices of order $\ell$ endowed with the inner product $\langle A, B \rangle = \trace(A B)$ and the corresponding norm
$\| A \|_F = \sqrt{\trace(A^2)}$, $A, B \in \mathbb{S}^{\ell}$, which is called the Frobenius norm, $\trace(\cdot)$ is 
the trace operator, $\mathbb{O}_{\ell \times \ell}$ is the zero matrix of order $\ell \times \ell$, and $\preceq$ is 
the L\"{o}wner partial order on the space $\mathbb{S}^{\ell}$, that is, $A \preceq B$ for some 
$A, B \in \mathbb{S}^{\ell}$ if and only if the matrix $B - A$ is positive semidefinite.

Problem \eqref{prob:Semidefinite} can be written as the problem $(\mathcal{P})$ with $Y = \mathbb{S}^{\ell}$ and $K$
being the cone of negative semidefinite matrices $\mathbb{S}^{\ell}_-$. Note that the binary relation induced by 
the cone $-K$ coincides with the L\"{o}wner partial order. The dual space $Y^*$ in this case can be identified with
$\mathbb{S}^{\ell}$, while the polar cone $K^*$ can be identified with the cone of positive semidefinite matrices
$\mathbb{S}^{\ell}_+$.

\begin{example}[Hestenes-Powell-Rockafellar's augmented Lagrangian] \label{ex:HPR_AugmLagr_SemiDef}
For any $y, \lambda \in \mathbb{S}^{\ell}$ and $c > 0$ define
\[
  \Phi(y, \lambda, c) = \frac{1}{2c} \Big( \trace\big([c y + \lambda]_+^2 \big) - \trace(\lambda^2) \Big),
\]
where $[ \cdot ]_+$ is the projection of a matrix onto the cone $\mathbb{S}_+^{\ell}$. This function $\Phi$ is a
particular case of the function $\Phi$ from Example~\ref{ex:RockafellarWetsAugmLagr} with 
$\sigma(y) = \| y \|_F^2 / 2$ and, therefore, it satisfies all basic assumptions with $\Lambda = Y^*$ (assumption
$(A11)$ holds true with $\Phi_0(\lambda) \equiv \lambda$). The corresponding augmented Lagrangian and optimization
methods for nonlinear semidefinite programming problems utilising this augmented Lagrangian were studied in
\cite{HuangYangTeo_Chapter,SunZhangWu2006,SunSunZhang2008,ZhaoSunToh2010,WenGoldfarbYin2010,Sun2011,LuoWuChen2012,
WuLuoDingChen2013,WuLuoYang2014,YamashitaYabe2015}.
\end{example}

One can also extend the exponential-type augmented Lagrangian/modified barrier function for inequality constrained
problems to the case of nonlinear semidefinite programming problems. To define such extension, recall that 
\textit{the matrix function}/\textit{L\"{o}wner's operator} \cite{HornJohnson,SunSun2008} associated with a function 
$\psi \colon \mathbb{R} \to \mathbb{R}$ is defined as follows:
\[
  \Psi(y) = E \diag\Big( \psi(\sigma_1(y)), \ldots, \psi(\sigma_{\ell}(y)) \Big) E^T \quad \forall y \in Y,
\]
where $y = E \diag(\sigma_1(y), \ldots, \sigma_{\ell}(y)) E^T$ is a spectral decomposition of a matrix 
$y \in \mathbb{S}^{\ell}$, while $\sigma_1(y), \ldots, \sigma_{\ell}(y)$ are the eigenvalue of $y$ listed in 
the decreasing order. Note that if the function $\psi$ is non-decreasing and $\psi(0) = 0$, then 
$\Psi(y) \in \mathbb{S}^{\ell}_-$ for any $y \in \mathbb{S}_-^{\ell}$.

\begin{example}[exponential-type augmented Lagrangian/modified barrier function] \label{ex:ExpTypeAugmLagr_SemiDef}
Let a function $\psi \colon \mathbb{R} \to \mathbb{R} \cup \{ + \infty \}$ be as in 
Example~\ref{ex:ExpTypeAugmLagr_2OrderCone}. For any $y, \lambda \in \mathbb{S}^{\ell}$ and $c > 0$ define
\[
  \Phi(y, \lambda, c) = \frac{1}{c} \big\langle \lambda, \Psi(c y) \big\rangle,
\]
if $c \sigma_1(y) < \varepsilon_0$, and $\Phi(y, \lambda, c) = + \infty$ otherwise. The corresponding augmented
Lagrangian $\mathscr{L}(\cdot)$ is an extension of augmented Lagrangians for inequality constrained 
optimization problems from  Examples~\ref{ex:ExpTypeAugmLagr} and \ref{ex:ModifiedBarrierFunction}. It was studied in
details in \cite{Stingl2006,Noll2007,LiZhang2009,ZhangLiWu2014,LuoWuLiu2015}.

Let $\Lambda = K^* = \mathbb{S}_+^{\ell}$. Then assumptions $(A1)$--$(A7)$, $(A9)$--$(A11)$, and $(A15)$ hold true in 
the general case (assumption $(A11)$ is satisfied with $\Phi_0(\lambda) \equiv \lambda$ by
\cite[Proposition~4.2]{LuoWuLiu2015}). Assumption $(A8)$ is satisfied, if the matrix function $\Psi(\cdot)$ is monotone
and convex (see, e.g. \cite{HornJohnson}). Assumptions $(A12)$ and $(A12)_s$ are satisfied if and only if 
$\varepsilon_0 < + \infty$. Restricted assumptions $(A13)$, $(A13)_s$, $(A14)$, and $(A14)_s$ hold true if and only if
$\psi(t)/t \to 0$ as $t \to - \infty$, while non-restricted versions of these assumptions hold true if and only if
$\psi$ is bounded below. Finally, assumption $(A9)_s$ is not satisfied for the function $\Phi$ from this example.
\end{example}

\begin{example}[penalized exponential-type augmented Lagrangian] \label{ex:PenalizedExpTypeAugmLagr_SemiDef}
Let a function $\psi \colon \mathbb{R} \to \mathbb{R}$ be a twice continuously differentiable non-decreasing convex 
function such that $\psi(t)/t \to + \infty$ as $t \to + \infty$, $\psi(0) = 0$ and $\psi'(0) = 1$. Let also 
$\xi \colon \mathbb{R} \to \mathbb{R}$ be a twice continuously differentiable non-decreasing convex function such that
$\xi(t) = 0$ for all $t \le 0$ and $\xi(t) > 0$ for all $t > 0$. Denote by $\Xi(\cdot)$ the L\"{o}wner's operator
associated with $\xi(\cdot)$. For any $y, \lambda \in \mathbb{S}^{\ell}$ and $c > 0$ define
\[
  \Phi(y, \lambda, c) = \frac{1}{c} \big\langle \lambda, \Psi(c y) \big\rangle
  + \frac{1}{c} \trace\Big( \Xi(c y) \Big).
\]
The corresponding augmented Lagrangian $\mathscr{L}(\cdot)$ was introduced in \cite{LuoWuLiu2015} and is an
extrension of the penalized exponential-type augmented Lagrangian from Example~\ref{ex:PenalizedExpTypeAugmLagr} to 
the case of nonlinear semidefinite programming problems.

Let $\Lambda = K^* = \mathbb{S}_+^{\ell}$. Then assumptions $(A1)$--$(A7)$, $(A9)$--$(A11)$, and $(A15)$ hold true 
in the general case (assumption $(A11)$ is satisfied with $\Phi_0(\lambda) \equiv \lambda$ by 
\cite[Proposition~4.2]{LuoWuLiu2015}). Assumption $(A8)$ is satisfied, provided the matrix functions $\Psi(\cdot)$ and 
$\Xi(\cdot)$ are monotone and convex. Assumptions $(A12)$ and $(A12)_s$ are satisfied, if $\xi(t) / t \to + \infty$ as 
$t \to +\infty$. Restricted assumptions $(A13)$, $(A13)_s$, $(A14)$, and $(A14)_s$ hold true if and only if 
$\psi(t)/t \to 0$ as $t \to - \infty$, while non-restricted versions of these assumptions hold true if and only if
$\psi$ is bounded below. Finally, assumption $(A9)_s$ is not satisfied, regardless of the choice of $\psi$ and $\xi$.
\end{example}

\subsection{Augmented Lagrangians for problems with pointwise inequality constraints}

Let $(T, \mathfrak{A}, \mu)$ be a measure space and $X$ be some space of functions $x \colon T \to \mathbb{R}^m$.
For example, $X$ can be defined as $L^p(T, \mathfrak{A}, \mu)$ or as the Sobolev space, when $T$ is an open subset
of $\mathbb{R}^d$. Let us consider problems with pointwise inequality constraints of the form:
\begin{equation} \label{prob:PointwiseConstr}
  \min\: f(x) \quad \text{subject to} \quad g(x(t), t) \le 0 \enspace \text{for a.e. } t \in T, \quad x \in Q,
\end{equation}
where $g \colon X \times T \to \mathbb{R}$ is a given function such that
$g(x(\cdot), \cdot) \in L^p(T, \mathfrak{A}, \mu)$ for some fixed $p \in [1, + \infty)$ and all $x \in X$.

One can rewrite problem \eqref{prob:PointwiseConstr} as the problem $(\mathcal{P})$ with 
$Y = L^p(T, \mathfrak{A}, \mu)$, $K$ being the cone of nonpositive function $L^p_-(T, \mathfrak{A}, \mu)$,
and $G(x)(\cdot) = g(x(\cdot), \cdot)$ for all $x \in X$. In this case the dual space $Y^*$ can be identified
with $L^q(T, \mathfrak{A}, \mu)$, where $q \in (1, + \infty]$ is the conjugate exponent of $p$, that is,
$1 / p + 1 / q = 1$ ($q = + \infty$, if $p = 1$). In turn, the polar cone $K^*$ can be identified with 
the cone of nonnegative functions $L^q_+(T, \mathfrak{A}, \mu)$.

For the sake of shortness we will consider only an augmented Lagrangian for problem \eqref{prob:PointwiseConstr} 
based on the Hestenes-Powell-Rockafellar augmented Lagrangian. However, it should be mentioned that one can 
define an augmented Lagrangian for this problem based on any other augmented Lagrangian for inequality 
constrained optimization problems.

\begin{example}[Hestenes-Powell-Rockafellar augmented Lagrangian] \label{ex:HPR_AugmLagr_Pointwise}
Suppose that either $p = 2$ or $p \ge 2$ and the measure $\mu$ is finite. For any 
$y \in Y := L^p(T, \mathfrak{A}, \mu)$, $\lambda \in Y^* \cong L^q(T, \mathfrak{A}, \mu)$, and $c > 0$ define
\[
  \Phi(y, \lambda, c) = \int_T \left( \lambda(t) \max\left\{ y(t), - \frac{\lambda(t)}{c} \right\}
  + \frac{c}{2} \max\left\{ y(t), - \frac{\lambda(t)}{c} \right\}^2 \right) d \mu(t).
\]
Observe that
\begin{multline*}
  \left| \lambda(t) \max\left\{ y(t), - \frac{\lambda(t)}{c} \right\}
  + \frac{c}{2} \max\left\{ y(t), - \frac{\lambda(t)}{c} \right\}^2 \right| 
  \\
  \le \frac{1 + c}{2} |y(t)|^2 + \left( \frac{1}{2} + \frac{1}{2c} \right) |\lambda(t)|^2.
\end{multline*}
Therefore, the value $\Phi(y, \lambda, c)$ is correctly defined and finite for any $y \in Y$, $\lambda \in Y^*$,
and $c > 0$, if $p = 2$ or $p \ge 2$ and the measure $\mu$ is finite.

Let $\Lambda = Y^*$. Then one can readily verify that all basic assumptions hold true in the general case, except 
for assumptions $(A12)$ and $(A12)_s$. Assumptions $(A12)$ and $(A12)_s$ are satisfied in the case $p = 2$, since
\[
  \Phi(y, \lambda, c) - \Phi(y, \lambda, c_0) \ge (c - c_0) \int_T \max\{ 0, y(t) \}^2 d \mu(t)
  = (c - c_0) \dist(y, K)^2
\]
for all $y \in Y$, $\lambda \in \Lambda$, and $c \ge c_0 > 0$ (see the proof of the validity of assumptions $(A12)$ and
$(A12)_s$ for the Rockafellar-Wets' augmented Lagrangian).
\end{example}

\subsection{Some comments on particular augmented Lagrangians}

Before we proceed to the analysis of the augmented dual problem and primal-dual augmented Lagrangian methods, let us
make a few general observations about the presented examples:
\begin{enumerate}
\item{All basic assumptions, except for assumption $(A9)_s$, are satisfied for all particular augmented Lagrangians 
presented above (under appropriate additional assumptions), except for the exponential-type augmented Lagrangian 
(Examples~\ref{ex:ExpTypeAugmLagr}, \ref{ex:ExpTypeAugmLagr_2OrderCone}, and \ref{ex:ExpTypeAugmLagr_SemiDef}), 
the p-th power augmented Lagrangian (Example~\ref{ex:pthPowerAugmLagr}), the hyperbolic-type augmented Lagrangian
(Example~\ref{eq:HyperbolicAugmLagr}), and He-Wu-Meng's augmented Lagrangian (Example~\ref{ex:HeWuMeng}). 
The exponential-type augmented Lagrangian (the case $\varepsilon_0 = + \infty$ in 
Examples~\ref{ex:ExpTypeAugmLagr_2OrderCone} and \ref{ex:ExpTypeAugmLagr_SemiDef}) and the p-th power augmented
Lagrangian do not satisfy assumptions $(A12)$ and $(A12)_s$, the hyperbolic-type augmented Lagrangian does not satisfy
assumptions $(A9)$, $(A12)$, and $(A12)_s$, while He-Wu-Meng's augmented Lagrangian does not satisfy assumption $(A9)$
and non-restricted versions of assumptions $(A13)$, $(A13)_s$, $(A14)$, and $(A14)_s$.
} 

\item{Assumption $(A9)_s$ is satisfied only for the Hestenes-Powell-Rockafellar augmented Lagrangian 
(Examples~\ref{ex:HestenesPowellAugmLagr}, \ref{ex:HPR_AugmLagr}, \ref{ex:HPR_AugmLagr_2OrderCone}, 
\ref{ex:HPR_AugmLagr_SemiDef}, and \ref{ex:HPR_AugmLagr_Pointwise}) and its generalization, Rockafellar-Wets' augmented
Lagrangian (Example~\ref{ex:RockafellarWetsAugmLagr}).
}

\item{For assumptions $(A1)$, $(A6)$--$(A8)$, $(A12)$--$(A15)$, and $(A12)_s$--$(A14)_s$ to be satisfied 
for the exponential-type augmented Lagrangian (Examples~\ref{ex:ExpTypeAugmLagr}, \ref{ex:ExpTypeAugmLagr_2OrderCone}, 
and \ref{ex:ExpTypeAugmLagr_SemiDef}), the penalized exponential-type augmented Lagrangian 
(Examples~\ref{ex:PenalizedExpTypeAugmLagr} and \ref{ex:PenalizedExpTypeAugmLagr_SemiDef}), the modified  barrier 
function (Example~\ref{ex:ModifiedBarrierFunction}), and the p-th power augmented Lagrangian 
(Example~\ref{ex:pthPowerAugmLagr}), and the hyperbolic-type augmented Lagrangian (Example~\ref{eq:HyperbolicAugmLagr})
it is \textit{necessary} that $\Lambda \subseteq K^*$. In contrast, for all other paritcular augmented Lagrangians
presented in this section these assumptions are satisfied for $\Lambda = Y^*$ (in the case of the He-Wu-Meng's augmented
Lagrangian only the restricted versions of assumptions $(A13)$, $(A13)_s$, $(A14)$, and $(A14)_s$ are satisfied for
$\Lambda = Y^*$). 
}

\item{Our theory of augmented Lagrangians encompasses penal\-ty functions of the form 
$F_c(\cdot) = f(\cdot) + c \dist(G(\cdot), K)^{\alpha}$ with $\alpha > 0$. One simply needs to define 
$\Phi(y, \lambda, c) := c \dist(y, K)^{\alpha}$. This function $\Phi$ satisfied assumptions $(A1)$, $(A2)$, $(A4)$,
$(A7)$, $(A9)$, $(A9)_s$, $(A10)$, $(A12)$--$(A15)$ and $(A12)_s$--$(A14)_s$ for any choice of the set $\Lambda$
(assumption $(A8)$ is satisfied, if $\alpha \ge 1$, while assumption $(A6)$ is satisfied, if $\Lambda \subseteq K^*$),
which means that the main results of this paper on the zero duality gap and convergence of augmented Lagrangian methods
can be applied to the penalty function $F_c(\cdot)$.
}
\end{enumerate}

\section{Duality theory}
\label{sect:DualityTheory}

One of the central concepts of the theory of augmented Lagrangians and corresponding optimization methods is 
\textit{the (augmented) dual problem}:
\[
  \max_{(\lambda, c)} \: \Theta(\lambda, c) \quad \text{subject to} \quad \lambda \in \Lambda, \enspace c > 0,
  \qquad \eqno{(\mathcal{D})}
\]
where 
\begin{equation} \label{eq:DualFunction}
  \Theta(\lambda, c) := \inf_{x \in Q} \mathscr{L}(x, \lambda, c) \quad \forall \lambda \in \Lambda, \: c > 0
\end{equation}
is \textit{the (augmented) dual function}. As is well-known and will be discussed in details below, convergence of 
augmented Lagrangian methods is interlinked with various properties of the dual problem. Therefore, before turning to
augmented Lagrangian methods, we need to analyse how standard duality results are translated into our axiomatic 
augmented Lagrangian setting.

\begin{remark} \label{rmrk:DualFunc_Concave_usc}
Note that if assumption $(A9)_s$ is satisfied, then the dual function $\Theta$ is concave and the augmented dual problem
$(\mathcal{D})$ is a concave optimization problem, even if the original problem $(\mathcal{P})$ is nonconvex. 
Furthermore, assumption $(A10)$ ensures that the dual function is upper semicontinuous, as the infimum of the family 
$\{ \mathscr{L}(x, \cdot) \}$, $x \in Q$, of upper semicontinuous functions.
\end{remark}

\subsection{Zero duality gap property}

Let us first study how optimal values of the problems $(\mathcal{P})$ and $(\mathcal{D})$ relate to each other. 
We start by showing that under an essentially nonrestrictive assumption the optimal value of the augmented dual problem
does not exceed the optimal value of the primal problem.

\begin{proposition}[weak duality] \label{prp:WeakDuality}
Let assumption $(A1)$ hold true. Then 
\[
  \Theta(\lambda, c) \le f(x) \quad \forall x \in \Omega, \: \lambda \in \Lambda, \: c > 0,
\]
where $\Omega$ is the feasible region of the problem $(\mathcal{P})$. In pacritular, $\Theta_* \le f_*$, where
$\Theta_*$ is the optimal value of the problem $(\mathcal{D})$ and $f_*$ is the optimal value of the problem
$(\mathcal{P})$.
\end{proposition}

\begin{proof}
By assumption $(A1)$ for any point $x \in \Omega$ and all $\lambda \in \Lambda$ and $c > 0$ one has 
$\mathscr{L}(x, \lambda, c) \le f(x)$. Hence applying the definition of $\Theta$ (see \eqref{eq:DualFunction}) and
the fact that $\Omega \subseteq Q$ one obtains the required result.
\end{proof}

As is well-known, the optimal values of the primal and dual problems might not coincide, especially for nonconvex 
problems. In this case $\Theta_* < f_*$ and the quantity $f_* - \Theta_* > 0$ is called \textit{the duality gap}. 

\begin{definition}
One says that there is \textit{zero duality gap} between the primal problem $(\mathcal{P})$ and the dual problem 
$(\mathcal{D})$ (or that the augmented Lagrangian $\mathscr{L}(\cdot)$ has \textit{the zero duality gap property}, or
that \textit{the strong duality} with respect to the augmented Lagrangian $\mathscr{L}(\cdot)$ holds true), if 
$\Theta_* = f_*$. 
\end{definition}

Our aim now is to understand what kind of assumptions one must impose on the function $\Phi$ to ensure that the
corresponding augmented Lagrangian $\mathscr{L}(x, \lambda, c) := f(x) + \Phi(G(x), \lambda, c)$ has the zero duality
gap property. To this end, we extend the standard result (see, e.g. \cite{RubinovHuangYang2002}) connecting the optimal
value of the dual problem with the behaviour of \textit{the optimal value (perturbation) function}
\[
  \beta(p) = \inf\big\{ f(x) \bigm| x \in Q \colon G(x) - p \in K \big\} \quad \forall p \in Y
\]
of the problem $(\mathcal{P})$ to our case. Denote by $\dom_{\lambda} \Theta$ the effective domain of $\Theta$
in $\lambda$, that is, 
$\dom_{\lambda} \Theta = \{ \lambda \in \Lambda \mid \exists c > 0 \colon \Theta(\lambda, c) > - \infty \}$.
Note that $\lambda \in \dom_{\lambda} \Theta$ if and only if the function $\mathscr{L}(\cdot, \lambda, c)$ is bounded
below on $Q$ for some $c > 0$. 

\begin{theorem}[optimal dual value formula] \label{thrm:DualOptVal_vs_OptValFunc}
Let assumptions $(A1)$, $(A7)$, and $(A12)$--$(A14)$ hold true. Then
\[
  \Theta_* := \sup_{\lambda \in \Lambda, c > 0} \Theta(\lambda, c) =
  \begin{cases}
    - \infty, & \text{if } \dom_{\lambda} \Theta = \emptyset,
    \\
    \min\left\{ f_*, \liminf_{p \to 0} \beta(p) \right\}, & \text{if } \dom_{\lambda} \Theta \ne \emptyset.
  \end{cases}
\]
In addition, $\Theta_* = \lim\limits_{c \to + \infty} \Theta(\lambda, c)$ for all $\lambda \in \dom_{\lambda} \Theta$.
\end{theorem}

\begin{proof}
Note that $\Theta(\lambda, c) = - \infty$ for all $\lambda \in \Lambda$ and $c > 0$, and $\Theta_* = - \infty$, if
$\dom_{\lambda} \Theta = \emptyset$. Therefore, below we can suppose that $\dom_{\lambda} \Theta \ne \emptyset$.

By assumption $(A7)$ the function $\Phi(y, \lambda, c)$ is non-decreasing in $c$. Therefore the functions 
$\mathscr{L}(x, \lambda, c)$ and $\Theta(\lambda, c)$ are non-decreasing in $c$ for all $x \in X$ and 
$\lambda \in \Lambda$. Hence, as is easy to see, one has
\[
  \sup_{\lambda \in \Lambda, c > 0} \Theta(\lambda, c) 
  = \sup_{\lambda \in \Lambda} \sup_{c > 0} \Theta(\lambda, c)
  = \sup_{\lambda \in \Lambda} \lim_{c \to +\infty} \Theta(\lambda, c)
  = \sup_{\lambda \in \dom_{\lambda} \Theta} \lim_{c \to +\infty} \Theta(\lambda, c).
\]
Consequently, it is sufficient to check that
\begin{equation} \label{eq:DualFuncLim_vs_PerturbFunc}
  \Theta_*(\lambda) := \lim\limits_{c \to + \infty} \Theta(\lambda, c) 
  = \min\left\{ f_*, \liminf_{p \to 0} \beta(p) \right\} 
  \quad \forall \lambda \in \dom_{\lambda} \Theta.
\end{equation}
Let us prove this equality.

Fix any $\lambda \in \dom_{\lambda} \Theta$ and any unbounded strictly increasing sequence 
$\{ c_n \} \subset (0, + \infty)$ such that the function $\mathscr{L}(\cdot, \lambda, c_0)$ is bounded below on $Q$
(such $c_0$ exists by the definition of $\dom_{\lambda} \Theta$). Then by Proposition~\ref{prp:WeakDuality} one has
\begin{equation} \label{eq:DualFuncIncreasLim}
\begin{split}
  f_* \ge \Theta_*(\lambda) \ge &\Theta(\lambda, c_n) \ge \Theta(\lambda, c_0) 
  > - \infty \quad \forall n \in \mathbb{N}, \quad 
  \\ 
  &\lim_{n \to \infty} \Theta(\lambda, c_n) = \Theta_*(\lambda).
\end{split}
\end{equation}
Let $\{ x_n \} \subset Q$ be a sequence such that 
$\mathscr{L}(x_n, \lambda, c_n) \le \Theta(\lambda, c_n) + 1/(n + 1)$ for all $n \in \mathbb{N}$. Observe that from
\eqref{eq:DualFuncIncreasLim} it follows that 
\begin{equation} \label{eq:LimMinimizersDualFunc}
  \lim_{n \to \infty} \mathscr{L}(x_n, \lambda, c_n) = \Theta_*(\lambda) \le f_*.
\end{equation}
Note also that due to assumption $(A7)$ for all $r > 0$, $n \in \mathbb{N}$, and $x \in Q$ such that 
$\dist(G(x), K) \ge r$ one has $\mathscr{L}(x, \lambda, c_n) = + \infty$, if $\Phi(G(x), \lambda, c_0) = + \infty$, and
\begin{multline*}
  \mathscr{L}(x, \lambda, c_n) = \mathscr{L}(x, \lambda, c_0) + \Phi(G(x), \lambda, c_n) - \Phi(G(x), \lambda, c_0)
  \ge \Theta(\lambda, c_0) 
  \\
  + \inf\Big\{ \Phi(y, \lambda, c_n) - \Phi(y, \lambda, c_0) 
  \Bigm| y \in Y, \: \dist(y, K) \ge r, \: |\Phi(y, \lambda, c_0)| < + \infty \Big\},
\end{multline*}
if $\Phi(G(x), \lambda, c_0) < + \infty$ (note that $\Phi(G(x), \lambda, c_0) > - \infty$, since 
$\lambda \in \dom_{\lambda} \Theta$). Therefore, by assumption $(A12)$ for any $r > 0$ one has 
$\mathscr{L}(x, \lambda, c_n) \to + \infty$ as $n \to \infty$ uniformly on the set 
$\{ x \in Q \mid \dist(G(x), K) \ge r \}$, which with the use of \eqref{eq:LimMinimizersDualFunc} implies that
$\dist(G(x_n), K) \to 0$ as $n \to \infty$. Let us consider two cases.

\textbf{Case I.} Suppose that there exists a subsequence $\{ x_{n_k} \}$ that is feasible for the problem
$(\mathcal{P})$. Then with the use of Lemma~\ref{lem:Assumpt(A16)} one gets
\begin{align*}
  \lim_{k \to \infty} \mathscr{L}(x_{n_k}, \lambda, c_{n_k}) 
  &\ge \liminf_{k \to \infty} \Big( f(x_{n_k}) + \inf_{y \in K} \Phi(y, \lambda, c_{n_k}) \Big) 
  \\
  &\ge \liminf_{k \to \infty} f(x_{n_k}) \ge f_*.
\end{align*}

\textbf{Case II.} Suppose now that $G(x_{n_k}) \notin K$ for some subsequence $\{ x_{n_k} \}$ Then with the use of
assumption $(A13)$ one gets
\[
  \lim_{k \to \infty} \mathscr{L}(x_{n_k}, \lambda, c_{n_k}) \ge 
  \liminf_{k \to \infty} f(x_{n_k}) \ge \liminf_{k \to \infty} \beta(p_k) \ge \liminf_{p \to 0} \beta(p),
\]
where $\{ p_k \} \subset Y$ is any sequence such that $G(x_{n_k}) - p_k \in K$ and $\| p_k \| \to 0$ as $k \to \infty$ 
(note that such sequence exists, since $\dist(G(x_n), K) \to 0$ as $n \to \infty$).

Combining the two cases and inequalities \eqref{eq:DualFuncIncreasLim} and \eqref{eq:LimMinimizersDualFunc} one obtains
that
\begin{equation} \label{eq:OptValueLowerEstimate}
  f_* \ge \Theta_*(\lambda) \ge \min\{ f_*, \liminf_{p \to 0} \beta(p) \}.
\end{equation}
To prove equality \eqref{eq:DualFuncLim_vs_PerturbFunc}, suppose that $f_* > \liminf_{p \to 0} \beta(p) =: \beta_*$. 

Let $\{ p_k \} \subset Y$ be any sequence such that $p_k \to 0$ and $\beta(p_k) \to \beta_*$ as $k \to \infty$. Let
also $\{ t_n \}$ be the sequence from assumption $(A14)$. Clearly, there exists a subsequence $\{ p_{k_n} \}$ such that
$\| p_{k_n} \| \le t_n$ for all $n \in \mathbb{N}$. By the definition of the optimal value function $\beta$ for any 
$n \in \mathbb{N}$ one can find $x_n \in Q$ such that $G(x_n) - p_{k_n} \in K$ (i.e. $\dist(G(x_n), K) \le t_n$) and 
$f(x_n) \le \beta(p_{k_n}) + 1/(n + 1)$ in the case when $\beta_* > - \infty$, and $f(x_n) \to - \infty$ as 
$n \to \infty$ in the case when $\beta_* = - \infty$.

If $\beta_* > - \infty$, then thanks to assumption $(A14)$ one has
\begin{align*}
  \Theta_*(\lambda) = \lim_{n \to \infty} \Theta(\lambda, c_n) \le \lim_{n \to \infty} \mathscr{L}(x_n, \lambda, c_n)
  = \lim_{n \to \infty} f(x_n) &= \lim_{n \to \infty} \beta(p_{k_n}) 
  \\
  &= \liminf_{p \to 0} \beta(p),
\end{align*}
which along with \eqref{eq:OptValueLowerEstimate} implies the required result. In turn, if $\beta_* = - \infty$, then 
due to assumption $(A14)$ one has
\[
  \Theta_*(\lambda) = \lim_{n \to \infty} \Theta(\lambda, c_n) \le \lim_{n \to \infty} \mathscr{L}(x_n, \lambda, c_n)
  = \lim_{n \to \infty} f(x_n) = - \infty = \liminf_{p \to 0} \beta(p),
\]
which implies that equality \eqref{eq:DualFuncLim_vs_PerturbFunc} holds true.
\end{proof}

\begin{corollary}[duality gap formula] \label{crlr:DualityGapFormula}
If under the assumptions of the previous theorem one has $\dom_{\lambda} \Theta \ne \emptyset$, then
\[
  f_* - \Theta_* = \max\left\{ 0, f_* - \liminf_{p \to 0} \beta(p) \right\}.
\]
In particular, if the duality gap is positive, then it is equal to $f_* - \liminf_{p \to 0} \beta(p)$.
\end{corollary}

\begin{remark} \label{rmrk:OptValFuncFiniteLim}
Although in the proof of Theorem~\ref{thrm:DualOptVal_vs_OptValFunc} we considered the case $\beta_* = - \infty$, in
actuality, the assumptions of this theorem ensure that $\beta_* > - \infty$. Namely, if assumptions $(A7)$ and $(A14)$
are satisfied and $\mathscr{L}(\cdot, \lambda, c)$ is bounded below on $Q$ for some $\lambda \in \Lambda$ and $c > 0$,
then $\beta_* = \liminf_{p \to 0} \beta(p) > - \infty$. Indeed, suppose by contradiction that $\beta_* = - \infty$.
Then there exists a sequence $\{ p_n \} \in Y$ such that $p_n \to 0$ and $\beta(p_n) \to - \infty$ as $n \to \infty$.
By the definition of the optimal value function one can find a sequence $\{ x_n \} \subset Q$ such that
$G(x_n) - p_n \in K$ for all $n \in \mathbb{N}$ and $f(x_n) \to - \infty$ as $n \to \infty$. Note that 
$\dist(G(x_n), K) \to 0$ as $n \to \infty$, since $p_n$ converges to zero. 

Let $\{ c_n \} \subset (c, + \infty)$ be any increasing unbounded sequence and $\{ t_k \}$ be the sequence 
from assumption $(A14)$. Clearly, one can find a subsequence $\{ x_{n_k} \}$ such that 
$\dist(G(x_{n_k}), K) \le t_k$ for all $k \in \mathbb{N}$. Then by assumption $(A14)$ one has
$\Phi(G(x_{n_k}), \lambda, c_{n_k}) \to 0$ as $k \to \infty$, which implies that 
\[
  \lim_{k \to \infty} \mathscr{L}(x_{n_k}, \lambda, c_{n_k}) 
  = \lim_{k \to \infty} f(x_{n_k}) = - \infty,
\]
which due to assumption $(A7)$ contradicts the fact that $\mathscr{L}(\cdot, \lambda, c)$ is bounded below on $Q$.
\end{remark}

\begin{remark} \label{rmrk:ConstraintsBoundedBelow}
The claim of Theorem~\ref{thrm:DualOptVal_vs_OptValFunc} remains to hold true, if only restricted versions of
assumptions $(A13)$ and $(A14)$ hold true, and one additionally assumes that the projection of the set $G(Q)$ onto the
cone $K$ is bounded. If this projection is bounded, then one can show that the sequences $\{ G(x_{n_k}) \}$ appearing in
the proof of the theorem are also bounded. Therefore, only restricted versions of assumptions $(A13)$ and $(A14)$ are
needed to prove the theorem in this case, which makes the theorem applicable, for example, to He-Wu-Meng's augmented
Lagrangian (Example~\ref{ex:HeWuMeng}). 

Let us note that the assumption on the boundedness of the projection of $G(Q)$ onto $K$ is not uncommon in 
the literature on augmented Lagrangians and primal-dual augmented Lagrangian methods (see, e.g.
\cite[Assumption~2]{LuoSunLi}, \cite[Assumption~2]{LiuYang2008}, \cite[Assumption~2]{WangLi2009},
\cite[condition~(2)]{WangLiuQu}, etc.). In many particular cases this assumption is not restrictive from the theoretical
point of view. For example, one can always guarantee that this assumption is satisfied for inequality constrainted
problems by replacing the constraints $g_i(x) \le 0$ with $e^{g_i(x)} - 1 \le 0$.
\end{remark}

As a simple corollary to Theorem~\ref{thrm:DualOptVal_vs_OptValFunc} we can obtain necessary and sufficient conditions 
for the augmented Lagrangian $\mathscr{L}(x, \lambda, c)$ to have the zero duality gap property.

\begin{theorem}[zero duality gap characterisation] \label{thrm:ZeroDualityGapCharacterization}
Let assumptions $(A1)$, $(A7)$, and $(A12)$--$(A14)$ be valid. Then the zero duality gap property holds true if and only
if the optimal value function $\beta$ is lower semicontinuous (lsc) at the origin and there exist $\lambda \in \Lambda$
and $c > 0$ such that the function $\mathscr{L}(\cdot, \lambda, c)$ is bounded below on $Q$.
\end{theorem}

\begin{proof}
Suppose that the zero duality gap property holds true. Then the optimal value of the dual problem is finite, which
implies that $\dom_{\lambda} \Theta \ne \emptyset$ or, equivalently, there exist $\lambda \in \Lambda$ and $c > 0$
such that the function $\mathscr{L}(\cdot, \lambda, c)$ is bounded below on $Q$. Moreover, 
$f_* = \Theta_* = \min\{ f_*, \liminf_{p \to 0} \beta(p) \}$ by Theorem~\ref{thrm:DualOptVal_vs_OptValFunc}, which
means that $\liminf_{p \to 0} \beta(p) \ge f_* = \beta(0)$, that is, the optimal value function $\beta$ is lsc 
at the origin.

Conversely, suppose that $\beta$ is lsc at the origin and $\dom_{\lambda} \Theta \ne \emptyset$. Then by
Theorem~\ref{thrm:DualOptVal_vs_OptValFunc} one has $\Theta_* = f_*$, that is, there is zero duality gap between 
the primal and dual problems.
\end{proof}

\begin{remark}
Let us note that one can prove the zero duality gap property for $\mathscr{L}(\cdot)$ under slightly less restrictive
assumptions on the function $\Phi$ than in the previous theorems. Namely, instead of assuming that the claims of
assumptions $(A12)$--$(A14)$, $(A16)$ are satisfied \textit{for all} $\lambda \in \Lambda$, it is sufficient to suppose
that \textit{there exists} $\lambda_0 \in \dom_{\lambda} \Theta$ satisfying these assumptions. Then arguing in the
same way as in the proof of Theorem~\ref{thrm:DualOptVal_vs_OptValFunc} one can check that
\[
  f_* \ge \sup_{\lambda \in \Lambda, c > 0} \Theta(\lambda, c) \ge \lim_{c \to + \infty} \Theta(\lambda_0, c) 
  = \min\big\{ f_*, \liminf_{p \to 0} \beta(p) \big\}.
\]
This inequality obviously implies that the zero duality gap property holds true, provided the optimal value function
$\beta$ is lsc at the origin. Although such small change in the assumptions of the theorem might seem insignificant, in
actuality it considerably broadens the class of augmented Lagrangians to which the sufficient conditions for the
validity of the zero duality gap property can be applied. For example, Theorems~\ref{thrm:DualOptVal_vs_OptValFunc} and
\ref{thrm:ZeroDualityGapCharacterization} are inapplicable to the exponential-type augmented Lagrangian
(Example~\ref{ex:ExpTypeAugmLagr}), since this augmented Lagrangian does not satisfy assumption $(A12)$. However, it
satisfies the claim of assumption $(A12)$ for any $\lambda \in \mathbb{R}_+^m$ that lies in the interior of
$\mathbb{R}_+^m$ (i.e. that does not have zero components) and, therefore, one can conclude that the zero duality gap
property holds true for the exponential-type augmented Lagrangian, provided the optimal value function is lsc at the
origin and there exists $\lambda_0 \in \dom_{\lambda} \Theta \cap \interior \mathbb{R}_+^m$.
\end{remark}

\begin{remark}
Theorem~\ref{thrm:ZeroDualityGapCharacterization} implies that under suitable assumptions the zero duality gap property
depends not on the properties of the augmented Lagrangian $\mathscr{L}(\cdot)$, but rather properties of the
optimization problem $(\mathcal{P})$ itself. Similarly, by Corollary~\ref{crlr:DualityGapFormula} the duality gap 
$f_* - \Theta_*$ does not depend on the augmented Lagrangian or even some characteristic of the dual problem 
$(\mathcal{D})$. It is completely predefined by the properties of the optimization problem under consideration. 
Thus, in a sense, the absence of the duality gap between the primal and dual problems, as well as the size
of the duality gap, when it is positive, are properties of optimization problems themselves, not augmented Lagrangians 
or augmented dual problems that are used for analysing and/or solving these problems.
\end{remark}

For the sake of completeness, let us also present a simple characterisation of the lower semicontinuity of the
optimal value function $\beta$, from which one can easily derive a number of well-known sufficient conditions for this
function to be lsc at the origin.

\begin{proposition} \label{prp:OptValFunc_LSC}
For the optimal value function $\beta$ to be lsc at the origin it is necessary and sufficient that there does not
exist a sequence $\{ x_n \} \subset Q$ such that $\dist(G(x_n), K) \to 0$ as $n \to \infty$ and 
$\liminf_{n \to \infty} f(x_n) < f_*$.
\end{proposition}

\begin{proof}
\textbf{Necessity.} Suppose that $\beta$ is lsc at the origin. Let $\{ x_n \} \subset Q$ be any sequence such that 
$\dist(G(x_n), K) \to 0$ as $n \to \infty$. Denote $p_n = G(x_n)$. Then $p_n \to 0$ as $n \to \infty$ and due to 
the lower semicontinuity of $\beta$ at the origin one has
\[
  \liminf_{n \to \infty} f(x_n) \ge \liminf_{n \to \infty} \beta(p_n) \ge \beta(0) = f_*.
\]
In other words, there does not exist a sequence $\{ x_n \} \subset Q$ satisfying the conditions from the formulation 
of the proposition.

\textbf{Sufficiency.} Suppose by contradiction that the function $\beta$ is not lsc at the origin. Then there exist 
$\varepsilon > 0$ and a sequence $\{ p_n \} \subset Y$ converging to zero and such that 
$\beta(p_n) \le \beta(0) - \varepsilon$ for all $n \in \mathbb{N}$. By the definition of the optimal value function for
any $n \in \mathbb{N}$ one can find $x_n \in Q$ such that $G(x_n) \in K + p_n$ and $f(x_n) \le f_* - \varepsilon / 2$
(recall that $\beta(0) = f_*$). Therefore $\dist(G(x_n), K) \to 0$ as $n \to \infty$ and 
$\liminf_{n \to \infty} f(x_n) < f_*$, which contradicts the assumptions of the proposition that there does not exist 
a sequence $\{ x_n \} \subset Q$ satisfying these conditions.
\end{proof}

\begin{corollary} \label{crlr:OptValFunc_LSC_ReflexiveCase}
Let the space $X$ be reflexive, the set $Q$ be weakly sequentially closed (in particular, one can suppose that $Q$ is
convex), and the functions $f$ and $\dist(G(\cdot), K)$ be weakly sequentially lsc on $Q$. Then for the optimal value
function $\beta$ to be lsc at the origin it is necessary and sufficient there does not exist a sequence
$\{ x_n \} \subset Q$ such that $\| x_n \| \to + \infty$ and $\dist(G(x_n), K) \to 0$ as $n \to \infty$, and 
$\liminf_{n \to \infty} f(x_n) < f_*$.
\end{corollary}

\begin{proof}
The necessity of the conditions from the formulation of the corollary for the lower semicontinuity of the function
$\beta$ follows directly from the previous proposition. Let us prove that they are also sufficient for the lower
semicontinuity of $\beta$.

Taking into account Proposition~\ref{prp:OptValFunc_LSC} it is sufficient to prove that there does not exists a bounded
sequence $\{ x_n \} \subset Q$ such that $\dist(G(x_n), K) \to 0$ as $n \to \infty$ and 
$\liminf_{n \to \infty} f(x_n) < f_*$. Suppose by contradiction that such bounded sequence exists. Replacing this
sequence with a subsequence, if necessary, one can assume that the sequence $\{ f(x_n) \}$ converges. Since the space
$X$ is reflexive, one can extract a subsequence $\{ x_{n_k} \}$ that weakly converges to some point $x_*$ that belongs
to the set $Q$, since this set is weakly sequentially closed. Furthermore, $G(x_*) \in K$, i.e. $x_*$ is feasible for
the problem $(\mathcal{P})$, since $\dist(G(x_n), K) \to 0$ as $n \to \infty$ and the function $\dist(G(\cdot), K)$ is
weakly sequentially lsc. Hence taking into account the fact that $f$ is also weakly sequentially lsc one gets that
\[
  f_* > \lim_{n \to \infty} f(x_n) = \lim_{k \to \infty} f(x_{n_k}) \ge f(x_*),
\]
which is impossible by virtue of the fact that the point $x_*$ is feasible for the problem $(\mathcal{P})$.
\end{proof}

\begin{corollary} \label{crlr:OptValFunc_LSC_SuffCond}
Let the assumptions of the previous corollary be satisfied and one of the following conditions hold true:
\begin{enumerate}
\item{the set $Q$ is bounded;}

\item{the function $f$ is coercive on $Q$, that is, for any sequence $\{ x_n \} \subset Q$ such that 
$\| x_n \| \to + \infty$ as $n \to \infty$ one has $f(x_n) \to + \infty$ as $n \to \infty$;}

\item{the function $\dist(G(\cdot), K)$ is coercive on $Q$;}

\item{the penalty function $f(\cdot) + c \dist(G(\cdot), K)^{\alpha}$ is coercive on $Q$ for some $c > 0$ and 
$\alpha > 0$.}
\end{enumerate}
Then the optimal value function $\beta$ is lsc at the origin.
\end{corollary}

Thus, by Theorem~\ref{thrm:ZeroDualityGapCharacterization} and Corollary~\ref{crlr:OptValFunc_LSC_ReflexiveCase}, in 
the case when the space $X$ reflexive (in particular, in the finite dimensional case), under some natural lower
semicontinuity assumptions the duality gap between the problems $(\mathcal{P})$ and $(\mathcal{D})$ is positive if and
only if there exists an unbounded sequence $\{ x_n \} \subset Q$ such that $\dist(G(x_n), K) \to 0$ as $n \to \infty$
and the lower limit of the sequence $\{ f(x_n) \}$ is smaller than the optimal value of the problem $(\mathcal{P})$.
Furthermore, one can verify that the infimum of all such lower limits is equal to $\liminf_{p \to 0} \beta(p)$ and,
therefore, defines the value of the duality gap $f_* - \Theta_*$. 

\begin{remark}
Many existing results on the zero duality gap property for various augmented Lagrangians are either particular
cases of Theorems~\ref{thrm:DualOptVal_vs_OptValFunc} and \ref{thrm:ZeroDualityGapCharacterization} combined with
Corollaries~\ref{crlr:OptValFunc_LSC_ReflexiveCase} and \ref{crlr:OptValFunc_LSC_SuffCond} or can be easily derived
directly from these theorems and corollaries, including \cite[Theorem~4.1]{HuangYang2003},
\cite[Theorem~2.2]{HuangYang2005}, \cite[Theorem~3]{BurachikGasimov}, \cite[Theorem~2.1]{LiuYang2008},
\cite[Theorem~4.1]{ZhouYang2009}, \cite[Theorem~2.1]{ZhouYang2012}, the claims about the zero duality gap property in
Theorems~7, 9, and 11 in \cite{YalchinKasimbeyli}, etc.
\end{remark}

\subsection{Optimal dual solutions and global saddle points}

As we will show below, dual convergence of augmented Lagrangian methods (that is, the convergence of the sequence of
multipliers) is directly connected with the existence of globally optimal solutions of the dual problem $(\mathcal{D})$.
Therefore, let us analyse main properties of optimal dual solutions that will help us to better understand dual 
convergence of augmented Lagrangian methods. For the sake of shortness, below we use the term 
\textit{optimal dual solution}, instead of globally optimal solution of the dual problem $(\mathcal{D})$. 

First, we make a simple observation about the role of the penalty parameter $c$ in optimal dual solutions.

\begin{proposition} \label{prp:DualSolution_2ndComponent}
Let assumption $(A7)$ hold true and $(\lambda_*, c_*)$ be an optimal dual solution. Then for any $c \ge c_*$ 
the equality $\Theta(\lambda_*, c) = \Theta(\lambda_*, c_*)$ holds true and the pair $(\lambda_*, c)$ is also 
an optimal dual solution.
\end{proposition}

\begin{proof}
Under the assumption $(A7)$ the augmented Lagrangian $\mathscr{L}(x, \lambda, c)$ is non-decreasing in $c$, which
obviously implies that the augmented dual function $\Theta(\lambda, c) = \inf_{x \in Q} \mathscr{L}(x, \lambda, c)$ is
non-decreasing in $c$ as well. Therefore, for any $c \ge c_*$ one has $\Theta(\lambda_*, c) \ge \Theta(\lambda, c_*)$,
which means that $(\lambda_*, c)$ is a globally optimal solution of the dual problem and 
$\Theta(\lambda_*, c) = \Theta(\lambda_*, c_*)$.
\end{proof}

With the use of the previous proposition we can describe the structure of the set of optimal dual solutions, which we
denote by $\mathscr{D}_*$. Define
\[
  c_*(\lambda) = \inf\big\{ c > 0 \bigm| (\lambda, c) \in \mathscr{D}_* \big\} \quad \forall \lambda \in \Lambda. 
\]
Note that by definition $c_*(\lambda) = + \infty$, if $(\lambda, c) \notin \mathscr{D}_*$ for any $c > 0$. In addition,
if $c_*(\lambda) < + \infty$ and assumption $(A7)$ holds true, then according to the previous proposition
$(\lambda, c) \in \mathscr{D}_*$ for any $c > c_*(\lambda)$. Following the work of Burachik et al.
\cite{BurachikIusemMelo2015}, we call the function $c_*(\cdot)$ \textit{the penalty map}. 

Let us prove some properties of the penalty map and describe the structure of the set of optimal dual solutions with 
the use of this map.

\begin{corollary} \label{crlr:OptDualSolSet_ExactPenaltyMap}
Let assumptions $(A7)$, $(A9)$, and $(A10)$ be valid and the dual problem $(\mathcal{D})$ have a globally optimal
solution. Then the set $\dom c_*(\cdot)$ is convex, the penalty map $c_*(\cdot)$ is a quasiconvex function and
\begin{align*}
  \mathscr{D}_* &= \Big\{ \{ \lambda \} \times [c_*(\lambda), + \infty) \Bigm| 
  \lambda \in \dom c_*(\cdot) \colon c_*(\lambda) > 0 \Big\}
  \\
  &\bigcup \Big\{ \{ \lambda \} \times (0, + \infty) \Bigm| 
  \lambda \in \dom c_*(\cdot) \colon c_*(\lambda) = 0 \Big\}.
\end{align*}
\end{corollary}

\begin{proof}
Let us first show that the set $\dom c_*(\cdot)$ is convex. Indeed, choose any 
$\lambda_1, \lambda_2 \in \dom c_*(\cdot)$, $\alpha \in [0, 1]$, and $c > \max\{ c_*(\lambda_1), c_*(\lambda_2) \}$.
Then by Propositon~\ref{prp:DualSolution_2ndComponent} both $(\lambda_1, c) \in \mathscr{D}_*$ and 
$(\lambda_2, c) \in \mathscr{D}_*$. Hence taking into account the fact that the dual function $\Theta(\lambda, c)$ is
concave in $\lambda$ by assumption $(A9)$ one gets
\[
  \Theta(\alpha \lambda_1 + (1 - \alpha) \lambda_2, c) \ge
  \alpha \Theta(\lambda_1, c) + (1 - \alpha) \Theta(\lambda_2, c) = \alpha \Theta_* + (1 - \alpha) \Theta_*
  = \Theta_*.
\]
Therefore $(\alpha \lambda_1 + (1 - \alpha) \lambda_2, c) \in \mathscr{D}_*$, which means that
$\alpha \lambda_1 + (1 - \alpha) \lambda_2 \in \dom c_*(\cdot)$ and the set $\dom c_*(\cdot)$ is convex. Moreover,
since $c > \max\{ c_*(\lambda_1), c_*(\lambda_2) \}$ was chosen aribtrarily, one has
$c_*(\alpha \lambda_1 + (1 - \alpha) \lambda_2) \le \max\{ c_*(\lambda_1), c_*(\lambda_2) \}$, that is,
the penalty map is a quasiconvex function.

As was noted above, for any $\lambda \in \dom c_*(\cdot)$ and $c > c_*(\lambda)$ one has 
$(\lambda, c) \in \mathscr{D}_*$. Hence bearing in mind the fact that the augmented dual function $\Theta$ is
upper semicontinuous (usc) by assumption $(A10)$ one can conclude that 
\begin{align*}
  \{ \lambda \} \times [c_*(\lambda), + \infty) &\subseteq \mathscr{D}_*
  \quad \forall \lambda \in \dom c_*(\cdot) \colon c_*(\lambda) > 0,
  \\
  \{ \lambda \} \times (0, + \infty) &\subseteq \mathscr{D}_*
  \quad \forall \lambda \in \dom c_*(\cdot) \colon c_*(\lambda) = 0.
\end{align*}
The validity of the converse inclusions follows directly from the definition of the penalty map and
Proposition~\ref{prp:DualSolution_2ndComponent}.
\end{proof}

\begin{remark}
Note that we need to consider the case $c_*(\lambda) = 0$ separately due to the fact that many particular augmented
Lagrangians are not defined for $c = 0$ (see examples in Section~\ref{sect:Examples}). However, if a given augmented
Lagrangian is correctly defined for $c = 0$ and assumptions $(A7)$, $(A9)$, and $(A10)$ are satisfied for 
$c \in [0, + \infty)$, then
\[
  \mathscr{D}_* = \bigcup_{\lambda_* \in \dom c_*(\cdot)} \{ \lambda_* \} \cup [c_*(\lambda_*), + \infty)
  = \epigraph c_*(\cdot),
\]
where $\epigraph c_*(\cdot)$ is the epigraph of the penalty map. This equality holds true, in particular, for 
Rockafellar-Wets' augmented Lagrangian from Example~\ref{ex:RockafellarWetsAugmLagr}. Let us also note that in the case
when assumption $(A9)_s$ is satisfied (e.g. in the case of Rockafellar-Wets' augmented Lagrangian) the penalty map is
convex, since in this case the dual function $\Theta$ is concave and, therefore, the set of optimal dual solutions
$\mathscr{D}_*$ is convex.
\end{remark}

Optimal dual solutions can be described in terms of global saddle points of the augmented Lagrangian.

\begin{definition}
A pair $(x_*, \lambda_*) \in Q \times \Lambda$ is called a global saddle point of the augmented Lagrangian 
$\mathscr{L}(\cdot)$, if there exists $c_* > 0$ such that
\begin{equation} \label{eq:GSP_def}
  \sup_{\lambda \in \Lambda} \mathscr{L}(x_*, \lambda, c) \le \mathscr{L}(x_*, \lambda_*, c) 
  \le \inf_{x \in Q} \mathscr{L}(x, \lambda_*, c) < + \infty
  \quad \forall c \ge c_*.
\end{equation}
The infimum of all such $c_*$ is denoted by $c_*(x_*, \lambda_*)$ and is call 
\textit{the least exact penalty parameter} for the global saddle point $(x_*, \lambda_*)$.
\end{definition}

\begin{remark}
It is worth noting that inequalities \eqref{eq:GSP_def} from the definition of global saddle point are obviously
satisfied as (and, therefore, can be replaced with) equalities.
\end{remark}

The following theorem, that combines together several well-known results (cf. \cite[Theorem~11.59]{RockafellarWets},
\cite[Theorem~2.1, part~(v)]{ShapiroSun}, \cite[Theorem~2.1]{ZhouZhouYang2014}, etc.), shows how optimal dual solutions
are interconnected with global saddle points. We present a complete proof of this theorem for the sake of completeness
and due to the fact that, to the best of the author's knowledge, all three claims of this theorem cannot be derived from
existing results within our axiomatic augmented Lagrangian setting.

\begin{theorem} \label{thrm:OptDualSol_vs_GlobalSaddlePoints}
Let assumptions $(A1)$--$(A3)$ and $(A7)$ be valid. Then the following statements hold true:
\begin{enumerate}
\item{if a global saddle point $(x_*, \lambda_*)$ of $\mathscr{L}(\cdot)$ exists, then the zero duality gap property
holds true, $x_*$ is a globally optimal solution of the problem $(\mathcal{P})$, and for any $c_* > c_*(x_*, \lambda_*)$
the pair $(\lambda_*, c_*)$ is an optimal dual solution;
\label{st:GSP_DualSol}
}

\item{if $(\lambda_*, c_*)$ is an optimal dual solution and the zero duality gap property holds true, then for any
globally optimal solution $x_*$ of the problem $(\mathcal{P})$ the pair $(x_*, \lambda_*)$ is a global saddle point of
$\mathscr{L}(\cdot)$ and $c_*(x_*, \lambda_*) \le c_*$;
\label{st:DualSol_GSP}
}

\item{if $(x_*, \lambda_*)$ is a global saddle point of $\mathscr{L}(\cdot)$, then
\[  
  f_* = \mathscr{L}(x_*, \lambda_*, c) = \inf_{x \in Q} \sup_{\lambda \in \Lambda} \mathscr{L}(x, \lambda, c) 
  = \sup_{\lambda \in \Lambda} \inf_{x \in Q} \mathscr{L}(x, \lambda, c) = \Theta_*
\]
for all $c > c_*(x_*, \lambda_*)$.
\label{st:MinMaxEqual}
}
\end{enumerate}
\end{theorem}

\begin{proof}
\textbf{Part~\ref{st:GSP_DualSol}.} Let $(x_*, \lambda_*)$ be a global saddle point of $\mathscr{L}(\cdot)$ and
some $c > c_*(x_*, \lambda_*)$ be fixed. Let us first show that $x_*$ is feasible. Indeed, assume that 
$G(x_*) \notin K$. Then by assumption $(A3)$ there exists a multiplier $\lambda_0 \in \Lambda$ such that 
$\Phi(G(x_*), t \lambda_0, c) \to + \infty$ as $t \to + \infty$, which contradicts the inequalities
\[
  \mathscr{L}(x_*, t \lambda_0, c) \le \sup_{\lambda \in \Lambda} \mathscr{L}(x_*, \lambda, c) 
  \le \mathscr{L}(x_*, \lambda_*, c) < + \infty
  \quad \forall c > c_*(x_*, \lambda_*)
\]
that follow from the definition of the global saddle point. Thus, $G(x_*) \in K$, that is, $x_*$ is feasible. 

By assumption $(A2)$ there exists $\widehat{\lambda} \in \Lambda$ such that 
$\Phi(G(x_*), \widehat{\lambda}, c) \ge 0$, which by the definition of global saddle point implies that
\[
  f(x_*) \le \mathscr{L}(x_*, \widehat{\lambda}, c) \le \sup_{\lambda \in \Lambda} \mathscr{L}(x_*, \lambda, c)
  \le \mathscr{L}(x_*, \lambda_*, c) \le f(x_*)
  \quad \forall c > c_*(x_*, \lambda_*),
\]
where the last inequality is valid by assumption $(A1)$. Consequently, one has
\[
  \mathscr{L}(x_*, \lambda_*, c) = f(x_*), \quad \Phi(G(x_*), \lambda_*, c) = 0 \quad \forall c > c_*(x_*, \lambda_*).
\]
Hence applying the definition of global saddle point and assumption $(A1)$ once more one gets that
\[
  f(x_*) = \mathscr{L}(x_*, \lambda_*, c) \le \inf_{x \in Q} \mathscr{L}(x, \lambda_*, c)
  \le \inf_{x \in Q \colon G(x) \in K} \mathscr{L}(x, \lambda_*, c) \le f(x)
\]
for any feasible $x$, which means that $x_*$ is a globally optimal solution of the problem $(\mathcal{P})$. 
Furthermore, one has
\[
  \Theta(\lambda_*, c) = \inf_{x \in Q} \mathscr{L}(x, \lambda_*, c) = \mathscr{L}(x_*, \lambda_*, c) = f(x_*) = f_*
  \quad \forall c > c_*(x_*, \lambda_*).
\]
Thetefore, by Proposition~\ref{prp:WeakDuality} the zero duality gap property holds true and the pair $(\lambda_*, c)$
is an optimal dual solution for any $c > c_*(x_*, \lambda_*)$.

\textbf{Part~\ref{st:DualSol_GSP}.} Let $(\lambda_*, c_*)$ and $x_*$ be globally optimal solutions of the problems
$(\mathcal{D})$ and $(\mathcal{P})$ respectively, and suppose that the zero duality gap between property holds true. Fix
any $c \ge c_*$. Then applying assumption $(A1)$ and Proposition~\ref{prp:DualSolution_2ndComponent} one obtains that
\[
  \mathscr{L}(x_*, \lambda_*, c) \le f(x_*) = f_* = \Theta_* = \Theta(\lambda_*, c) 
  = \inf_{x \in Q} \mathscr{L}(x, \lambda_*, c)
  \quad \forall c > c_*.
\]
Consequently, $\mathscr{L}(x_*, \lambda_*, c) = f(x_*)$, and applying assumption $(A1)$ once again one gets
\[
  \sup_{\lambda \in \Lambda} \mathscr{L}(x_*, \lambda, c) \le f(x_*) = \mathscr{L}(x_*, \lambda_*, c)
  = \inf_{x \in Q} \mathscr{L}(x, \lambda_*, c)
  \quad \forall c > c_*,
\]
which obviously means that $(x_*, \lambda_*)$ is a global saddle point of the augmented Lagrangian and 
$c_*(x_*, \lambda_*) \le c_*$.

\textbf{Part~\ref{st:MinMaxEqual}.} Let $(x_*, \lambda_*)$ be a global saddle point. Choose any 
$c > c_*(x_*, \lambda_*)$. From the proof of the first statement of the theorem it follows that
\[
  \mathscr{L}(x_*, \lambda_*, c) = f_* = \Theta_* = \Theta(\lambda_*, c) = \inf_{x \in Q} \mathscr{L}(x, \lambda_*, c)
  = \sup_{\lambda \in \Lambda} \inf_{x \in Q} \mathscr{L}(x, \lambda, c),
\]
where the last equality and the fact that $\Theta_* = \Theta(\lambda_*, c)$ follow from the fact that $(\lambda_*, c)$
is an optimal dual solution.

By the definition of global saddle point 
$\mathscr{L}(x_*, \lambda_*, c) = \sup_{\lambda \in \Lambda} \mathscr{L}(x_*, \lambda, c)$, which implies that
\[
  \mathscr{L}(x_*, \lambda_*, c) \ge \inf_{x \in Q} \sup_{\lambda \in \Lambda} \mathscr{L}(x, \lambda, c).
\]
On the other hand, by the same definition one also has
\[
  \mathscr{L}(x_*, \lambda_*, c) = \inf_{x \in Q} \mathscr{L}(x, \lambda_*, c) 
  \le \inf_{x \in Q} \sup_{\lambda \in \Lambda} \mathscr{L}(x, \lambda, c).
\]
Thus, $\mathscr{L}(x_*, \lambda_*, c) = \inf_{x \in Q} \sup_{\lambda \in \Lambda} \mathscr{L}(x, \lambda, c)$, and 
the proof is complete.
\end{proof}

\begin{remark} \label{rmrk:OptDualSol_vs_GlobalSaddlePoints}
Note that assumption $(A7)$ is not needed for the validity of the first and third statements of the theorem, since it is
not used in the proofs of these statements. In turn, assumptions $(A2)$ and $(A3)$ are not needed for the validity of
the second statement of the theorem.
\end{remark}

Combining the first and second statements of the previous theorem one obtains the two following useful results.

\begin{corollary} \label{crlr:GSP_vs_DualSol}
Let assumptions $(A1)$--$(A3)$ and $(A7)$ hold true. Then a global saddle point of $\mathscr{L}(\cdot)$ exists if and
only if there exist globally optimal solutions of the primal problem $(\mathcal{P})$ and the dual problem 
$(\mathcal{D})$ and the zero duality gap property holds true.
\end{corollary}

\begin{corollary} \label{crlr:GSP_PrimalOptSol}
Let assumptions $(A1)$--$(A3)$ and $(A7)$ be valid and $(x_*, \lambda_*)$ be a global saddle point of 
$\mathscr{L}(\cdot)$. Then for any globally optimal solution $z_*$ of the problem $(\mathcal{P})$ the pair
$(z_*, \lambda_*)$ is also a global saddle point of $\mathscr{L}(\cdot)$ and 
$c_*(x_*, \lambda_*) = c_*(z_*, \lambda_*) = c_*(\lambda_*)$.
\end{corollary}

Thus, the least exact penalty parameter $c_*(x_*, \lambda_*)$ does not depend on a globally optimal solution $x_*$
of the problem $(\mathcal{P})$ and is equal to the value of the penalty map $c_*(\lambda_*)$.

\begin{remark}
As was shown in \cite[Proposition~9]{Dolgopolik2018}, if the functions $f$ and $G$ are differentiable, then under some
natural assumptions on the function $\Phi$ any global saddle point of the augmented Lagrangian $\mathscr{L}(\cdot)$ is a
KKT-point of the problem $(\mathcal{P})$. This result implies that if there are two globally optimal solutions of the
problem $(\mathcal{P})$ having disjoint sets of multipliers satisfying KKT optimality conditions (note that problems
having such optimal solutions are necessarily nonconvex), then there are no global saddle points of the augmented
Lagrangian $\mathscr{L}(\cdot)$ and by Corollary~\ref{crlr:GSP_vs_DualSol} either the duality gap between the primal and
dual problems is positive or the dual problem has no globally optimal solutions. As we will show below, this fact leads
to the unboundedness of the sequence of multipliers or the sequence of penalty parameters generated by augmented
Lagrangian methods for problems having optimal solutions with disjoint sets of Lagrange multipliers.
\end{remark}

Let us give an example illustrating the previous remark.

\begin{example} \label{ex:DijointSetsOfMultipliers}
Let $X = Y = \mathbb{R}$. Consider the following optimization problem:
\begin{equation} \label{prob:DisjointMultiplierSets}
  \min \: f(x) = - x^2 \quad \text{subject to} \enspace g_1(x) = x - 1 \le 0, \enspace g_2(x) = - x - 1 \le 0.
\end{equation}
This problem has two globally optimal solutions: $1$ and $-1$. The corresponding Lagrange multipliers are
$(2, 0)$ and $(0, 2)$. Thus, the sets of Lagrange multipliers corresponding to two different globally
optimal solutions are disjoint. 

The optimal value function for problem \eqref{prob:DisjointMultiplierSets} has the form:
\begin{align*}
  \beta(p) &= \inf\Big\{ - x^2 \Bigm| x - 1 - p_1 \le 0, \: -x - 1 - p_2 \le 0 \Big\}
  \\
  &= \begin{cases}
    - \max\big\{ |1 - p_1|, |1 + p_2| \big\}^2, & \text{if } p_1 - p_2 \le 2,
    \\
    + \infty, & \text{if } p_1 - p_2 > 2.
  \end{cases}
\end{align*}
The function $\beta$ is obviously continuous at the origin. Therefore, under the assumptions of 
Theorem~\ref{thrm:ZeroDualityGapCharacterization} the zero duality gap property holds true. In particular, it holds true
for the Hestenes-Powell-Rockafellar augmented Lagrangian for problem \eqref{prob:DisjointMultiplierSets}:
\begin{equation} \label{eq:HPRLagr_DisjMultSetsProblem}
\begin{split}
  \mathscr{L}(x, \lambda, c) = - x^2 &+ \lambda_1 \max\left\{ x - 1, -\frac{\lambda_1}{c} \right\}
  + \frac{c}{2} \max\left\{ x - 1, -\frac{\lambda_1}{c} \right\}^2
  \\
  &+ \lambda_2 \max\left\{ -x - 1, -\frac{\lambda_2}{c} \right\}
  + \frac{c}{2} \max\left\{ -x - 1, -\frac{\lambda_2}{c} \right\}^2.
\end{split}
\end{equation}
However, the corresponding augmented dual problem has no globally optimal solutions. 

Indeed, if an optimal dual solution $(\lambda_*, c_*)$ exists, then by
Theorem~\ref{thrm:OptDualSol_vs_GlobalSaddlePoints} for any globally optimal solution $x_*$ of problem
\eqref{prob:DisjointMultiplierSets} the pair $(x_*, \lambda_*)$ is a global saddle point of the augmented Lagrangian and
$c_*(x_*, \lambda_*) \le c_*$. Therefore by the third statement of Theorem~\ref{thrm:OptDualSol_vs_GlobalSaddlePoints}
and the definition of global saddle point for any $c > c_*$ one has
\begin{equation} \label{eq:DijointMultipliers_GSP_equal}
  f_* = -1 = \mathscr{L}(1, \lambda_*, c) = \mathscr{L}(-1, \lambda_*, c) 
  = \inf_{x \in \mathbb{R}} \mathscr{L}(x, \lambda_*, c).
\end{equation}
Hence applying assumption $(A1)$ one gets that
\begin{align*}
  0 &= \Phi(G(1), \lambda_*, c) = (\lambda_*)_2 \max\left\{ -2, -\frac{(\lambda_*)_2}{c} \right\}
  + \frac{c}{2} \max\left\{ -2, -\frac{(\lambda_*)_2}{c} \right\}^2
  \\
  0 &= \Phi(G(-1), \lambda_*, c) = (\lambda_*)_1 \max\left\{ -2, -\frac{(\lambda_*)_1}{c} \right\}
  + \frac{c}{2} \max\left\{ -2, -\frac{(\lambda_*)_1}{c} \right\}^2 
\end{align*}
Clearly, there exists $c_0 > 0$ such that $(\lambda_*)_1 / c < 2$ and $(\lambda_*)_2 / c < 2$ for any $c \ge c_0$. 
Therefore, by the equalities above for any $c \ge c_0$ one has 
\[
  0 = - \frac{(\lambda_*)_1^2}{2 c}, \quad 0 = - \frac{(\lambda_*)_2^2}{2 c}
\] 
or, equivalently, $\lambda_* = 0$. However, as one can easily check,
\[
  \inf_{x \in \mathbb{R}} \mathscr{L}(x, \lambda_*, c) = \inf_{x \in \mathbb{R}} \mathscr{L}(x, 0, c)
  = - 1 - \frac{2}{c - 2} < f_* \quad \forall c > 2,
\]
which contradicts \eqref{eq:DijointMultipliers_GSP_equal}. Thus, the augmented dual problem has no globally optimal 
solutions. In the following section we will show how a standard primal-dual augmented Lagrangian method behaves for 
problem~\eqref{prob:DisjointMultiplierSets} (see Example~\ref{ex:LackOfDualSol_vs_DualConvergence}).
\end{example}

Another object in the augmented duality theory, that is directly connected with optimal dual solutions and global saddle
points, is \textit{augmented Lagrange multiplier}, which we introduce below by analogy with the theory of 
Rockafel\-lar-Wets' augmented Lagrangians 
\cite{ShapiroSun,RuckmannShapiro,ZhouZhouYang2014,BurachikYangZhou,Dolgopolik2017}.

\begin{definition}
A vector $\lambda_* \in \Lambda$ is called an \textit{augmented Lagrange multiplier} (of the augmented Lagrangian
$\mathscr{L}(\cdot)$), if there exists $x_* \in Q$ such that the pair $(x_*, \lambda_*)$ is a global saddle
point of the augmented Lagrangian. The set of all augmented Lagrange multipliers is denoted by $\mathscr{A}_*$.
\end{definition}

\begin{remark}
{(i)~Our definition of the augmented Lagrange multiplier is equivalent to the one used in the context of
Rockafellar-Wets' augmented Lagrangians by \cite[Theorem~2.1]{ShapiroSun}.}

{(ii)~By Theorem~\ref{thrm:OptDualSol_vs_GlobalSaddlePoints}, in the case when the zero duality gap property is
satisfied and there exists a globally optimal solution of the primal problem, a vector $\lambda_* \in \Lambda$ is an
augmented Lagrange multiplier if and only if there exists $c_* > 0$ such that $(\lambda_*, c_*)$ is an optimal dual
solution. Consequently, $\mathscr{A}_* = \dom c_*(\cdot)$.
}
\end{remark}

Let us point out some interesting properties of augmented Lagrange multipliers.

\begin{proposition} \label{prp:AugmLagrMult}
Let assumption $(A1)$ and the zero duality gap property hold true, and there exist a globally optimal solution of 
the problem $(\mathcal{P})$. Then a vector $\lambda_* \in \Lambda$ is an augmented Lagrange multiplier if and only if
there exists $c_* > 0$ such that 
\begin{equation} \label{eq:AugmLagrMultEquality}
  \Theta(\lambda_*, c_*) := \inf_{x \in Q} \mathscr{L}(x, \lambda_*, c_*) = f_*.
\end{equation}
Furthermore, if this equality and assumptions $(A2)$, $(A3)$, and $(A7)$ are satisfied, then the following statements
hold true:
\begin{enumerate}
\item{$\Theta(\lambda_*, c) = f_*$ for all $c \ge c_*$;}

\item{the infimum of all $c_*$ for which \eqref{eq:AugmLagrMultEquality} holds true is equal to $c_*(\lambda_*)$;
}

\item{for all $c_* > c_*(\lambda_*)$ the infimum in \eqref{eq:AugmLagrMultEquality} is attained at every globally
optimal solution of the problem $(\mathcal{P})$;
}

\item{if the function $c \mapsto \Phi(y, \lambda_*, c)$ is strictly increasing on $\dom \Phi(y, \lambda_*, \cdot)$ for
any $y \notin K$ and on $T(y) := \{ c \in (0, + \infty) \colon - \infty < \Phi(y, \lambda_*, c) < 0 \}$ for any 
$y \in K$, then for all $c_* > c_*(\lambda_*)$ the infimum in \eqref{eq:AugmLagrMultEquality} is attained at some 
$x \in Q$ if and only if $x$ is a globally optimal solution of the problem $(\mathcal{P})$.
}
\end{enumerate}
\end{proposition}

\begin{proof}
\textbf{Part 1.} Let $\lambda_*$ be an augmented Lagrange multiplier. Then by
Theorem~\ref{thrm:OptDualSol_vs_GlobalSaddlePoints} there exists $c_* > 0$ such that $(\lambda_*, c_*)$ is an optimal
dual solution. Hence with the use of the fact that the zero duality gap property holds true one can conclude that
equality \eqref{eq:AugmLagrMultEquality} is valid.

Suppose now that equality \eqref{eq:AugmLagrMultEquality} is satisfied for some $\lambda_* \in \Lambda$ and $c_* > 0$.
Then by Proposition~\ref{prp:WeakDuality} the pair $(\lambda_*, c_*)$ is an optimal dual solution, which by
Theorem~\ref{thrm:OptDualSol_vs_GlobalSaddlePoints} implies that $\lambda_*$ is an augmented Lagrange multiplier.

\textbf{Part 2.1.} Suppose that that equality \eqref{eq:AugmLagrMultEquality} is satisfied for some 
$\lambda_* \in \Lambda$ and $c_* > 0$, and assumptions $(A2)$, $(A3)$, and $(A7)$ hold true. Then, as was noted
earlier, the function $\Theta(\lambda, c)$ is non-decreasing in $c$, which along with Proposition~\ref{prp:WeakDuality}
imply that $\Theta(\lambda_*, c) = f_*$ for all $c \ge c_*$.

\textbf{Part 2.2.} The fact that the infimum of all $c_*$ for which \eqref{eq:AugmLagrMultEquality} holds true is equal
to $c_*(\lambda_*)$ follows directly from the definition of the penalty map.

\textbf{Part 2.3.} If a pair $(\lambda_*, c_*)$ satisfies equality \eqref{eq:AugmLagrMultEquality}, then it is an
optimal dual solution. Consequently, by Theorem~\ref{thrm:OptDualSol_vs_GlobalSaddlePoints} for any globally optimal
solution $x_*$ of the problem $(\mathcal{P})$ the pair $(x_*, \lambda_*)$ is a global saddle point of the augmented
Lagrangian. By the definition of global saddle point it means that the infimum in \eqref{eq:AugmLagrMultEquality} is
attained at $x_*$ (see \eqref{eq:GSP_def}).

\textbf{Part 2.4.} Suppose finally that the function $c \mapsto \Phi(y, \lambda_*, c)$ is strictly increasing on
$\dom \Phi(y, \lambda_*, \cdot)$ for any $y \notin K$ and on $T(y)$ for any $y \in K$, and the infimum in
\eqref{eq:AugmLagrMultEquality} is attained at some $x \in Q$. If $x$ is feasible and $\Phi(G(x), \lambda_*, c_*) < 0$
or $x$ is infeasible, then for any $c_*(\lambda_*) < c < c_*$ by our assumption on the function $\Phi$
one has $f_* = \mathscr{L}(x, \lambda_*, c_*) > \mathscr{L}(x, \lambda_*, c) \ge f_*$, which is obviously impossible.
Therefore, $x$ is feasible and $\Phi(G(x), \lambda_*, c_*) = 0$, which means that $f(x) = f_*$, that is, $x_*$ is a
globally optimal solution of the problem $(\mathcal{P})$.
\end{proof}

Thus, if $\lambda_* \in \Lambda$ is an augmented Lagrange multiplier, then under some additional assumptions the problem
$(\mathcal{P})$ has the same optimal value and the same globally optimal solutions as the problem 
\[
  \min_x \: \mathscr{L}(x, \lambda_*, c) \quad \text{subject to } x \in Q
\]
for any $c > c_*(\lambda_*)$. That is why in \cite{RockafellarWets} augmented Lagrange multipliers were called
multipliers supporting an \textit{exact penalty representation} (see \cite[Definition~11.60]{RockafellarWets} and
\cite{HuangYang2003,HuangYang2005,ZhouYang2009,BurachikIusemMelo,WangYangYang,ZhouYang2012,BurachikIusemMelo2015}).

\begin{remark}
The assumption that the function $c \mapsto \Phi(y, \lambda_*, c)$ is strictly increasing on 
$\dom \Phi(y, \lambda_*, \cdot)$ for any $y \notin K$ and on $T(y)$ for any $y \in K$ is very mild and satisfied in many
particular cases. For example, it is satisfied \textit{for any} $\lambda_* \in \Lambda$ for Rockafellar-Wets'
augmented Lagrangian (Example~\ref{ex:RockafellarWetsAugmLagr}), provided the function $\sigma$ has 
a valley at zero and is continuous at this point, the Hestenes-Powell-Rockafellar augmented Lagrangian
(Examples~\ref{ex:HestenesPowellAugmLagr}, \ref{ex:HPR_AugmLagr_2OrderCone}, \ref{ex:HPR_AugmLagr_SemiDef},
\ref{ex:HPR_AugmLagr_Pointwise}), the sharp Lagrangian (Example~\ref{ex:SharpLagrangian_eq}), Mangasarian's  augmented
Lagrangian(Examples~\ref{ex:MangasarianAugmLagr_eq} and \ref{ex:MangasarianAugmLagr_ineq}), the essentially quadratic
augmented Lagrangian (Example~\ref{ex:HPR_AugmLagr}), the cubic augmented Lagrangian (Example~\ref{ex:CubicAugmLagr}),
the penalized exponential-type augmented Lagrangian (Examples~\ref{ex:PenalizedExpTypeAugmLagr} and
\ref{ex:PenalizedExpTypeAugmLagr_SemiDef}), provided the function $\xi$ is strictly convex on $(0, + \infty)$, and
He-Wu-Meng's augmented Lagrangian (Example~\ref{ex:HeWuMeng}). This assumption is also satisfied for any $\lambda_* \in
\Lambda$ that does not have zero components for the exponential-type augmented Lagrangian
(Example~\ref{ex:ExpTypeAugmLagr}), the hyperbolic-type augmented Lagrangian (Example~\ref{eq:HyperbolicAugmLagr}), and
the modified barrier function (Example~\ref{ex:ModifiedBarrierFunction}), provided the function $\phi$ is strictly
convex in these examples, and for the p-th power augmented Lagrangian (Example~\ref{ex:pthPowerAugmLagr}).
\end{remark}

\begin{remark}
To the best of the author's knowledge, the penalty map, augmented Lagrange multipliers, and an exact penalty
representation have been introduced and studied earlier only in the context of Rockafellar-Wets' augmented Lagrangians.
In this section we demonatrated (see Corollary~\ref{crlr:OptDualSolSet_ExactPenaltyMap} and
Proposition~\ref{prp:AugmLagrMult}) that there is nothing specific in these concepts that is inherently connected to
Rockafellar-Wets' augmented Lagrangians. They can be naturally introduced and studied for any other augmented
Lagrangian, including the (penalized) exponential-type augmented Lagrangian, modified barrier functions, Mangasarian's
augmented Lagrangian, etc., that are typically not considered in the theory of Rockafellar-Wets' augmented Lagrangians.
\end{remark}

\subsection{Optimal dual solutions for convex problems}

In the case when the problem $(\mathcal{P})$ is convex, optimal dual solutions, roughly speaking, do not depend on 
the penalty parameter $c$ (more precisely, the penalty map $c_*(\lambda_*)$ does not depend on 
$\lambda_* \in \mathscr{A}_*$), and one can give a very simple (and well-known) description of the set of optimal dual
solutions in terms of \textit{Lagrange multipliers}.

Let the function $f$ and the set $Q$ be convex, and the mapping $G$ be convex with respect to binary relation $\preceq$,
that is,
\[
  G(\alpha x_1 + (1 - \alpha) x_2) \preceq \alpha G(x_1) + (1 - \alpha) G(x_2) 
  \quad \forall x_1, x_2 \in X, \: \alpha \in [0, 1].
\]
Then the problem $(\mathcal{P})$ is convex. Moreover, in this case under assumption $(A8)$ the augmented Lagrangian
$\mathscr{L}(x, \lambda, c)$ is convex in $x$. Indeed, by applying first the fact that $\Phi(\cdot, \lambda, c)$ is
non-decreasing with respect to $\preceq$ and then the fact that this function is convex one obtains that for any 
$x_1, x_2 \in X$ and $\alpha \in [0, 1]$ the following inequalities hold true:
\begin{align*}
  \mathscr{L}(\alpha x_1 &+ (1 - \alpha) x_2, \lambda, c) 
  = f(\alpha x_1 + (1 - \alpha) x_2) + \Phi(G(\alpha x_1 + (1 - \alpha) x_2), \lambda, c)
  \\
  &\le \alpha f(x_1) + (1 - \alpha) f(x_2) + \Phi(\alpha G(x_1) + (1 - \alpha) G(x_2), \lambda, c)
  \\
  &\le \alpha f(x_1) + (1 - \alpha) f(x_2) + \alpha \Phi(G(x_1), \lambda, c) + (1 - \alpha) \Phi(G(x_2), \lambda, c)
  \\
  &= \alpha \mathscr{L}(x_1, \lambda, c) + (1 - \alpha) \mathscr{L}(x_2, \lambda, c),
\end{align*}
which means that the function $\mathscr{L}(\cdot, \lambda, c)$ is convex for any $\lambda \in \Lambda$ and $c > 0$
(in the case when $\Phi(\alpha G(x_1) + (1 - \alpha) G(x_2), \lambda, c) = - \infty$, instead of the inequalities above
one should apply \cite[Theorem~I.4.2]{Rockafellar}).

Denote by $L(x, \lambda) = f(x) + \langle \lambda, G(x) \rangle$ the standard Lagrangian for the problem 
$(\mathcal{P})$, where $\lambda \in K^*$. It is easily seen that the function $L(\cdot, \lambda)$ is convex. Recall that
a vector $\lambda_* \in K^*$ is called \textit{a Lagrange multiplier} of the problem $(\mathcal{P})$ at a feasible
point $x_*$, if $0 \in \partial_x L(x_*, \lambda_*) + N_Q(x_*)$ and $\langle \lambda_*, G(x_*) \rangle = 0$, where 
$\partial_x L(x_*, \lambda_*)$ is the subdifferential of the function $L(\cdot, \lambda_*)$ at $x_*$ in the sense of
convex analysis and $N_Q(x_*) = \{ x^* \in X^* \mid \langle x^*, x - x_* \rangle \le 0 \: \forall x \in Q \}$ is 
the normal cone to the set $Q$ at $x_*$ (see, e.g. \cite[Definition~3.5]{BonnansShapiro}). The existence of a Lagrange
multiplier at $x_*$ is a sufficient, and in the case when $0 \in \interior (G(Q) - K)$ necessary, global optimality
condition, and the set $\Lambda_*$ of Lagrange multipliers of the problem $(\mathcal{P})$ is a nonempty, convex,
weak${}^*$ compact set that does not depend on an optimal solution $x_*$ (see \cite[Theorem~3.6]{BonnansShapiro}).
Furthermore, $\lambda_*$ is a Lagrange multiplier at $x_*$ if and only if $(x_*, \lambda_*)$ is a saddle point of 
the Lagrangian $L(\cdot)$:
\[
  \sup_{\lambda \in K^*} L(x_*, \lambda) \le L(x_*, \lambda_*) \le \inf_{x \in Q} L(x, \lambda_*).
\]
Note finally that if $L(\cdot, \lambda_*)$ is directionally differentiable at $x_*$, then $\lambda_*$ is a Lagrange
multiplier at $x_*$ if and only if $[L(\cdot, \lambda_*)]'(x_*, h) \ge 0$ for all $h \in T_Q(x_*)$ and 
$\langle \lambda_*, G(x_*) \rangle = 0$, where $[L(\cdot, \lambda_*)]'(x_*, h)$ is the directional derivative of
$L(\cdot, \lambda_*)$ at $x_*$ in the direction $h$, and $T_Q(x_*)$ is the contingent cone to the set $Q$ at the point
$x_*$ (cf. \cite[Lemma~3.7]{BonnansShapiro}).

Let us now present a complete characterisation of the set optimal dual solutions in terms of Lagrange multipliers in 
the convex case. Roughly speaking, this result shows that under suitable assumptions there is essentially no difference
between standard duality theory, based on the Lagrangian $L(\cdot)$, and augmented duality theory, based on 
the augmented Lagrangian $\mathscr{L}(\cdot)$. For the sake of simplicity we will prove this result under the assumption
that the functions $f$ and $G$ are directionally differentiable at a globally optimal solution of the problem
$(\mathcal{P})$, although in various particular cases (e.g. in the case of inequality constrained problems) this result
can be proved without this assumption.

\begin{theorem} \label{thrm:OptDualSol_ConvexCase}
Let the following conditions hold true:
\begin{enumerate}
\item{$f$ and $Q$ are convex, $G$ is convex with respect to the binary relation $\preceq$;}

\item{assumptions $(A1)$, $(A4)$, $(A5)$, $(A7)$, $(A8)$, and $(A11)$ hold true;}

\item{$K^* \subseteq \Lambda$;} 

\item{$f$ and $G$ are directionally differentiable at a globally optimal solution $x_*$ of the problem 
$(\mathcal{P})$.
}
\end{enumerate}
Then a Lagrange multiplier of the problem $(\mathcal{P})$ exists if and only if the zero duality gap property holds true
and there exists a globally optimal solution $(\lambda_*, c_*)$ of the augmented dual problem $(\mathcal{D})$ with
$\lambda_* \in K^*$. 

Moreover, if a Lagrange multiplier of the problem $(\mathcal{P})$ exists, then $(\lambda_*, c_*)$ with 
$\lambda_* \in K^*$ is an optimal dual solution if and only if $\Phi_0(\lambda_*)$ is a Lagrange multiplier of 
the problem $(\mathcal{P})$ and $c_* > 0$, where $\Phi_0$ is from assumption $(A11)$.
\end{theorem}

\begin{proof}
Suppose that a Lagrange multiplier $\lambda_* \in K^*$ exists. Then, as was noted above, 
$[L(\cdot, \lambda_*)]'(x_*, h) \ge 0$ for all $h \in T_Q(x_*)$. By assumption $(A11)$ the function 
$\Phi(\cdot, \lambda, c)$ is Fr\'{e}chet differentiable and its Fr\'{e}chet derivative 
$D_y \Phi(y, \lambda, c) = \Phi_0(\lambda)$ is a surjective mapping from $K^*$ onto $K^*$. Therefore, there exists
$\mu_* \in K^*$ such that $\Phi_0(\mu_*) = \lambda_*$. Applying the chain rule for directional derivatives one gets
\[
  [\mathscr{L}(\cdot, \mu_*, c)]'(x_*, h) = f'(x_*, h) + \langle \Phi_0(\mu_*), G'(x_*, h) \rangle 
  = [L(\cdot, \lambda_*)]'(x_*, h) \ge 0
\]
for any $c > 0$ and $h \in T_Q(x_*)$. As was noted above, under the assumptions of the theorem the function 
$\mathscr{L}(\cdot, \mu_*, c)$ is convex. Consequently, the inequality above implies that $x_*$ is a point of 
global minimum of $\mathscr{L}(\cdot, \mu_*, c)$ on the set $Q$. Recall that $\langle \lambda_*, G(x_*) \rangle = 0$, 
since $\lambda_*$ is a Lagrange multiplier. Hence by assumption $(A11)$ one has $\langle \mu_*, G(x_*) \rangle = 0$,
and with the use of assumption $(A4)$ one gets
\[
  f_* = f(x_*) = \mathscr{L}(x_*, \mu_*, c) = \inf_{x \in Q} \mathscr{L}(x, \mu_*, c) = \Theta(\mu_*, c)
  \quad \forall c > 0,
\]
which by Proposition~\ref{prp:WeakDuality} means that the zero duality gap property is satisfied and $(\mu_*, c)$ with
any $c > 0$ is an optimal dual solution.

Suppose now that the zero duality gap property holds true and there exists an optimal solution $(\mu_*, c_*)$ of the
problem $(\mathcal{D})$ with $\mu_* \in K^*$. Then by Theorem~\ref{thrm:OptDualSol_vs_GlobalSaddlePoints} (see also
Remark~\ref{rmrk:OptDualSol_vs_GlobalSaddlePoints}) the pair $(x_*, \mu_*)$ is a global saddle point of the augmented
Lagrangian and $c_*(\mu_*) \le c_*$. Therefore, by the definition of global saddle point $x_*$ is a point of global
minimum of $\mathscr{L}(\cdot, \mu_*, c)$ on the set $Q$ and 
\begin{equation} \label{eq:GSP_vs_PrimalDualOptVal}
  \mathscr{L}(x_*, \mu_*, c) = \Theta(\mu_*, c) = \Theta(\mu_*, c_*) = \Theta_* = f_* \quad \forall c > c_*
\end{equation}
(here we used Proposition~\ref{prp:DualSolution_2ndComponent}). Hence with the use of assumption $(A11)$ one obtains
that
\[
  0 \le [\mathscr{L}(\cdot, \mu_*, c)]'(x_*, h) = f'(x_*, h) + \langle \Phi_0(\mu_*), G'(x_*, h) \rangle 
  = [L(\cdot, \lambda_*)]'(x_*, h) 
\]
for any $h \in T_Q(x_*)$, where $\lambda_* = \Phi_0(\mu_*)$. Moreover, $\langle \mu_*, G(x_*) \rangle = 0$ and,
therefore, $\langle \lambda_*, G(x_*) \rangle = 0$ by assumption $(A11)$, 
since otherwise by assumption $(A5)$ one has $\mathscr{L}(x_*, \mu_*, c) < f(x_*) = f_*$, which contradicts 
\eqref{eq:GSP_vs_PrimalDualOptVal}. Thus, $\lambda_*$ is a Lagrange multiplier.

It remains to note that, as was shown above, if $\lambda_*$ is a Lagrange multiplier, then for any $c > 0$ and 
$\mu_* \in K^*$ such that $\Phi_0(\mu_*) = \lambda_*$ the pair $(\mu_*, c)$ is an optimal dual solution. 
Conversely, if $(\mu_*, c_*)$ with $\mu_* \in K^*$ is an optimal dual solution, then $\lambda_* = \Phi_0(\mu_*)$ 
is a Lagrange multiplier.
\end{proof}

\begin{corollary} \label{crlr:LagrangeMultipliersVsOptDualSol_ConvexCase}
Let $f$ and $Q$ be convex, $G$ be convex with respect to the binary relation $\preceq$, $f$ and $G$ be directionally
differentiable at an optimal solution $x_*$ of the problem $(\mathcal{P})$, $K^* \subseteq \Lambda$, and suppose that
assumptions $(A1)$--$(A8)$ and $(A11)$ hold true. Then a Lagrange multiplier of the problem $(\mathcal{P})$ exists if 
and only if the zero duality gap property holds true and there exists an optimal dual solution. Moreover, if a Lagrange
multiplier of the problem $(\mathcal{P})$ exists, then $(\lambda_*, c_*)$ is a globally optimal solution of the dual
problem $(\mathcal{D})$ if and only if $\Phi_0(\lambda_*)$ is a Lagrange multiplier of the problem $(\mathcal{P})$ and
$c_* > 0$.
\end{corollary}

\begin{proof}
Let $(\lambda_*, c_*)$ be an optimal dual solution. Then by Theorem~\ref{thrm:OptDualSol_vs_GlobalSaddlePoints} the pair
$(x_*, \lambda_*)$ is a global saddle point, and with the use of \cite[Proposition~2]{Dolgopolik2018} one can conclude
that $\lambda_* \in K^*$. Thus, for any optimal dual solution $(\lambda_*, c_*)$ one has $\lambda_* \in K^*$ and 
the claim of the corollary follows directly from Theorem~\ref{thrm:OptDualSol_ConvexCase}.
\end{proof}

With the use of the previous corollary we can finally describe the structure of the set of optimal dual solutions
$\mathscr{D}_*$, the penalty map $c_*(\cdot)$, and the set of augmented Lagrange multipliers $\mathscr{A}_*$ 
in the convex case.

\begin{corollary} \label{crlr:OptDualSolSet_LagrangeMultipliers_ConvexCase}
Let the assumptions of the previous corollary be valid. Then $\mathscr{A}_* = \Phi_0^{-1}(\Lambda_*)$, 
$c_*(\lambda_*) = 0$ for any $\lambda_* \in \mathscr{A}_*$, and the following equality holds true:
\begin{equation} \label{eq:OptDualSolSet_ConvexCase}
  \mathscr{D}_* = \Phi_0^{-1}\big(\Lambda_*\big) \times (0, + \infty).
\end{equation}
\end{corollary}

\begin{remark}
The fact that optimal dual solutions for convex problems do not depend on the penalty parameter motivates one to
consider a slightly different augmented dual problem in the convex case:
\begin{equation} \label{eq:AugmDualProblem_ConvexCase}
  \min_{\lambda} \: \Theta_c(\lambda) \quad \text{subject to } \lambda \in \Lambda.
\end{equation}
Here $\Theta_c(\lambda) = \Theta(c, \lambda)$ and $c > 0$ is fixed, that is, the penalty parameter is not considered as
a variable of augmented dual problem, but rather as a fixed external parameter. Note that the function $\Theta_c(\cdot)$
is concave, provided assumption $(A9)$ holds true, which is satisfied for most particular augmented Lagrangians, in
contrast to the much more restrictive assumption $(A9)_s$, which is satisfied, to the best of the author's knowledge,
only for Rockafellar-Wets' augmented Lagrangian and, in particular, the Hestenes-Powell-Rockafellar augmented
Lagrangian. Taking into account the concavity of the function $\Theta_c(\cdot)$ one can consider primal-dual augmented
Lagrangian methods based on solving problem \eqref{eq:AugmDualProblem_ConvexCase}, instead of the augmented dual problem
$(\mathcal{D})$, i.e. augmented Lagrangian methods with fixed penalty parameter. Convergence analysis of such methods
is always based on the use of a particular structure of an augmented Lagrangian under consideration (see the survey of
such methods in \cite{Iusem99}), which makes it difficult to extend such analysis to the general axiomatic augmented
Lagrangian setting adopted in this article.
\end{remark}

Let us show that in the case when assumptions $(A5)$, $(A6)$, and $(A11)$ are \textit{not} satisfied, equality 
\eqref{eq:OptDualSolSet_ConvexCase} might be no longer valid and the penalty map might not be identically equal to zero
on $\dom c_*(\cdot)$, even when the problem $(\mathcal{P})$ is convex.

\begin{example}
Let $X = Y = \mathbb{R}$. Consider the following optimization problem:
\begin{equation} \label{prob:SharpLagrangian}
  \min \: f(x) = -x \quad \text{subject to} \enspace g(x) = x \le 0.
\end{equation}
Let $\mathscr{L}(\cdot)$ be the sharp Lagrangian for this problem (i.e. the augmented Lagrangian from 
Example~\ref{ex:RockafellarWetsAugmLagr} with $\sigma(y) = \| y \|$). Then
\begin{align*}
  \mathscr{L}(x, \lambda, c) &= f(x) + \inf_{p \in (- \infty, -g(x)]} \big( - \lambda p + c |p| \big)
  \\
  &= -x + \begin{cases}
    \lambda x + c |x|, & \text{if } \lambda > c
    \\
    (\lambda + c) \max\{ 0, x \}, & \text{if } |\lambda| \le c
    \\
    - \infty, & \text{if } \lambda < - c,
  \end{cases}
\end{align*}
and, as one can readily verify by carefully writing down all particular cases,
\[
  \Theta(\lambda, c) = \begin{cases}
    0, & \text{if } c \ge |\lambda - 1|,
    \\
    - \infty, & \text{otherwise.}
  \end{cases}
\]
Consequently, one has 
\[
  \mathscr{D}_* = \Big\{ (\lambda, c) \in \mathbb{R}^2 \Bigm| c \ge |\lambda - 1| \Big\},
  \quad c_*(\lambda) = |\lambda - 1| \quad \forall \lambda \in \mathbb{R},
\]
that is, the claims of Corollary~\ref{crlr:OptDualSolSet_LagrangeMultipliers_ConvexCase} do not hold true for the sharp
Lagrangian (recall that this Lagrangian does not satisfy assumptions $(A5)$, $(A6)$, and $(A11)$).
\end{example}

\section{Convergence analysis of augmented Lagrangian methods}
\label{sect:ConvergenceAnalysis}

The goal of this section is to prove general convergence theorems for a large class of augmented Lagrangian methods
and to analyse interrelations between convergence of augmented Lagrangian methods, zero duality gap property, and 
the existence of global saddle points/optimal dual solutions. We aim at presenting such results that explicitly
highlight this kind of interrelations, instead of implicitly using them within the proofs, as it is usually done 
in the literature.

\subsection{Model augmented Lagrangian method}

We present all theoretical results for the following model augmented Lagrangian method given in 
Algorithm~\ref{alg:ModelAugmLagrMethod}. In order to include various particular cases into the general theory, we do not
specify a way in which multipliers and the penalty parameter are updated by the method, that is, they can be updated in
\textit{any} way that satisfies certain assumptions presented below. It makes our results applicable to the vast
majority of existing augmented Lagrangian methods, including methods with nonmonotone penalty parameter updates as in
\cite{BirginMartinez2012}, methods based on maximizing the augmented dual function with the use of bundle methods as in
\cite{CordovaOliveiraSagastizabal}, etc. However, one should underline that our convergence analysis is by no means
universal and there are augmented Lagrangian methods to which it cannot be applied. We will briefly discuss some such
methods further in this section (see Remark~\ref{rmrk:ModifiedSubgradientMethod} below).

\begin{algorithm} \label{alg:ModelAugmLagrMethod}
\caption{Model augmented Lagrangian method}

\textbf{Initialization.} Choose an initial value of the multipliers $\lambda_0 \in \Lambda$, a minimal value of 
the penalty parameter $c_{min} > 0$, an initial value of the penalty parameter $c_0 \ge c_{\min}$, and a sequence 
$\{ \varepsilon_n \} \subset (0, + \infty)$ of tolerances. Put $n = 0$.

\textbf{Step 1. Solution of subproblem.} Find an $\varepsilon_n$-optimal solution $x_n$ of the problem
\[
  \min \: \mathscr{L}(x, \lambda_n, c_n) \quad \text{subject to} \quad x \in Q,
\]
that is, find $x_n \in Q$ such that 
$\mathscr{L}(x_n, \lambda_n, c_n) \le \mathscr{L}(x, \lambda_n, c_n) + \varepsilon_n$
for all $x \in Q$.

\textbf{Step 2. Multiplier update.} Choose some $\lambda_{n + 1} \in \Lambda$.

\textbf{Step 3. Penalty parameter update.} Choose some $c_{n + 1} \ge c_{min}$. Increment $n$ and go to
\textbf{Step 1.}
\end{algorithm}

Let us comment on Step~1 on Algorithm~\ref{alg:ModelAugmLagrMethod}. For the purposes of theoretical analysis of
primal-dual augmented Lagrangian methods it is often assumed that the augmented Lagrangian subproblem
\[
  \min \: \mathscr{L}(x, \lambda_n, c_n) \quad \text{subject to} \quad x \in Q,
\]
is solved exactly, i.e. that $x_n$ is a globally optimal solution of this problem
\cite{TsengBertsekas,Polyak2001,Gasimov,BurachikGasimov,LiuZhang2007,LiuZhang2008,LiZhang2009,BurachikKaya,
BirginMartinez}. Moreover, even when it is assumed that this subproblem is solved only approximately as in
\cite{BurachikKayaMammadov,LuoSunLi,WangLi2009,BurachickIusemMelo_SharpLagr,Burachick2011,BurachikIusemMelo2013,
BurachikLiu2023,BurachikKayaLiu2023}, one almost always assumes that $\varepsilon_n \to 0$ as 
$n \to \infty$, and the case when $\varepsilon_n$ does not tend to zero is not properly analysed (papers
\cite{LuoWuChen2012,CordovaOliveiraSagastizabal} are very rare exceptions to this rule). However, from the practical
point of view the assumption that $\varepsilon_n \to 0$ as $n \to \infty$ cannot be satisfied, especially in the
infinite dimensional case, due to round off errors, discretisation errors, etc. The value $\varepsilon_n > 0$ should be
viewed as \textit{an unavoidable error} reflecting the overall precision with which computations \textit{can be
performed} that does not tend to zero with iterations. To take this unavoidable error into account, below we present a
detailed analysis of the model augmented Lagrangian method without assuming that $\varepsilon_n \to 0$ as 
$n \to \infty$, and then show how corresponding convergence theorems can be clarified and strengthened by imposing this
additional \textit{purely theoretical} assumption.

It should also be noted that practical augmented Lagrangian methods must include stopping criteria. We do not include 
a stopping criterion in our formulation of the model augmented Lagrangian method, because we are interested only in its
theoretical (asymptotic) analysis, that is, in the analysis of the way sequences $\{ (x_n, \lambda_n, c_n) \}$ generated
by this method behave as $n \to \infty$. This asymptotic analysis can be used to devise appropriate stopping criteria
for practical implementations of augmented Lagrangian methods.

Below we will utilise the following natural assumptions on the model augmented Lagrangian method and sequences generated
by this method that are satisfied in many particular cases:
\begin{itemize}
\item[(B1)]{for any $n \in \mathbb{N}$ the function $\mathscr{L}(\cdot, \lambda_n, c_n)$ is bounded below on $Q$;}

\item[(B2)]{the sequence of multipliers $\{ \lambda_n \}$ is bounded;}

\item[(B3)]{if the sequence of penalty parameters $\{ c_n \}$ is bounded, then one has $\dist(G(x_n), K) \to 0$ as 
$n \to \infty$; if, in addition, some subsequence $\{ \lambda_{n_k} \}$ is bounded, then 
$\Phi(G(x_{n_k}), \lambda_{n_k}, c_{n_k}) \to 0$ as $k \to \infty$;
}

\item[(B4)]{if the sequence of penalty parameters $\{ c_n \}$ is unbounded, then $c_n \to + \infty$ as $n \to \infty$.}
\end{itemize}

The assumption $(B1)$ is a basic assumption for all primal-dual augmented Lagrangian methods, which is needed to ensure
that the sequence $\{ x_n \}$ is correctly defined. The assumption $(B2)$ is often imposed for the purposes of
convergence analysis of augmented Lagrangian methods and, as is noted in \cite{BirginMartinez}, is usually satisfied in
practice for traditional rules for updating multipliers. Moreover, various techniques can be used to guarantee 
the validity of assumption $(B2)$, such as safeguarding and normalization of multipliers
\cite{LuoSunLi,LuoSunWu,BirginMartinez}. 

We formulate $(B3)$ as an assumption due to the fact that we do not impose any restrictions on the way in which
the penalty parameter $c_n$ is updated. For many augmented Lagrangian methods, penalty parameter updates are
specifically designed to ensure that assumption $(B3)$ is satisfied by default (see the rules for updating the penalty
parameter and corresponding convergence analysis in \cite{LuoSunLi,LuoSunWu,WangLi2009,BirginMartinez} and other
aforementioned papers on augmented Lagrangian methods). Finally, assumption $(B4)$ is needed only in the case of methods
with nonmonotone penalty parameter updates. It excludes the undesirable situation of unboundedly increasing oscillations
of the penalty parameter (e.g. $c_{2n} = n$ and $c_{2n + 1} = 1$ for all $n \in \mathbb{N}$), which cannot be properly
analysed within out general augmented Lagrangian setting.

\begin{remark} \label{rmrk:ModifiedSubgradientMethod}
{(i)~Assumption $(B3)$ plays one of the key roles in our convergence analysis of the model augmented Lagrangian method. 
Therefore, this analysis is inapplicable to those methods for which assumption $(B3)$ is not satisfied, such as 
the modified subgradient algorithm (the MSG) proposed by Gasimov \cite{Gasimov} (the fact that assumption $(B3)$ is not
satisfied for the MSG in the general case follows from \cite[Example~1]{BurachikGasimov}).
}

{(ii)~The convergence analysis of the model augmented Lagrangian method presented below heavily relies on the assumption
on boundedness of the sequence of multipliers and cannot be applied in the case when the sequence $\{ \lambda_n \}$ does
not have at least a bounded subsequence. However, there are primal-dual augmented Lagrangian methods for which one can
prove convergence of the sequence $\{ x_n \}$ to the set of globally optimal solutions of the problem $(\mathcal{P})$
even in the case when the norm of the multipliers $\lambda_n$ increases unboundedly with iterations. Augmented
Lagrangian methods with the so-called \textit{conditional multiplier updating} (see \cite{ConnGouldToint},
\cite[Algorithm~3]{LuoSunLi}, \cite[Algorithm~3]{LuoSunWu}, \cite[Algorithm~3]{LuoWuChen2012}) and the algorithms from
\cite{WangLi2009} are examples of such methods. The main idea behind these methods consists in designing multiplier and
penalty parameter updating rules in such a way as to ensure that an increase of the norm of the multipliers 
$\| \lambda_n \|$ is compensated by a sufficient increase of the penalty parameter $c_n$, so that one can prove that
\begin{equation} \label{eq:MethodsWithUnbounded_Mult_PenParam}
  \lim_{n \to \infty} \dist(G(x_n), K) = 0, \quad 
  \lim_{n \to \infty} \Phi(G(x_n), \lambda_n, c_n) = 0
\end{equation}
even if $\| \lambda_n \| \to + \infty$ as $n \to \infty$. It is possible to extend convergence analysis of these methods
to our axiomatic augmented Lagrangian setting by either imposing some restrictive assumptions on the function $\Phi$ or
directly assuming that relations \eqref{eq:MethodsWithUnbounded_Mult_PenParam} hold true. We do not present such 
extension here and leave it as a problem for future research.
}
\end{remark}

\subsection{Convergence analysis of the method}

Let us now turn to convergence analysis. We start with two simple observations. The first one can be viewed as 
a generalization of some existing results (e.g. \cite[Theorem~5]{Gasimov} and 
\cite[Theorem~1]{CordovaOliveiraSagastizabal}) to the case of general cone constrained problems and arbitrary augmented
Lagrangians satisfying a certain assumption.

\begin{lemma} \label{lem:AuxProblemAtCurrentIter}
Let the function $y \mapsto \Phi(y, \lambda, c)$ be non-decreasing with respect to the binary relation $\preceq$ for any
$\lambda \in \Lambda$ and $c > 0$, and let $\{ (x_n, \lambda_n, c_n) \}$ be the sequence generated by the model
augmented Lagrangian method. Then for any $n \in \mathbb{N}$ the point $x_n$ is an $\varepsilon_n$-optimal solution of
the problem
\begin{equation} \label{prob:AuxProblemAtCurrentIter}
  \min \: f(x) \quad \text{subject to } G(x) \preceq G(x_n), \quad x \in Q,
\end{equation}
Moreover, if $G(x_n) = 0$, then $x_n$ is an $\varepsilon_n$-optimal solution of the problem $(\mathcal{P})$.
\end{lemma}

\begin{proof}
Suppose by contradiction that the claim of the lemma is false. Then one can find a feasible point $x$ of problem
\eqref{prob:AuxProblemAtCurrentIter} such that $f(x_n) > f(x) + \varepsilon_n$. Hence by our assumption on the function
$\Phi$ one gets
\begin{align*}
  \mathscr{L}(x, \lambda_n, c_n) = f(x) + \Phi(G(x), \lambda_n, c_n)
  &< f(x_n) - \varepsilon_n + \Phi(G(x_n), \lambda_n, c_n) 
  \\
  &= \mathscr{L}(x_n, \lambda_n, c_n) - \varepsilon_n,
\end{align*}
which contradicts the definition on $x_n$. 

Suppose now that $G(x_n) = 0$. Recall that the inequality $G(x) \preceq G(x_n)$ means that $G(x_n) - G(x) \in -K$ or, 
equivalently, $G(x) \in K + G(x_n) = K$. Therefore the feasible region of the problem $(\mathcal{P})$ coincides with
the feasible region of problem \eqref{prob:AuxProblemAtCurrentIter}, which implies that an $\varepsilon_n$-optimal
solution of this problem is also an $\varepsilon_n$-optimal solution of the problem $(\mathcal{P})$.
\end{proof}

\begin{remark}
The assumption that the function $y \mapsto \Phi(y, \lambda, c)$ is non-decreasing is satisfied for all particular
augmented Lagrangians presented in this paper, except for the one from Example~\ref{ex:ExpTypeAugmLagr_2OrderCone}.
\end{remark}

The second observation is connected with the augmented dual function. Recall that if the function $\Phi$ satisfies
assumption $(A1)$, then $\mathscr{L}(x, \lambda, c) \le f(x)$ for any feasible point $x$. Therefore, the following
result holds true.

\begin{lemma} \label{lem:AugmLagrValueSeq}
Let assumption $(A1)$ be valid. Then 
\[
  \Theta(\lambda_n, c_n) \le \mathscr{L}(x_n, \lambda_n, c_n) 
  \le \Theta(\lambda_n, c_n) + \varepsilon_n \le f_* + \varepsilon_n
  \quad \forall n \in \mathbb{N}.
\]
\end{lemma}

Thus, if the sequence $\{ \Theta(\lambda_n, c_n) \}$ is bounded below, then the corresponding sequence 
$\{ \mathscr{L}(x_n, \lambda_n, c_n) \}$ is bounded. Let us analyse how these sequences behave in the limit. As we will
see below, this analysis is a key ingredient in the global convergence theory of augmented Lagrangian methods.

The following cumbersome technical result, which can be viewed as a partial generalization of
Theorem~\ref{thrm:DualOptVal_vs_OptValFunc}, plays a key role in our convergence analysis of the model augmented
Lagrangian method. The proof of this result is very similar to the proof of Theorem~\ref{thrm:DualOptVal_vs_OptValFunc},
and we include it only for the sake of completeness.

\begin{lemma} \label{lem:DualValuesConverg}
Let assumptions $(A1)$, $(A13)_s$, and $(A14)_s$ hold true. Suppose also that a sequence 
$\{ (x_n, \lambda_n, c_n) \} \subset Q \times \dom \Theta$ is such that:
\begin{enumerate}
\item{$\dist(G(x_n, K) \to 0$ as $n \to \infty$,}

\item{the sequence $\{ \lambda_n \}$ is bounded,}

\item{$c_n \to + \infty$ as $n \to \infty$,}

\item{the sequence $\{ \mathscr{L}(x_n, \lambda_n, c_n) - \Theta(\lambda_n, c_n) \}$ is bounded above.}
\end{enumerate}
Let finally $\varepsilon_* = \limsup_{n \to \infty} (\mathscr{L}(x_n, \lambda_n, c_n) - \Theta(\lambda_n, c_n))$.
Then
\begin{align} \label{eq:DualFunc_MinimizingSeq}
  \vartheta_* - \varepsilon_* \le \liminf_{n \to \infty} \Theta(\lambda_n, c_n)
  &\le \limsup_{n \to \infty} \Theta(\lambda_n, c_n) \le \vartheta_*
  \\ \label{eq:AugmLagrValueAlongMinSeq}
  \vartheta_* \le \liminf_{n \to \infty} \mathscr{L}(x_n, \lambda_n, c_n)
  &\le \limsup_{n \to \infty} \mathscr{L}(x_n, \lambda_n, c_n) \le \vartheta_* + \varepsilon_*
  \\ \label{eq:ObjFuncValueAlongMinSeq}
  \vartheta_* \le \liminf_{n \to \infty} f(x_n) &\le \limsup_{n \to \infty} f(x_n) \le \vartheta_* + \varepsilon_*.
\end{align}
where $\vartheta_* = \min\big\{ f_*, \liminf_{p \to 0} \beta(p) \big\}$.
\end{lemma}

\begin{proof}
Note that the upper estimate for the limit superior in \eqref{eq:ObjFuncValueAlongMinSeq} follows directly from 
the upper estimate in \eqref{eq:AugmLagrValueAlongMinSeq} and assumption $(A13)_s$. In turn, the lower estimate for the
limit inferior in \eqref{eq:ObjFuncValueAlongMinSeq} follows directly from the fact that $\dist(G(x_n), K) \to 0$ as 
$n \to \infty$.

Indeed, if some subsequence $\{ x_{n_k} \}$ is feasible for the problem $(\mathcal{P})$, then obviously 
$\liminf_{k \to \infty} f(x_{n_k}) \ge f_*$. In turn, if each member of a subsequence $\{ x_{n_k} \}$ is infeasible
for the problem $(\mathcal{P})$, then $f(x_{n_k}) \ge \beta(p_k)$ for any $p_k \in Y$ such that 
$G(x_{n_k}) - p_k \in K$. Since $\dist(G(x_{n_k}), K) \to 0$ as $k \to \infty$, one can choose $p_k \in Y$, 
$k \in \mathbb{N}$, such that $p_k \to 0$ as $k \to \infty$. Consequently, one has
\[
  \liminf_{k \to \infty} f(x_{n_k}) \ge \liminf_{k \to \infty} \beta(p_k) \ge \liminf_{p \to 0} \beta(p),
\]
which obviously implies that the the lower estimate in \eqref{eq:ObjFuncValueAlongMinSeq} holds true.

Thus, we need to prove only inequalities \eqref{eq:DualFunc_MinimizingSeq} and \eqref{eq:AugmLagrValueAlongMinSeq}. Let
us consider two cases.

\textbf{Case I.} Suppose that there exists a subsequence $\{ x_{n_k} \}$ that is feasible for the problem
$(\mathcal{P})$. Then with the use of Lemma~\ref{lem:Assumpt(A16)} one gets
\begin{align*}
  \liminf_{k \to \infty} \mathscr{L}(x_{n_k}, \lambda_{n_k}, c_{n_k}) 
  &\ge \liminf_{k \to \infty} \Big( f(x_{n_k}) 
  + \inf_{\lambda \in B(0, \eta) \cap \Lambda} \inf_{y \in K} \Phi(y, \lambda, c_{n_k}) \Big) 
  \\
  &\ge \liminf_{k \to \infty} f(x_{n_k}) \ge f_*,
\end{align*}
where $\eta = \sup_k \| \lambda_{n_k} \|$.

\textbf{Case II.} Suppose now that there exists a subsequence $\{ x_{n_k} \}$ such that $G(x_{n_k}) \notin K$ 
for all $k \in \mathbb{N}$. Then with the use of assumption $(A13)_s$ one gets
\[
  \liminf_{k \to \infty} \mathscr{L}(x_{n_k}, \lambda_{n_k}, c_{n_k}) \ge 
  \liminf_{k \to \infty} f(x_{n_k}) \ge \liminf_{k \to \infty} \beta(p_k) \ge \liminf_{p \to 0} \beta(p),
\]
where $\{ p_k \} \subset Y$ is any sequence such that $G(x_{n_k}) - p_k \in K$ for all $k \in \mathbb{N}$ and 
$\| p_k \| \to 0$ as $k \to \infty$ (clearly, such sequence exists, since $\dist(G(x_n), K) \to 0$ as $n \to \infty$). 

Combining the two cases one gets that the lower estimates for the limit inferiors in \eqref{eq:DualFunc_MinimizingSeq}
and \eqref{eq:AugmLagrValueAlongMinSeq} hold true. Let us now prove the upper estimates for the limit superiors. Note
that the upper estimate in \eqref{eq:AugmLagrValueAlongMinSeq} follows directly from the upper estimate in
\eqref{eq:DualFunc_MinimizingSeq} and Lemma~\ref{lem:AugmLagrValueSeq}. Therefore, it suffies to prove only the upper
estimate for the limit superior of $\{ \Theta(\lambda_n, c_n) \}$.

By Proposition~\ref{prp:WeakDuality} one has $\Theta(\lambda_n, c_n) \le f_*$ for all $n \in \mathbb{N}$, which implies
that $\limsup_{n \to \infty} \Theta(\lambda_n, c_n) \le f_*$. If $\beta_* := \liminf_{p \to 0} \beta(p) \ge f_*$, then
the proof is complete. Therefore, suppose that $\beta_* < f_*$.

By the definition of limit inferior there exists a sequence $\{ p_k \} \subset Y$ such that $p_k \to 0$ and 
$\beta(p_k) \to \beta_*$ as $k \to \infty$. Let $\{ t_n \}$ be the sequence from assumption $(A14)_s$. Since $p_k \to 0$
as $k \to \infty$, there exists a subsequence $\{ p_{k_n} \}$ such that $\| p_{k_n} \| \le t_n$ for all 
$n \in \mathbb{N}$. 

By the definition of the optimal value function $\beta$ for any $n \in \mathbb{N}$ one can find $x_n \in Q$ such that
$G(x_n) - p_{k_n} \in K$ (which implies that $\dist(G(x_n), K) \le t_n$) and $f(x_n) \le \beta(p_{k_n}) + 1/n$, 
if $\beta(p_{k_n}) > - \infty$, and $f(x_n) \le -n$ otherwise. By applying assumption $(A14)_s$, one gets that
\begin{align*}
  \limsup_{n \to \infty} \Theta(\lambda_n, c_n) \le \limsup_{n \to \infty} \mathscr{L}(z_n, \lambda_n, c_n)
  &= \limsup_{n \to \infty} f(z_n) 
  \\
  &= \lim_{n \to \infty} \beta(p_{k_n}) = \liminf_{p \to 0} \beta(p),
\end{align*}
which means that the upper estimate in \eqref{eq:DualFunc_MinimizingSeq} holds true.
\end{proof}

\begin{remark}
The claim of the lemma above (as well as the claim of Theorem~\ref{thrm:DualValues_AugmLagr_Limit} based on that lemma)
remains to hold true, if only restricted versions of assumptions $(A13)_s$ and $(A14)_s$ are satisfied, and one
additionally assumes that the projection of the set $G(Q)$ onto the cone $K$ is bounded (see 
Remark~\ref{rmrk:ConstraintsBoundedBelow}).
\end{remark}

\begin{corollary} \label{crlr:DualFunc_MinimizingSeq_ExactLimit}
Let assumptions $(A1)$, $(A13)_s$, and $(A14)_s$ hold true. Let also a sequence 
$\{ (\lambda_n, c_n) \} \subset \dom \Theta$ be such that the sequence $\{ \lambda_n \}$ is bounded and 
$c_n \to + \infty$ as $n \to \infty$. Suppose finally that there exists a sequence $\{ z_n \} \subset Q$ such that
$\dist(G(z_n), K) \to 0$ and $(\mathscr{L}(z_n, \lambda_n, c_n) - \Theta(\lambda_n, c_n)) \to 0$ as $n \to \infty$.
Then
\begin{equation} \label{eq:DualFunc_MinimizingSeq_ExactLimit}
  \lim_{n \to \infty} \Theta(\lambda_n, c_n) = \min\big\{ f_*, \liminf_{p \to 0} \beta(p) \big\}
\end{equation}
\end{corollary}

\begin{proof}
Apply the previous lemma to the sequence $\{ (x_n, \lambda_n, c_n) \}$ with $x_n = z_n$ for all $n \in \mathbb{N}$.
\end{proof}

\begin{remark}
The corollary above strengthens Lemma~\ref{lem:DualValuesConverg}. Namely, it states that if there exists a sequence 
$\{ z_n \}$ satisfying the assumptions of the corollary, then one can actually replace the lower and upper estimates
\eqref{eq:DualFunc_MinimizingSeq} with equality \eqref{eq:DualFunc_MinimizingSeq_ExactLimit}. Note also that by 
the definition of augmented dual function there \textit{always} exists a sequence $\{ z_n \} \subset Q$ such that 
$(\mathscr{L}(z_n, \lambda_n, c_n) - \Theta(\lambda_n, c_n)) \to 0$ as $n \to \infty$. The main assumption of 
the corollary is that one can find a sequence $\{ z_n \}$ not only satisfying this condition, but also such that
$\dist(G(x_n), K) \to 0$ as $n \to \infty$.
\end{remark}

Let us also provide necessary and sufficient conditions for the sequence $\{ \dist(G(x_n), K) \}$ to converge to zero.

\begin{lemma} \label{lem:FeasibilityMeasConvergence}
Let assumptions $(A1)$, $(A7)$, and $(A12)_s$ hold true and a sequence 
$\{ (x_n, \lambda_n, c_n) \} \subset Q \times \dom \Theta$ be such that:
\begin{enumerate}
\item{the sequence $\{ \lambda_n \}$ is bounded,}

\item{$c_n \to + \infty$ as $n \to \infty$,}

\item{the sequence $\{ \mathscr{L}(x_n, \lambda_n, c_n) - \Theta(\lambda_n, c_n) \}$ is bounded above,}

\item{there exists $\tau > 0$ such that the function $\inf_n \Phi(G(\cdot), \lambda_n, \tau)$ is bounded below on $Q$.}
\end{enumerate}
Then for the sequence $\{ \dist(G(x_n), K) \}$ to converge to zero it is sufficient that the sequence $\{ f(x_n) \}$ is
bounded below. This condition becomes necessary, when $\liminf_{p \to 0} \beta(p) > - \infty$. 
\end{lemma}

\begin{proof}
If only a finite number of elements of the sequence $\{ x_n \}$ is infeasible for the problem $(\mathcal{P})$, then
the claim of the lemma is trivial, since in this case $G(x_n) \in K$ and $f(x_n) \ge f_*$ for any sufficiently large
$n$. Therefore, replacing, if necessary, the sequence $\{ x_n \}$ with its subsequence one can suppose that 
$G(x_n) \notin K$ for all $n \in \mathbb{N}$.

\textbf{Sufficiency.} Denote $\varepsilon_n = \mathscr{L}(x_n, \lambda_n, c_n) - \Theta(\lambda_n, c_n)$ and
\[
  \overline{\varepsilon} = \sup_n \varepsilon_n < + \infty, \quad
  \underline{f} = \inf_n f(x_n) > - \infty, \quad
  \underline{\Phi} = \inf_n \inf_{x \in Q} \Phi(G(\cdot), \lambda_n, \tau) > - \infty.
\]
By Lemma~\ref{lem:AugmLagrValueSeq} one has 
\begin{equation} \label{eq:AugmLagrSeqValues_UpperBound}
  \mathscr{L}(x_n, \lambda_n, c_n) \le \Theta(\lambda_n, c_n) + \varepsilon_n \le f_* + \overline{\varepsilon}
  \quad \forall n \in \mathbb{N},
\end{equation}
that is, the sequence $\{ \mathscr{L}(x_n, \lambda_n, c_n) \}$ is bounded above.

Fix any $r > 0$. Due to the boundedness of the sequence $\{ \lambda_n \}$ and assumption $(A12)_s$ there exists 
$t(r) \ge \tau$ such that for any $c \ge t(r)$, $n \in \mathbb{N}$, and $x \in E_n$ one has
\[
  \Phi(G(x), \lambda_n, c) - \Phi(G(x), \lambda_n, \tau) 
  \ge f_* - \underline{f} + \overline{\varepsilon} + 1 - \underline{\Phi},
\]
which implies that $\Phi(G(x), \lambda_n, c) \ge f_* - \underline{f} + \overline{\varepsilon} + 1$, where
\[
  E_n := \Big\{ x \in Q \Bigm| \dist(G(x), K) \ge r, \: \Phi(G(x), \lambda_n, \tau) < + \infty \Big\}
\]
Therefore, for any $n \in \mathbb{N}$ such that $c_n \ge t(r)$ and $\dist(G(x_n), K) \ge r$ one has
\[
  \mathscr{L}(x_n, \lambda_n, c_n) = f(x_n) + \Phi(G(x_n), \lambda_n, c_n) \ge f_* + \overline{\varepsilon} + 1,
\]
which contradicts \eqref{eq:AugmLagrSeqValues_UpperBound}. Consequently, for any $n \in \mathbb{N}$ such that 
$c_n \ge t(r)$ one has $\dist(G(x_n), K) < r$, which implies that $\dist(G(x_n), K) \to 0$ as $n \to \infty$.

\textbf{Necessity.} Suppose by contradiction that $\dist(G(x_n), K) \to 0$, but the sequence $\{ f(x_n) \}$ is
unbounded below. Then for any sequence $\{ p_n \} \subset Y$ such that $G(x_n) - p_n \in K$ and $p_n \to 0$ as 
$n \to \infty$ (at least one such sequence exists, since that $\dist(G(x_n), K) \to 0$ as $n \to \infty$) one has
\[
  - \infty = \liminf_{n \to \infty} f(x_n) \ge \liminf_{n \to \infty} \beta(p_n) \ge \liminf_{p \to 0} \beta(p),
\]
which contradicts our assumption.
\end{proof}

\begin{remark}
{(i)~The last assumption of the lemma is satisfied, in particular, if for any bounded set $\Lambda_0 \subset \Lambda$
there exists $\tau > 0$ such that the function $\inf_{\lambda \in \Lambda_0} \Phi(\cdot, \lambda, \tau)$ is bounded
below on $Y$. This assumption is satisfied for all particular examples of augmented Lagrangians from
Section~\ref{sect:Examples}, except for He-Wu-Meng's augmented Lagrangian (Example~\ref{ex:HeWuMeng}) under appropriate
additional assumptions. Namely, in the case of Rockafellar-Wet's augmented Lagrangian
(Example~\ref{ex:RockafellarWetsAugmLagr}) one needs to additionally assume that 
$\sigma(\cdot) \ge \sigma_0 \| \cdot \|^{\alpha}$ for some $\sigma_0 > 0$ and $\alpha \ge 1$, while in the case of the
(penalized) exponential-type augmented Lagrangians (Examples~\ref{ex:ExpTypeAugmLagr}, 
\ref{ex:PenalizedExpTypeAugmLagr}, \ref{ex:ExpTypeAugmLagr_2OrderCone}, \ref{ex:ExpTypeAugmLagr_SemiDef}, and 
\ref{ex:PenalizedExpTypeAugmLagr_SemiDef}) and the hyperbolic-type augmented Lagrangian
(Example~\ref{eq:HyperbolicAugmLagr}) one needs to additionally assume that the function $\phi$/$\psi$ is bounded 
below. In all other example the assumption on the boundedness below of the function $\Phi$ is satisfied without any 
additional assumptions.
}

{(ii)~The last assumption of the lemma is satisfied for He-Wu-Meng's augmented Lagrangian and the (penalized)
exponential/hyperbolic-type augmented Lagrangians with unbounded below functions $\phi$/$\psi$, if the projection of 
the set $G(Q)$ onto $K$ is bounded. In particular, in the case of inequality constrained problems it is sufficient to
suppose that the functions $g_i$, defining the constraints $g_i(x) \le 0$, are bounded below. As was noted in 
Remark~\ref{rmrk:ConstraintsBoundedBelow}, this assumption is not restrictive from the theoretical point of view.
}

{(iii)~It should be noted that we used the last assumption of the lemma in order to implicitly prove that assumption 
$(A12)_s$ implies that
\begin{equation} \label{eq:BasicAssumpt_Coercivity}
  \lim_{c \to \infty} \inf_{n \in \mathbb{N}} \inf\Big\{ \Phi(y, \lambda_n, c) \Bigm|
  y \in Y \colon \dist(y, K) \ge r \Big\} =  + \infty
\end{equation}
for any $r > 0$. Therefore one might wonder whether it would be better to formulate \eqref{eq:BasicAssumpt_Coercivity}
as a basic assumption and use it instead of assumption $(A12)_s$ and the assumption on the boundedness below of the
function $\Phi(G(\cdot), \lambda_n, \tau)$. Note, however, that in most particular cases the boundedness below of 
the function $\Phi$ is a necessary condition for \eqref{eq:BasicAssumpt_Coercivity} to hold true. In particular, one can
easily check that if a separable augmented Lagrangian (see \eqref{eq:SeparableAugmLagr_IneqConstr}) satisfies condition 
\eqref{eq:BasicAssumpt_Coercivity}, then each function $\Phi_i$ is bounded below. That is why we opted to use assumption
$(A12)_s$ along with the assumption on the boundedness below of the function $\Phi$ instead of condition
\eqref{eq:BasicAssumpt_Coercivity}.
}
\end{remark}

Now we are ready to estimate the limit of the sequence $\{ \mathscr{L}(x_n, \lambda_n, c_n) \}$ and the corresponding
sequences $\{ \Theta(\lambda_n, c_n)$ and $\{ f(x_n) \}$ of the augmented dual and objective functions' values for 
sequences $\{ (x_n, \lambda, c_n) \}$ generated by the model augmented Lagrangian method. Recall that $\Theta_*$ is the
optimal value of the augmented dual problem.

\begin{theorem}[main convergence theorem] \label{thrm:DualValues_AugmLagr_Limit}
Let $\{ (x_n, \lambda_n, c_n) \}$ be the sequence generated by the model augmented Lagrangian method, and suppose that
the following conditions are satisfied:
\begin{enumerate}
\item{assumptions $(A1)$, $(A7)$, $(A12)_s$--$(A14)_s$, and $(A15)$ hold true;
}

\item{assumptions $(B1)$--$(B4)$ hold true;}

\item{the sequence $\{ \varepsilon_n \}$ is bounded and $\limsup_{n \to \infty} \varepsilon_n = \varepsilon_*$;}

\item{for any bounded set $\Lambda_0 \subset \Lambda$ there exists $\tau > 0$ such that the function 
$\inf_{\lambda \in \Lambda_0} \Phi(G(\cdot), \lambda, \tau)$ is bounded below on $Q$.} 
\end{enumerate}
Then in the case when the sequence $\{ c_n \}$ is bounded one has $\dist(G(x_n), K) \to 0$ as $n \to \infty$, and in the
case when the sequence $\{ c_n \}$ is unbounded one has $\dist(G(x_n), K) \to 0$ if and only if the sequence 
$\{ f(x_n) \}$ is bounded below. Furthermore, if the sequence $\{ f(x_n) \}$ is bounded below, then
\begin{align} \label{eq:DualFuncLimits}
  \Theta_* - \varepsilon_* \le \liminf_{n \to \infty} \Theta(\lambda_n, c_n)
  &\le \limsup_{n \to \infty} \Theta(\lambda_n, c_n) \le \Theta_*,
  \\ \label{eq:AugmLagrValuesLimits}
  \Theta_* \le \liminf_{n \to \infty} \mathscr{L}(x_n, \lambda_n, c_n)
  &\le \limsup_{n \to \infty} \mathscr{L}(x_n, \lambda_n, c_n) \le \Theta_* + \varepsilon_*
  \\ \label{eq:ObjFuncValuesLimits}
  \Theta_* \le \liminf_{n \to \infty} f(x_n)
  &\le \limsup_{n \to \infty} f(x_n) \le \Theta_* + \varepsilon_*.
\end{align}
\end{theorem}

\begin{proof}
If the sequence $\{ c_n \}$ is unbounded, then by assumption $(B4)$ one has $c_n \to + \infty$ as $n \to \infty$, and 
the claim of the theorem follows directly from Lemmas~\ref{lem:DualValuesConverg} and
\ref{lem:FeasibilityMeasConvergence} (the fact that $\lim_{p \to 0} \beta(p) > - \infty$ follows directly from
Remark~\ref{rmrk:OptValFuncFiniteLim}). Therefore, suppose that the sequence of penalty parameters $\{ c_n \}$ is
bounded. Note that inequalities \eqref{eq:DualFuncLimits} in  this case follow from the first inequality in
\eqref{eq:AugmLagrValuesLimits} and the definition of $\Theta_*$. Note further that by assumption $(B3)$ one has
$\dist(G(x_n), K) \to 0$ and $\Phi(G(x_n), \lambda_n, c_n) \to 0$ as $n \to \infty$. Consequently, one has
\[
  \liminf_{n \to \infty} \mathscr{L}(x_n, \lambda_n, c_n) = \liminf_{n \to \infty} f(x_n), \quad
  \limsup_{n \to \infty} \mathscr{L}(x_n, \lambda_n, c_n) = \limsup_{n \to \infty} f(x_n).
\]
Thus, it is sufficient to prove either of the inequalities \eqref{eq:AugmLagrValuesLimits} and 
\eqref{eq:ObjFuncValuesLimits}. We divide the rest of the proof into two parts.

\textbf{Part 1. Lower estimate.} Let $\{ x_{n_k} \}$ be any subsequence such that
\[
  \lim_{k \to \infty} f(x_{n_k}) = \liminf_{n \to \infty} f(x_n)
\]
(at least one such subsequence exists by the definition of limit inferior). Suppose at first that there exists 
a subsequence of the sequence $\{ x_{n_k} \}$, which we denote again by $\{ x_{n_k} \}$, that is feasible for 
the problem $(\mathcal{P})$. Then $f(x_{n_k}) \ge f_*$ for all $k \in \mathbb{N}$ and the lower estimate for the limit
inferior in \eqref{eq:ObjFuncValuesLimits} holds true by Proposition~\ref{prp:WeakDuality}.

Suppose now that $x_{n_k}$ is infeasible for the problem $(\mathcal{P})$ for all $k$ greater than some 
$k_0 \in \mathbb{N}$. Since $\dist(G(x_n), K) \to 0$ as $n \to \infty$, for any $k \in k_0$ one can find
$p_k \in Y$ such that $G(x_{n_k}) - p_k \in K$ and $p_k \to 0$ as $k \to \infty$. Consequently, 
$f(x_{n_k}) \ge \beta(p_k)$ for all $k \ge k_0$ and
\[
  \lim_{k \to \infty} f(x_{n_k}) \ge \liminf_{k \to \infty} \beta(p_k) \ge 
  \liminf_{p \to 0} \beta(p),
\]
which by Theorem~\ref{thrm:DualOptVal_vs_OptValFunc} implies that the lower estimate for the limit inferior in
\eqref{eq:ObjFuncValuesLimits} is valid.

\textbf{Part 2. Upper estimate.} Suppose at first that $f_* \le \liminf_{p \to 0} \beta(p)$. Then by 
Lemma~\ref{lem:AugmLagrValueSeq} and Theorem~\ref{thrm:DualOptVal_vs_OptValFunc} one has
\[
  \limsup_{n \to \infty} \mathscr{L}(x_n, \lambda_n, c_n) \le f_* + \varepsilon_*
  = \Theta_* + \varepsilon_*,
\]
that is, the upper estimate for the limit superior in \eqref{eq:AugmLagrValuesLimits} holds true.

Let us now consider the case $f_* > \liminf_{p \to 0} \beta(p) =: \beta_*$. By the definition of limit inferior 
there exists $\{ p_n \} \subset Y$ such that $p_n \to 0$ and $\beta(p_n) \to \beta_*$ as $n \to \infty$. 
Note that $\beta_* > - \infty$ by Remark~\ref{rmrk:OptValFuncFiniteLim} and, therefore, one can suppose that
$\beta(p_n) > - \infty$ for all $n \in \mathbb{N}$.

By the definition of the optimal value function one can find a sequence $\{ z_n \} \subset Q$ such that 
$G(z_n) - p_n \in K$ and $f(z_n) \le \beta(p_n) + 1/n$ for all $n \in \mathbb{N}$. Hence, in particular, 
$\dist(G(z_n), K) \to 0$ as $n \to \infty$, which by assumption $(A15)$ implies that 
$\limsup_{n \to \infty} \Phi(G(z_n), \lambda_n, c_n) \le 0$. Consequently, one has
\[
  \limsup_{n \to \infty} \Theta(\lambda_n, c_n) \le \limsup_{n \to \infty} \mathscr{L}(z_n, \lambda_n, c_n)
  \le \lim_{n \to \infty} f(z_n) = \lim_{n \to \infty} \beta(p_n) = \beta_*.
\]
Recall that by the definition of $x_n$ one has 
$\mathscr{L}(x_n, \lambda_n, c_n) \le \Theta(\lambda_n, c_n) + \varepsilon_n$. Therefore the inequality above along
with Theorem~\ref{thrm:DualOptVal_vs_OptValFunc} imply that the upper estimate for the limit superior 
in \eqref{eq:AugmLagrValuesLimits} holds true.
\end{proof}

\begin{remark} \label{rmrk:AugmLagrMethod_LimitAlongSubseq}
{(i)~The previous theorem can be slightly generalized in the following way. Namely, suppose that assumption $(B2)$ does 
not hold true, but there exists a bounded subsequence $\{ \lambda_{n_k} \}$. Then the claim of
Theorem~\ref{thrm:DualValues_AugmLagr_Limit} holds true for the corresponding subsequence 
$\{ (x_{n_k}, \lambda_{n_k}, c_{n_k}) \}$. In the case when the sequence $\{ c_n \}$ is unbounded, one simply needs
to apply Lemmas~\ref{lem:DualValuesConverg} and \ref{lem:FeasibilityMeasConvergence} to this subsequence. In the case
when the sequence $\{ c_n \}$ is bounded, one just needs to repeat the proof of the theorem with the sequence
$\{ (x_n, \lambda_n, c_n) \}$ replaced by the subsequence $\{ (x_{n_k}, \lambda_{n_k}, c_{n_k}) \}$.
}

{(ii)~Note that from the proof of Theorem~\ref{thrm:DualValues_AugmLagr_Limit} it follows that the sequence 
$\{ f(x_n) \}$ is always bounded below in the case when the sequence $\{ c_n \}$ is bounded.
}

{(iii)~In many papers on augmented Lagrangians and augmented Lagrangians methods, it is assumed by default that 
the function $f$ is bounded below on the set $Q$ (see \cite[Assumption~1]{LuoSunLi} \cite[Assumption~2.3]{LuoSunWu},
\cite[Assumption~1]{LiuYang2008} \cite[Assumption~1]{WangLi2009}, \cite[Assumption~1]{LuoWuChen2012},
\cite[Assumption~(1)]{WangLiuQu}, etc.). From the theoretical point of view this assumption is not restrictive, since
one can always replace the objective function $f$ with $e^{f(\cdot)}$. This assumption ensures that 
$\dist(G(x_n), K) \to 0$ as $n \to \infty$ for sequences $\{ (x_n, \lambda_n, c_n) \}$ generated by the model augmented
Lagrangian method, regardless of whether the sequence of penalty parameters is bounded or not. Furthermore, in the case
when the sequence $\{ c_n \}$ increases unboundedly, this assumptions guarantees that estimates
\eqref{eq:DualFuncLimits} can be replaced with equality \eqref{eq:DualFunc_MinimizingSeq_ExactLimit} by
Corollary~\ref{crlr:DualFunc_MinimizingSeq_ExactLimit} and Lemma~\ref{lem:FeasibilityMeasConvergence}.
}
\end{remark}

\begin{corollary}
Let the assumptions of Theorem~\ref{thrm:DualValues_AugmLagr_Limit} hold true, and suppose that $\varepsilon_n \to 0$
as $n \to \infty$. Then
\[
  \lim_{n \to \infty} \Theta(\lambda_n, c_n) = \lim_{n \to \infty} \mathscr{L}(x_n, \lambda_n, c_n)
  = \lim_{n \to \infty} f(x_n) = \Theta_*.
\]
\end{corollary}

\subsection{Primal convergence}

Now we are ready to prove general theorems on convergence of the sequence $\{ x_n \}$ generated by the model augmented
Lagrangian method. Denote by $\Delta_* = f_* - \Theta_*$ the duality gap between the primal problem $(\mathcal{P})$ and 
the augmented dual problem $(\mathcal{D})$.

\begin{theorem}[primal convergence vs. duality gap] \label{thrm:GlobalConv_vs_DualityGap}
Let the assumptions of Theorem~\ref{thrm:DualValues_AugmLagr_Limit} be valid, the sequence $\{ f(x_n) \}$ be bounded
below, and the functions $f$ and $\dist(G(\cdot), K)$ be lsc on $Q$. Then the sequence $\{ x_n \}$ has limit points,
only if $\Delta_* \le \varepsilon_*$ (that is, the  duality gap is smaller than the tolerance with which the augmented
Lagrangian subproblems are solved). Furthermore, all limit points of the sequence $\{ x_n \}$ (if such points exist)
are $(\varepsilon_* - \Delta_*)$-optimal solutions of the problem $(\mathcal{P})$.
\end{theorem}

\begin{proof}
Suppose that there exists a limit point $x_*$ of the sequence $\{ x_n \}$, i.e. there exists a subsequence 
$\{ x_{n_k} \}$ that converges to $x_*$. Recall that $\dist(G(x_n), K) \to 0$ as $n \to \infty$ and
\[
  \limsup_{n \to \infty} f(x_n) \le \Theta_* + \varepsilon_*
\]
by Theorem~\ref{thrm:DualValues_AugmLagr_Limit}. Hence taking into account the semicontinuity assumptions 
one can conclude that $\dist(G(x_*), K) = 0$, i.e. $x_*$ is feasible for the problem $(\mathcal{P})$, and 
\[
  f_* \le f(x_*) \le \Theta_* + \varepsilon_* \le f_* + \varepsilon_* - \Delta_*.
\]  
Therefore, $\Delta_* \le \varepsilon_*$ and $x_*$ is an $(\varepsilon_* - \Delta_*)$-optimal solution
of the problem $(\mathcal{P})$.
\end{proof}

\begin{corollary}[primal convergence vs. zero duality gap]
If under the assumptions of the previous theorem $\varepsilon_n \to 0$ as $n \to \infty$, then for the sequence 
$\{ x_n \}$ to have limit points it is necessary that there is zero duality gap between the primal problem
$(\mathcal{P})$ and the augmented dual problem $(\mathcal{D})$. Furthermore, in this case all limit points of 
the sequence $\{ x_n \}$ (if such points exist) are globally optimal solutions of the problem $(\mathcal{P})$.
\end{corollary}

In the case when the space $X$ is reflexive (in particular, finite dimensional), we can prove a somewhat stronger 
result. Namely, we can show that if the zero duality gap property is not satisfied, then the sequence $\{ x_n \}$
necessarily escapes to infinity as $n \to \infty$.

\begin{theorem}[boundedness vs. duality gap] 
Let the assumptions of Theorem~\ref{thrm:DualValues_AugmLagr_Limit} be valid, the space $X$ be reflexive, the set $Q$ 
be weakly sequentially closed, the functions $f$ and $\dist(G(\cdot), K)$ be weakly sequentially lsc on $Q$, 
the sequence $\{ f(x_n) \}$ be bounded below. Then the following statements hold true:
\begin{enumerate}
\item{for the sequence $\{ x_n \}$ to have a bounded subsequence it is necessary that $\Delta_* \le \varepsilon_*$;}

\item{all weakly limit points of the sequence $\{ x_n \}$ (if such points exist at all) are 
$(\varepsilon_* - \Delta_*)$-optimal solutions of the problem $(\mathcal{P})$;}

\item{if $\varepsilon_n \to 0$ as $n \to \infty$, then for the sequence $\{ x_n \}$ to have a bounded subsequence 
it is necessary that the zero duality gap property holds true; furthermore, in this case all weakly limit points of 
the sequence $\{ x_n \}$ are globally optimal solutions of the problem $(\mathcal{P})$.}
\end{enumerate}
\end{theorem}

\begin{proof}
Bearing in mind the fact that any bounded sequence in a reflexive Banach space has a weakly convergent subsequence 
and arguing in the same way as in the proof of Theorem~\ref{thrm:GlobalConv_vs_DualityGap} one can easily verify that 
all claims of this theorem hold true.
\end{proof}

\subsection{Dual convergence}

Let us now turn to analysis of dual convergence, that is, convergence of the sequence of multipliers $\{ \lambda_n \}$
or, more precisely, convergence of the dual sequence $\{ (\lambda_n, c_n) \}$. Although convergence of multipliers for
some particular augmented Lagrangian methods can be studied, even in the case when the sequence of multipliers 
$\{ c_n \}$ increases unboundedly, with the use of optimality conditions, only convergence of the whole sequence 
$\{ (\lambda_n, c_n) \}$ is apparently connected with some fundamental properties of the augmented dual problem. Such
connection might exist in the case when the penalty parameter increases unboundedly, but an analysis of such connection
is a challenging task, which we leave as an open problem for future research.

We start our study of the dual convergence with a simple auxiliary result that provides an important characterisation
of limit points of the sequence $\{ (\lambda_n, c_n) \}$.

\begin{lemma} \label{lem:DualLimit}
Let all assumptions of Theorem~\ref{thrm:DualValues_AugmLagr_Limit} be valid, except for assumption $(B2)$. Suppose
also that assumption $(A10)$ holds true and the sequence $\{ c_n \}$ is bounded. Then any limit point 
$(\lambda_*, c_*)$ of the sequence $\{ (\lambda_n, c_n) \}$ (if such point exists) is an $\varepsilon_*$-optimal
solution of the dual problem.
\end{lemma}

\begin{proof}
Let a subsequence $\{ (\lambda_{n_k}, c_{n_k}) \}$ converge to some 
$(\lambda_*, c_*) \in \Lambda \times (0, + \infty)$. Then, in particular, the sequence $\{ \lambda_{n_k} \}$ is bounded
and by Theorem~\ref{thrm:DualValues_AugmLagr_Limit} (see also Remark~\ref{rmrk:AugmLagrMethod_LimitAlongSubseq}) one has
\[
  \liminf_{k \to \infty} \Theta(\lambda_{n_k}, c_{n_k}) \ge \Theta_* - \varepsilon_*.
\]
Hence bearing in mind the fact that the function $\Theta$ is upper semicontinuous by assumption $(A10)$
(see Remark~\ref{rmrk:DualFunc_Concave_usc}) one can conclude that 
$\Theta(\lambda_*, c_*) \ge \Theta_* - \varepsilon_*$, that is, $(\lambda_*, c_*)$ is an $\varepsilon_*$-optimal
solution of the dual problem.
\end{proof}

\begin{corollary}
Let the assumptions of Lemma~\ref{lem:DualLimit} be valid and suppose that the functions $f$ and $\dist(G(\cdot), K)$
are lsc on $Q$, while the augmented Lagrangian $\mathscr{L}(\cdot)$ is lsc on 
$Q \times \Lambda \times (0, + \infty)$. Then any limit point $(x_*, \lambda_*, c_*)$ of the sequence 
$\{ (x_n, \lambda_n, c_n) \}$ is an $2 \varepsilon_*$-saddle point of the augmented Lagrangian, that is,
\begin{equation} \label{eq:EpsilonSaddlePoint}
  \sup_{\lambda \in \Lambda} \mathscr{L}(x_*, \lambda, c_*) - 2 \varepsilon_* \le \mathscr{L}(x_*, \lambda_*, c_*)
  \le \inf_{x \in Q} \mathscr{L}(x, \lambda_*, c_*) + \varepsilon_*.
\end{equation}
\end{corollary}

\begin{proof}
Suppose that a subsequence $\{ (x_{n_k}, \lambda_{n_k}, c_{n_k}) \}$ converges to some triplet
$(x_*, \lambda_*, c_*) \in Q \times \Lambda \times (0, + \infty)$. Then by Theorem~\ref{thrm:DualValues_AugmLagr_Limit}
(see also Remark~\ref{rmrk:AugmLagrMethod_LimitAlongSubseq}) one has $\dist(G(x_{n_k}), K) \to 0$ as $k \to \infty$
and
\[
  \limsup_{k \to \infty} f(x_{n_k}) \le \Theta_* + \varepsilon_*, \quad
  \limsup_{k \to \infty} \mathscr{L}(x_{n_k}, \lambda_{n_k}, c_{n_k}) \le \Theta_* + \varepsilon_*. 
\]
Therefore, by the lower semicontinuity assumptions one has $G(x_*) \in K$ and
\[
  f(x_*) \le \Theta_* + \varepsilon_*, \quad
  \mathscr{L}(x_*, \lambda_*, c_*) \le \Theta_* + \varepsilon_*.
\]
On the other hand, Proposition~\ref{prp:WeakDuality} and Lemma~\ref{lem:DualLimit} imply that
\[
  \Theta_* \le f(x_*), \quad 
  \Theta_* - \varepsilon_* \le \Theta(\lambda_*, c_*) \le \mathscr{L}(x_*, \lambda_*, c_*).
\]
Hence with the use of assumption $(A1)$ one gets that
\[
  \sup_{\lambda \in \Lambda} \mathscr{L}(x_*, \lambda, c_*) \le f(x_*)
  \le \mathscr{L}(x_*, \lambda_*, c_*) + 2 \varepsilon_*,
\]
that is, the first inequality in \eqref{eq:EpsilonSaddlePoint} is satisfied.

By the definition of $x_n$ one has
\[
  \mathscr{L}(x_{n_k}, \lambda_{n_k}, c_{n_k}) \le \mathscr{L}(x, \lambda_{n_k}, c_{n_k}) + \varepsilon_{n_k}
  \quad \forall k \in \mathbb{N} \: \forall x \in Q,
\]
which implies that
\[
  \liminf_{k \to \infty} \mathscr{L}(x_{n_k}, \lambda_{n_k}, c_{n_k})
  \le \limsup_{k \to \infty} \mathscr{L}(x, \lambda_{n_k}, c_{n_k}) + \varepsilon_* \quad \forall x \in Q.
\]
Hence with the use of assumption $(A10)$ and the fact that the function $\mathscr{L}(\cdot)$ is lsc one obtains
\[
  \mathscr{L}(x_*, \lambda_*, c_*) \le \mathscr{L}(x, \lambda_*, c_*) + \varepsilon_*
  \quad \forall x \in Q,
\]
that is, the second inequality in \eqref{eq:EpsilonSaddlePoint} is satisfied.
\end{proof}

\begin{remark}
{(i)~It should be noted that the augmented Lagrangian is lsc on $Q \times \Lambda \times (0, + \infty)$ for all
particular examples of the function $\Phi$ from Section~\ref{sect:Examples}, except for Rockafellar-Wets' augmented
Lagrangian, if the function $f$ is lsc on $Q$ and the function $G$ is continuous on this set. In the case of
Rockafellar-Wets' augmented Lagrangian (Example~\ref{ex:RockafellarWetsAugmLagr}) one needs to impose some additional
assumptions, such as $\sigma(\cdot) = 0.5 \| \cdot \|^2$ or $K = \{ 0 \}$ and $\sigma(\cdot) = \| \cdot \|$. Let us
also mention that in the case of inequality constrained problems, instead of assuming that $G$ is continuous, it is
sufficient to suppose that the functions $g_i$ defining the constrained $g_i(x) \le 0$ are only lower semicontinuous.
}

{(ii)~If the augmented Lagrangian $\mathscr{L}(\cdot)$ is continuous on $Q \times \Lambda \times (0, + \infty)$, then
one can replace $2 \varepsilon_*$ in the first inequality in \eqref{eq:EpsilonSaddlePoint} with $\varepsilon_*$. Indeed,
in this case by Theorem~\ref{thrm:DualValues_AugmLagr_Limit} one has $\mathscr{L}(x_*, \lambda_*, c_*) \ge \Theta_*$
and, therefore, $f(x_*) \le \mathscr{L}(x_*, \lambda_*, c_*) + \varepsilon_*$, which implies the required result.
}
\end{remark}

With the use of Lemma~\ref{lem:DualLimit} we can easily show how dual convergence is connected with existence of optimal
dual solutions/global saddle points of the augmented Lagrangian.

\begin{theorem}[dual convergence vs. existence of optimal dual solutions] \label{thrm:DualConv_vs_OptDualSolExist}
Let $\{ (x_n, \lambda_n, c_n) \}$ be the sequence generated by the model augmented Lagrangian method, and suppose that
the following conditions are satisfied:
\begin{enumerate}
\item{assumptions $(A1)$, $(A7)$, $(A10)$, $(A12)_s$--$(A14)_s$, and $(A15)$ hold true;}

\item{assumptions $(B1)$, $(B3)$, and $(B4)$ hold true;}

\item{$\varepsilon_n \to 0$ as $n \to \infty$;}

\item{for any bounded set $\Lambda_0 \subset \Lambda$ there exists $\tau > 0$ such that the function 
$\inf_{\lambda \in \Lambda_0} \Phi(G(\cdot), \lambda, \tau)$ is bounded below on $Q$.} 
\end{enumerate}
Then for the sequence of penalty parameters $\{ c_n \}$ to be bounded and the sequence of multipliers $\{ \lambda_n \}$
to have a limit point it is necessary that a globally optimal solution of the dual problem $(\mathcal{D})$ exists.
\end{theorem}

\begin{proof}
By Lemma~\ref{lem:DualLimit}, under the assumptions of the theorem any limit point of the sequence 
$\{ (\lambda_n, c_n) \}$ is a globally optimal solution of the dual problem. Therefore, for the existence of such 
limit (or, equivalently, for the sequence $\{ c_n \}$ to be bounded and the sequence $\{ \lambda_n \}$ to have a limit
point) it is necessary that a globally optimal solution of the dual problem exists.
\end{proof}

\begin{theorem}[full convergence vs. existence of global saddle points] \label{thrm:FullConv_vs_GSP_Exists} 
Let all assumptions of Theorem~\ref{thrm:DualConv_vs_OptDualSolExist} be valid and suppose that the functions $f$ and
$\dist(G(\cdot), K)$ are lsc on the set $Q$. The for the sequence of penalty parameters $\{ c_n \}$ to be bounded and
the sequence $\{ (x_n, \lambda_n) \}$ to have a limit point it is necessary that there exists a global saddle
point of the augmented Lagrangian $\mathscr{L}(\cdot)$. Moreover, for any limit point $\{ (x_*, \lambda_*, c_*) \}$ of
the sequence $\{ (x_n, \lambda_n, c_n) \}$ (if such point exists) the pair $(x_*, \lambda_*)$ is a global saddle point
of the augmented Lagrangian and $c_* \ge c_*(x_*, \lambda_*)$.
\end{theorem}

\begin{proof}
Let $(x_*, \lambda_*, c_*)$ be a limit point of the sequence $\{ (x_n, \lambda_n, c_n) \}$. Then by 
Lemma~\ref{lem:DualLimit} the pair $(\lambda_*, c_*)$ is an optimal dual solution. In turn, by applying
Theorem~\ref{thrm:DualValues_AugmLagr_Limit} (see also Remark~\ref{rmrk:AugmLagrMethod_LimitAlongSubseq}) one can
readily verify that the zero duality gap property is satisfied and $x_*$ is a globally optimal
solution of the problem $(\mathcal{P})$. Therefore, by Theorem~\ref{thrm:OptDualSol_vs_GlobalSaddlePoints} the pair
$(x_*, \lambda_*)$ is a global saddle point of $\mathscr{L}(x, \lambda, c)$ and $c_* \ge c_*(x_*, \lambda_*)$.
Consequently, for the existence of a limit point of the sequence $\{ (x_n, \lambda_n, c_n) \}$ it is necessary
that there exists a global saddle point of the augmented Lagrangian.
\end{proof}

In the case when the space $Y$ is reflexive one can prove somewhat stronger versions of the previous theorems that
uncover a connection between the \textit{boundedness} of the sequences of multipliers and penalty parameters and 
the existence of optimal dual solutions/global saddle points.

\begin{theorem}[boundedness vs. existence of optimal dual solutions] \label{thrm:Boundedness_vs_OptDualSolExist} 
Suppose that $\{ (x_n, \lambda_n, c_n) \}$ is the sequence generated by the model augmented Lagrangian method, and
let the following conditions be satisfied:
\begin{enumerate}
\item{the space $Y$ is reflexive;}

\item{assumptions $(A1)$, $(A7)$, $(A9)_s$, $(A10)$, $(A12)_s$--$(A14)_s$, and $(A15)$ hold true;}

\item{assumptions $(B1)$, $(B3)$, and $(B4)$ hold true;}

\item{the sequence $\{ \varepsilon_n \}$ is bounded and $\limsup_{n \to \infty} \varepsilon_n = \varepsilon_*$.}

\item{for any bounded set $\Lambda_0 \subset \Lambda$ there exists $\tau > 0$ such that the function 
$\inf_{\lambda \in \Lambda_0} \Phi(G(\cdot), \lambda, \tau)$ is bounded below on $Q$.} 
\end{enumerate}
Then the following statements hold true:
\begin{enumerate}
\item{any weakly limit point of the sequence $\{ (\lambda_n, c_n) \}$ (if such point exists) is an
$\varepsilon_*$-optimal solution of the problem $(\mathcal{D})$;
}

\item{if $\varepsilon_n \to 0$ as $n \to \infty$, then for the boundedness of the sequence $\{ (\lambda_n, c_n) \}$ it
is necessary that there exists a globally optimal solution of the augmented dual problem $(\mathcal{D})$;
}

\item{if, in addition, $X$ is reflexive, $Q$ is weakly sequentially closed, the functions $f$ and $\dist(G(\cdot), K)$
are weakly sequentially lsc on $Q$, and $\varepsilon_n \to 0$ as $n \to \infty$, then for the boundedness of 
the sequence $\{ (x_n, \lambda_n, c_n) \}$ it is necessary that there exists a global saddle point of the augmented
Lagrangian; furthermore, for any weakly limit point $(x_*, \lambda_*, c_*)$ of the sequence 
$\{ (x_n, \lambda_n, c_n) \}$ the pair $(x_*, \lambda_*)$ is a global saddle point of $\mathscr{L}(x, \lambda, c)$ and
$c_* \ge c_*(x_*, \lambda_*)$.
}
\end{enumerate}
\end{theorem}

\begin{proof}
The proof of this theorem almost literally repeats the proofs of Lemma~\ref{lem:DualLimit} and
Theorems~\ref{thrm:DualConv_vs_OptDualSolExist} and \ref{thrm:FullConv_vs_GSP_Exists}. One only needs to use the facts 
that (i) any bounded sequence from a reflexive Banach space has a weakly convergent subsequence, (ii) the augmented
dual function $\Theta(\lambda, c)$ is usc and concave by assumptions $(A9)_s$ and $(A10)$, and (iii) any usc
concave function defined on a Banach space is also weakly sequentially usc. 
\end{proof}

\begin{remark}
It should be underlined that the previous theorem provides necessary conditions for the boundedness of sequences
$\{ (x_n, \lambda_n, c_n) \}$ generated by the model augmented Lagrangian method \textit{irrespective} of the way in
which the sequences of multipliers $\{ \lambda_n \}$ and penalty parameters $\{ c_n \}$ are updated. As long as the
assumptions of the theorem are satisfied, the existence of a global saddle point is a necessary conditions for the
boundedness of the sequence $\{ (x_n, \lambda_n, c_n) \}$. Similarly, the existence of an optimal dual solution is 
a necessary condition for the boundedness of the sequences $\{ \lambda_n, \}$ and $\{ c_n \}$, \textit{regardless} of
the way in which they are updated.
\end{remark}

Let us finally return to Example~\ref{ex:DijointSetsOfMultipliers} in order to demonstrate how one can apply general
convergence results obtained in this seciton to better understand and predict behaviour of primal-dual augmented
Lagrangian methods.

\begin{example} \label{ex:LackOfDualSol_vs_DualConvergence}
Let $X = Y = \mathbb{R}$. Consider the following optimization problem:
\begin{equation} \label{prob:DisjointMultiplierSetsRepeat}
  \min \: f(x) = - x^2 \quad \text{subject to} \enspace g_1(x) = x - 1 \le 0, \enspace g_2(x) = - x - 1 \le 0.
\end{equation}
Let $\mathscr{L}(\cdot)$ be the Hestenes-Powell-Rockafellar augmented Lagrangian for this problem (see
\eqref{eq:HPRLagr_DisjMultSetsProblem}). As was shown in Example~\ref{ex:DijointSetsOfMultipliers}, the zero duality
gap property holds true in this case, but the augmented dual problem has no globally optimal solutions.

Let the multipliers and the penalty parameter be updated in accordance with the classic augmented Lagrangian method
(see, e.g. \cite[Algorithm~4.1]{BirginMartinez}), that is, 
\begin{equation} \label{eq:MultipUpdate_DijointMult}
  \lambda_{(n + 1), 1} = \max\big\{ \lambda_{n1} + c_n g_1(x_n), 0 \big\}, \quad
  \lambda_{(n + 1), 2} = \max\big\{ \lambda_{n2} + c_n g_2(x_n), 0 \big\}
\end{equation}
and
\begin{equation} \label{eq:PenaltyUpdate_DijointMult}
  c_{n + 1} = \begin{cases}
    c_n, & \text{if } n = 0 \text{ or } \| V_n \| \le \tau \| V_{n - 1} \|,
    \\
    \gamma c_n, & \text{otherwise},
  \end{cases}
\end{equation}
where $\tau \in (0, 1)$ and $\gamma > 1$ are fixed parameters and
\[
  V_{ni} := \min\left\{ - g_i(x_n), \frac{\mu_{ni}}{c_n} \right\}, \quad i \in \{ 1, 2 \}.
\]
One can readily check that assumptions $(B1)$, $(B3)$, and $(B4)$ hold true in this case, if $c_{\min} \ge 2$. Hence
with the use of Theorem~\ref{thrm:Boundedness_vs_OptDualSolExist} one can conclude that the sequence 
$\{ (\lambda_n, c_n) \}$ has no bounded subsequence, which means that either $c_n \to + \infty$ or 
$\| \lambda_n \| \to + \infty$ as $n \to \infty$. Let us provide a numerical example to illustrate this claim.

\begin{table} [ht!]
\caption{First 10 iterations of the model augmented Lagrangian method for problem
\eqref{prob:DisjointMultiplierSetsRepeat}.}
\begin{tabular}{| c | c | c | c | c | c | c | c | c | c | c |} 
  \hline
  $n$ & 0 & 1 & 2 & 3 & 4 & 5 & 6 & 7 & 8 & 9 \\
  \hline
  $x_n$ & 2 & -3 & 1.5 & -1.5 & 1.2 & -1.2 & 1.0909 & -1.0909 & 1.0435 & -1.0435 \\
  \hline
  $\lambda_{n1}$ & 1 & 4 & 0 & 3 & 0 & 2.4 & 0 & 2.1818 & 0 & 2.087 \\
  \hline
  $\lambda_{n2}$ & 1 & 0 & 6 & 0 & 3 & 0 & 2.4 & 0 & 2.1818 & 0 \\
  \hline
  $c_n$ & 3 & 3 & 6 & 6 & 12 & 12 & 24 & 24 & 48 & 48 \\
  \hline
\end{tabular}
\label{tab:AugmLagrMethod_DijointMultipliers}
\end{table}

Let $\tau = 0.9$, $\gamma = 2$, $c_0 = 3$, and $\lambda_0 = (1, 1)$. First 10 iterations of the model augmented 
Lagrangian method with multiplier updates \eqref{eq:MultipUpdate_DijointMult} and penalty parameter updates
\eqref{eq:PenaltyUpdate_DijointMult} are given in Table~\ref{tab:AugmLagrMethod_DijointMultipliers}. Let us note that
the points of global minimum of the augmented Lagrangian were computed analytically. The results of computer simulation
have shown that $c_n \to \infty$, $\lambda_{2n} \to (2, 0)$, $\lambda_{2n + 1} \to (0, 2)$, $x_{2n} \to 1$, and 
$x_{2n + 1} \to -1$ as $n \to \infty$ (this fact can be proved analytically, but its proof is rather cumbersome, which
is why we do not present it here). Thus, the iterations of the method oscillate between gradually shrinking
neighbourhoods of two globally optimal solutions of problem \eqref{prob:DisjointMultiplierSetsRepeat} and the KKT
multipliers corresponding to these solutions, while the penalty parameter increases unboundedly.
\end{example}

\begin{remark}
The convergence theory for the model augmented Lagrangian me\-thod presented in this paper generalizes and unifies many
existing results on convergence of augmented Lagrangian methods. In particular, many such results either directly
follow from Theorems~\ref{thrm:DualValues_AugmLagr_Limit}, \ref{thrm:GlobalConv_vs_DualityGap}, and
\ref{thrm:FullConv_vs_GSP_Exists} and Lemma~\ref{lem:DualLimit} or can be easily deduced from them, including
\cite[Proposition~2.1]{Bertsekas}, \cite[Theorem~6.1, part (1)]{Polyak2001}, \cite[Theorems~1--3 and 7]{LuoSunLi},
\cite[Theorems~2.4 and 3.1]{LuoSunWu}, \cite[Theorems~4 and 7]{LuoWuChen2012}, \cite[Theorem~5.2]{BirginMartinez},
\cite[Theorem~5, part (iii)]{CordovaOliveiraSagastizabal}, etc.
\end{remark}

\bibliographystyle{abbrv}  
\bibliography{AugmLagrMethods_bibl}

\end{document}